\documentclass[11pt, reqno, final]{amsart}
\usepackage[usenames, dvipsnames]{xcolor}

\usepackage{setspace}
\usepackage[english]{babel}
\usepackage[latin1]{inputenc}
\usepackage[mathscr]{euscript}
\usepackage{booktabs}
\usepackage{enumitem}
\usepackage[all]{xy}

\usepackage[refpage,noprefix]{nomencl}
\usepackage[pdfborder={0 0 0}, colorlinks = true, linkcolor=blue, citecolor=OliveGreen]{hyperref}
\usepackage{url}
\usepackage{geometry}
\usepackage{amsmath,amsthm,amssymb, bm}

\newcommand{\R}{\mathbb{R}} 
\newcommand{\Q}{\mathbb{Q}} 
\newcommand{\uhp}{\mathbb{H}} 
\newcommand{\B}{\mathbb{B}} 
\newcommand{\C}{\mathbb{C}} 
\newcommand{\Z}{\mathbb{Z}} 
\newcommand{\F}{\mathbb{F}} 
\newcommand{\N}{\textup{N}} 

\newcommand{\Zhat}{\hat{\mathbb{Z}}}
\newcommand{\adeles}{\mathbb{A}}

\newcommand*{\domain}{\mathbb{D}}

\newcommand*{\reg}{\mathrm{reg}}

\newcommand{\OD}{\calO_{\kb}}

\newcommand{\ODhatt}{\hat{\calO}_{\kb}^{\times}}

\newcommand{\sage}{\texttt{sage}}

\newcommand{\bs}{\backslash}

\newcommand{\calC}{\mathcal{C}}

\newcommand{\calE}{\mathcal{E}}
\newcommand{\calF}{\mathcal{F}}

\newcommand{\calH}{\mathcal{H}}

\newcommand{\calN}{\mathcal{N}}
\newcommand{\calO}{\mathcal{O}}
\newcommand{\calP}{\mathcal{P}}

\newcommand{\calX}{\mathcal{X}}
\newcommand{\calY}{\mathcal{Y}}
\newcommand{\calZ}{\mathcal{Z}}

\renewcommand{\Re}{\mathrm{Re}}
\renewcommand{\Im}{\mathrm{Im}}

\newcommand{\FJ}{\mathscr{F}}
\newcommand{\GJ}{\mathscr{G}}

\newcommand{\different}[1]{\partial_#1} 
\newcommand{\diffinv}[1]{\partial_#1^{-1}} 

\newcommand{\frakA}{\mathfrak{A}}
\newcommand{\fraka}{\mathfrak{a}}
\newcommand{\frakb}{\mathfrak{b}}
\newcommand{\frakc}{\mathfrak{c}}

\newcommand{\frake}{\mathfrak{e}}

\newcommand{\frakf}{\mathfrak{f}}

\newcommand{\frakn}{\mathfrak{n}}
\newcommand{\frakp}{\mathfrak{p}}
\newcommand{\frakP}{\mathfrak{P}}

\newcommand{\abcd}{\begin{pmatrix}a & b \\ c & d\end{pmatrix}}

\newcommand{\smallabcd}{\left(\begin{smallmatrix}a & b \\ c & d\end{smallmatrix}\right)}

\newcommand{\Tmatrix}{\begin{pmatrix}1 & 1 \\ 0 & 1\end{pmatrix}}
\newcommand{\Smatrix}{\begin{pmatrix}0 & -1 \\ 1 & 0\end{pmatrix}}
\newcommand{\smallTmatrix}{\left(\begin{smallmatrix}1 & 1 \\ 0 & 1\end{smallmatrix}\right)}
\newcommand{\smallSmatrix}{\left(\begin{smallmatrix}0 & -1 \\ 1 & 0\end{smallmatrix}\right)}

\newcommand{\twomat}[4]{\begin{pmatrix}
                        #1 & #2 \\ #3 & #4 
                       \end{pmatrix}}

\newcommand{\rprod}{\sideset{}{'}\prod}

\newcommand{\inverse}{^{-1}}

\newcommand{\abs}[1]{\left\vert#1\right\vert}

\DeclareMathOperator{\End}{End}

\DeclareMathOperator{\Aut}{Aut}

\DeclareMathOperator{\SL}{SL}
\DeclareMathOperator{\GL}{GL}

\DeclareMathOperator{\Mp}{Mp}
\DeclareMathOperator{\Cl}{Cl}
\newcommand{\kb}{{\mathsf{k}}}
\newcommand{\Hkb}{\mathsf{H}_{\mathsf{k}}}
\newcommand{\Clk}{\Cl_{\kb}}

\newcommand{\Clks}{\Cl_{\kb}^{2}}

\DeclareMathOperator{\tr}{tr}

\DeclareMathOperator{\Gal}{Gal}
\DeclareMathOperator{\divisor}{div}

\DeclareMathOperator{\SO}{SO}

\DeclareMathOperator{\GSpin}{GSpin}

\DeclareMathOperator{\diag}{diag}

\DeclareMathOperator{\ord}{ord}

\DeclareMathOperator{\vol}{vol}

\DeclareMathOperator{\supp}{supp}
\DeclareMathOperator{\sgn}{sgn}

\DeclareMathOperator*{\CT}{CT}

\DeclareMathOperator{\res}{res}

\DeclareMathOperator{\del}{\partial}
\DeclareMathOperator{\delbar}{\overline{\partial}}

\newcommand{\CPK}{C_{P,K}}

\newtheorem{theorem}{Theorem}[section]
\newtheorem{proposition}[theorem]{Proposition}
\newtheorem{lemma}[theorem]{Lemma}
\newtheorem{corollary}[theorem]{Corollary}

\theoremstyle{definition}
\newtheorem{definition}[theorem]{Definition}

\newtheorem{remark}[theorem]{Remark}
\newtheorem{assumption}[theorem]{Assumption}
\newtheorem{conjecture}[theorem]{Conjecture}
\newtheorem*{warning}{Warning}

\newcommand{\QD}[1][]{Q_\Delta\ifthenelse{\equal{#1}{}}{}{\left(#1\right)}}

\newcommand{\G}{\mathbb{G}}

\newcommand{\pre}[1]{{\widetilde{#1}}}
\newcommand{\prep}[1]{\widetilde{#1}^{+}}

\newcommand{\kronecker}[2]{\left(\frac{#1}{#2}\right)}

\newcommand{\zmatrix}{\begin{pmatrix}  z & -z^{2} \\ 1 & -z \end{pmatrix}}

\DeclareMathOperator{\Diff}{Diff}

\DeclareMathOperator{\Spec}{Spec}
\DeclareMathOperator{\CH}{CH}

\DeclareMathOperator{\Pic}{Pic}

\newcommand{\deghat}{\widehat{\deg}}

\usepackage[style=alphabetic, backend=bibtex8, doi=false, url=true, firstinits=true, eprint=false, maxnames=6, sorting=nyt]{biblatex}
\numberwithin{equation}{section}

\makenomenclature
\makeindex

\begin{document}
\title[CM values of regularized theta lifts]{CM values of regularized theta lifts and harmonic weak Maa\ss{} forms of weight one}
 
\author{Stephan Ehlen}
\email{stephan.ehlen@math.uni-koeln.de}
\address{University of Cologne, Weyertal 86-90, D-50931 Cologne, Germany}

\begin{abstract}
  We study special values of regularized theta lifts at
  complex multiplication (CM) points. In particular, we show that CM
  values of Borcherds products can be expressed in terms of finitely
  many Fourier coefficients of certain harmonic weak Maa\ss{} forms of
  weight one. As it turns out, these coefficients are logarithms of
  algebraic integers whose prime ideal factorization is determined by
  special cycles on an arithmetic curve. Our results imply a
  conjecture of Duke and Li \cite{DukeLiMock} and give a new proof 
  of the modularity of a certain arithmetic generating series of weight one 
  studied by Kudla, Rapoport and Yang \cite{kry-tiny}.
\end{abstract}

\thanks{This work was partly supported by DFG grant BR-2163/2-1 and a CRM-ISM postdoctoral fellowship from the CRM and McGill University in Montreal, Quebec, Canada.}

\maketitle

\setcounter{tocdepth}{1}
\tableofcontents

\nomenclature[C]{$\C$}{The field of complex numbers\nomnorefpage}
\nomenclature[Fq]{$\F_{q}$}{A finite field with $q$ elements\nomnorefpage}
\nomenclature[Fqb]{$\overline{\F}_{q}$}{An algebraic closure of a finite field with $q$  elements\nomnorefpage}
\nomenclature[Z]{$\Z$}{The ring of integers\nomnorefpage}
\nomenclature[Z0]{$\Z_{>0}$}{The set of positive integers\nomnorefpage}
\nomenclature[Zp]{$\Z_p$}{The $p$-adic integers\nomnorefpage}
\nomenclature[Qp]{$\Q_p$}{The field of $p$-adic numbers\nomnorefpage}
\nomenclature[Q]{$\Q$}{The field of rational numbers\nomnorefpage}
\nomenclature[e]{$e(x)$}{$=e^{2\pi i x}$\nomnorefpage}
\nomenclature[q]{$q$}{Usually $q=e^{2\pi i \tau}$\nomnorefpage}
\nomenclature[t]{$\tau$}{$\tau = u+iv \in \uhp$\nomnorefpage}
\nomenclature[k]{$k$}{Usually a half-integer or integer (a weight), often $k=1-n/2$\nomnorefpage}
\nomenclature[kb]{$\kb$}{A field, usually an imaginary quadratic field\nomnorefpage}
\nomenclature[kD]{$\kb_D$}{$=\Q(\sqrt{D})$\nomnorefpage}
\nomenclature[()]{$(a,b)$}{$=\gcd(a,b)$, the greatest common divisor of $a$ and $b$\nomnorefpage}
\nomenclature[Im]{$\Im(z)$}{The imaginary part of $z$\nomnorefpage}
\nomenclature[Re]{$\Re(z)$}{The real part of $z$\nomnorefpage}
\nomenclature[H]{$\uhp$}{The complex upper half-plane, $\uhp = \{ z \in \C\ \mid\ \Im(z)>0\}$\nomnorefpage}
\nomenclature[abc]{$\fraka, \frakb, \frakc$}{Fractional (most of the time integral) ideals\nomnorefpage}

\begingroup
\section{Introduction}
Special values of automorphic forms often encode interesting arithmetic information.
The type of automorphic forms we consider in this article
are obtained as regularized theta lifts of harmonic weak Maa\ss{} forms.
We will study their values at CM points on orthogonal Shimura varieties.
This includes the CM values of Borcherds products and as a special case (generalizations of) 
the singular moduli considered by Gross and Zagier in their seminal paper \cite{grosszagier-singularmoduli}. 
The average value over a CM cycle was previously studied by Bruinier-Yang \cite{bryfaltings} and Schofer \cite{schofer}
using a seesaw dual pair and the Siegel-Weil formula.
We use the same seesaw identity without applying the Siegel-Weil formula
to relate the values at individual points to coefficients of harmonic weak Maa\ss{} forms of weight one.
We then show that the coefficients of an appropriate normalization of these weak Maa\ss{} forms are logarithms of algebraic numbers whose
prime factorizations are determined by certain special cycles on the stack of elliptic curves with complex multiplication.
We use this relation to prove a conjecture by Duke and Li \cite{DukeLiMock}.

\subsection{Setup and notation}
\label{sec:setup-notation}
Let $L$ be an even lattice with quadratic form $Q$ of type $(2,n)$
and write $V = L \otimes_\Z \Q$ for the corresponding rational quadratic space.
We write $(\cdot,\cdot)$ for the associated bilinear form, such that $(x,x)=2Q(x)$.
Associated with such a lattice is a finite quadratic module given by the discriminant group $A_{L} = L'/L$,
where $L'$ is the dual of $L$ with respect to $(\cdot,\cdot)$, and the reduction of $Q$ modulo $\Z$.
For simplicity, we assume in the introduction that $n$ is even.
The Weil representation \cite{Weil, shintani-weil, boautgra} $\rho_{L}$ of $\SL_2(\Z)$
associated with $L$ acts on the group ring $S_{L} = \C[A_{L}]$.
We write $\phi_{\mu}$ with $\mu \in A_L$ for the standard basis elements of $S_L$.
Moreover, we write $\langle \cdot, \cdot \rangle$ for the $\C$-bilinear pairing such that $\langle \phi_{\mu},\phi_{\mu} \rangle = 1$
and $\langle \phi_{\mu},\phi_{\lambda} \rangle = 0$ for $\lambda \neq \mu$.

We denote the space of $S_L$-valued \emph{weakly holomorphic modular forms}
of weight $k$ transforming with representation $\rho_{L}$ by $M_{k,L}^{!}$
and the spaces of holomorphic modular forms and cusp forms by $M_{k,L}$ and $S_{k,L}$, respectively.
We denote the hermitian symmetric domain associated with the orthogonal group $\SO_{V}(\R)$
by $\domain$ and realize $\domain$ as the Grassmannian of oriented $2$-dimensional positive definite subspaces of $V(\R) = V \otimes_{\Q} \R$.

A function $f: \uhp \rightarrow S_{L}$ is called a \emph{harmonic weak Maa\ss{} form} of weight $k$
and representation $\rho_{L}$ if it transforms like a modular form of weight $k$ with respect to $\rho_L$,
is harmonic with respect to the weight $k$ Laplace operator and grows at most exponentially at
the cusp $\infty$. We write $\calH_{k,L}$ for the space of such functions.
An element $f \in \calH_{k,L}$ admits a unique decomposition $f = f^{+}+f^{-}$
into a \emph{holomorphic part} $f^{+}$ and a \emph{non-holomorphic part} $f^{-}$.
We refer to Section \ref{chap:prelims} for more details.

The antilinear differential operator $\xi_{k}$ defined by
\begin{equation}
  \xi_k(f)(\tau) := 2i v^{k} \overline{\frac{\partial}{\partial\overline{\tau}} f(\tau)}
\end{equation}
plays an important role in the theory of harmonic weak Maa\ss{} forms.
It was shown by Bruinier and Funke \cite{brfugeom} that
$\xi_k:\mathcal{H}_{k,L} \longrightarrow M^{!}_{2-k,L^{-}}$
is surjective \cite[Theorem 3.7]{brfugeom} with kernel $M_{k,L}^{!}$.
Here, the lattice $L^{-}$ is given by the lattice $L$ together with the quadratic form $-Q$.
The associated Weil representation can be identified with the dual of
$\rho_{L}$. We denote by $H_{k,L} \subset \calH_{k,L}$ the subspace of those harmonic weak Maa\ss{} forms
that map to a cusp form under $\xi_k$.

Let $H = \GSpin_{V}$ be the general spin group, 
which is a central extension of the special orthogonal group $\SO_{V}$ and
let $K \subset H(\adeles_{f})$ be a compact open subgroup such that $K$ stabilizes $L$ and acts trivially on $A_{L}$; here, $\adeles_{f}$ are finite adeles of $\Q$.
There is a \emph{Shimura variety} $X_K$ with complex points
\[
  X_{K}(\C) = H(\Q) \bs (\domain \times H(\adeles_{f})/K).
\]
For an appropriate choice of $K$, the Siegel theta function $\Theta_{L}(\tau,z,h)$ 
attached to $L$ defines a non-holomorphic function on $X_K$ (here, $\tau \in \uhp$, $z \in \domain$, and $h \in H(\adeles_f)$).
In $\tau \in \uhp$, it transforms like a modular form of weight $(2-n)/2$ with representation $\rho_L$, but is also non-holomorphic.

The \emph{regularized theta lift} of $f \in H_{k,L}$ with $k = (2-n)/2$ is defined as
\begin{equation}
  \label{eq:intro-int1}
  \Phi_{L}(z,h,f) = \int_{\SL_2(\Z) \bs \uhp}^{\reg} \langle f(\tau), \overline{\Theta_{L}(\tau,z,h)} \rangle v^{k} \frac{dudv}{v^{2}};
\end{equation}
where $\tau = u+iv \in \uhp$ and the integral is regularized using Borcherds' regularization (following Harvey-Moore) \cite{boautgra}.
We obtain a $\Gamma_{L}$-invariant function $\Phi_{L}(z,h,f)$ which is real analytic outside a divisor $Z(f)$ 
given by codimension one sub-Grassmannians of $\domain$ depending only on the Fourier coefficients of $f^+$ with negative index, its \emph{principal part}.

\subsection{CM values}
\label{sec:cm-values}
We will now describe the CM points we are considering. 
Let $U \subset V(\Q)$ be a rational 2-dimensional positive definite subspace
of $V$. Then $U$ defines two rational points $z_{U}^{\pm}$ in $\domain$,
given by $U(\R)$ together with the two possible orientations.
Let $T = \GSpin_{U}$ and $K_{T} = K \cap T(\adeles_f)$.
We obtain a CM cycle in $X_K$ by considering the Shimura variety with complex points
\[
 Z(U)(\C) = T(\Q) \bs (\{z_{U}^{\pm}\} \times T(\adeles_f) / K_{T}) \hookrightarrow X_K(\C).
\]
For simplicity, we assume in the introduction that $L = P \oplus N$, where
$P = L \cap U$ is positive definite and $2$-dimensional and 
$N = L \cap U^{\perp}$ is a negative definite $n$-dimensional lattice in $V$.
Under this assumption, we have $S_L \cong S_P \otimes_\C S_N$ and 
$\Theta_{L}(\tau,z_{U}^{\pm},h) = \Theta_{P}(\tau,h) \otimes_\C \Theta_{N}(\tau)$ for $h \in T(\adeles_f)$.

Our first result is an analytic formula for the CM value $\Phi_L(z,h,f)$ for any $(z,h) \in Z(U)$.
\begin{theorem}
  \label{thm:cmval-intro}
  Let $f \in H_{1-n/2,L}$ and let $\pre{\Theta}_P(\tau,h) \in \calH_{1,P^{-}}$
  be a harmonic weak Maa\ss{} form of weight $1$
  with the property that $\xi_1(\pre{\Theta}_P(\tau,h)) = \Theta_{P}(\tau,h)$.
  Then for any $(z,h) \in Z(U)$ the value of $\Phi_L(z,h,f)$ is given by
  \begin{align} \label{eq:CMintro}
    \Phi_L(z,h,f) &= \CT \left( \langle f^+(\tau), \Theta_{N^{-}}(\tau) \otimes \prep{\Theta}_{P}(\tau,h) \rangle \right) \\
    &\quad - \int_{\SL_2(\Z) \bs \uhp}^{\reg} \langle \overline{\xi_{1-n/2}(f)}(\tau), \Theta_{N^{-}}(\tau) \otimes \pre{\Theta}_{P}(\tau,h) \rangle\, v^{1+n/2}\,\frac{dudv}{v^{2}}. \notag
  \end{align}
\end{theorem}
Here, we write $\CT( \cdot )$ for the constant term in the Fourier expansion. It is a finite sum of products of coefficients of
$f^{+}$ and $\Theta_{N^{-}}(\tau) \otimes \prep{\Theta}_{P}(\tau,h)$.
Note that if $f$ is weakly holomorphic, then the regularized integral above vanishes, as $\xi_{1-n/2}(f) = 0$ in that case.
Moreover, in this case the left-hand side of \eqref{eq:CMintro} is the logarithm of a \emph{Borcherds product} \cite{boautgra},
a meromorphic modular form for $\Gamma_L$ with divisor $Z(f)$.
In particular, we obtain a formula for CM values of Borcherds products which involves only a finite number
of coefficients of $\prep{\Theta}_{P}(\tau,h)$ weighted by representation numbers of the lattice $N$
and the coefficients of $f^{+}$. Note that $\pre{\Theta}_{P}(\tau,h)$
is not uniquely determined. It can be modified by adding \emph{any} weakly holomorphic modular form
and the theorem is still valid.

If we consider the weighted average value of $\Phi_L(z,h,f)$ over the CM-cycle $Z(U)$,
we obtain a slight generalization of the results of Schofer \cite{schofer} (who considered weakly holomorphic $f$) and Bruinier-Yang \cite{bryfaltings}.
Their formula involves the function $\calE_P(\tau)$, a special value of the derivative of an \emph{incoherent Eisenstein series}.
The coefficients of $\calE_P(\tau)$ have been calculated in many cases, 
also by Schofer \cite{schofer}, Bruinier and Yang \cite{bryfaltings} and Kudla and Yang \cite{KudlaYangEisenstein}. 
The Eisenstein series $\calE_P(\tau)$ is a harmonic weak Maa\ss{} form and has the property that 
$\xi_{1}(\calE_P(\tau)) = E_P(\tau)$ is the genus Eisenstein series attached to $P$.
Using these explicit formulas, 
it was shown in \cite{kry-tiny} and \cite{bryfaltings}, that $\calE_P(\tau)$ is in fact the generating
series of the arithmetic degrees of certain special cycles on an arithmetic moduli stack.

\subsection{Harmonic weak Maa\ss{} forms and arithmetic geometry}
\label{sec:arithm-geom-harm}
Motivated by these results and to make Theorem \ref{thm:cmval-intro} more explicit, we study the coefficients of the holomorphic parts
of appropriate choices for $\pre{\Theta}_{P}(\tau,h)$ and their arithmetic meaning.
We will show that they are also related to the special cycles
studied by Kudla, Rapoport and Yang \cite{kryderivfaltings} and give a new proof 
of the modularity of the degree generating series.
We let $\kb$ be an imaginary quadratic number field of discriminant $D < 0$. 
We assume in the introduction that $D=-l$ for a prime $l \equiv 3 \bmod{4}$ with $l>3$
(so that the class number of $\kb$ is odd and there is only one genus). We will
remove this restriction in the body of the paper and work with odd fundamental discriminants.
We write $\Clk$ for the class group of $\kb$ and $h_\kb$ denotes the class number of $\kb$.
Let $P = \calO_\kb$ be the ring of integers in $\kb$, which is a lattice of type $(2,0)$
together with the quadratic form $Q(x) = \N(x) = xx'$, where $x \mapsto x'$ denotes
the non-trivial Galois automorphism of $\kb$.
The dual lattice $P' = \different{\kb}^{-1}$ is given by the inverse different in $\kb$
and the discriminant group $P'/P$ is cyclic of order $\abs{D}$.
If we let $K = \hat\calO_\kb^\times = \calO_\kb^\times \otimes_Z \hat \Z$
and denote $H = \GSpin_V$ for $V = \kb$, then $H(\adeles_f)$
is given by the finite ideles $\adeles^\times_{\kb,f}$ over $\kb$,
$\domain = \{ z^\pm \}$ consists of two points (given by $\kb \otimes_\Q \R$ with the two possible choices of orientation), and
$X_K(\C)$ is given by two copies of the class group (see also Lemma \ref{lem:ZUCl})
\[
  X_K(\C) = H(\Q) \bs (\{ z^\pm \} \times H(\adeles_f) / K) \cong \{ z^\pm \} \times \Clk.
\]
Moreover, if $\fraka$ is the ideal determined by the idele $h$, we can identify the theta function $\Theta_P(\tau,h) = \Theta_P(\tau,z^\pm, h)$ with
\[
  \Theta_{\fraka^2}(\tau) = \sum_{\mu \in \different{\kb}^{-1}\fraka^2/\fraka^2}\sum_{x \in \fraka^2 + \mu} e\left(\frac{\N(x)}{\N(\fraka^2)}\tau\right) \phi_\mu,
\]
which only depends on the class $[\fraka]^2 \in \Clk$ of $\fraka^2$.
Note that since the class number $\Clk$ is odd, every class $[\frakb] \in \Clk$
is of the form $[\frakb] = [\fraka]^2$ for some $\fraka$ (i.e. there is only one genus).

Let $C_{D}$ be the (Deligne-Mumford) moduli stack of elliptic curves with complex multiplication
by $\calO_\kb$. The coarse moduli scheme of $C_{D}$
is isomorphic to $\Spec \calO_{\Hkb}$, where $\Hkb$ is the Hilbert class field of $\kb$.
Kudla, Rapoport and Yang define cycles $\calZ(m)$ on this arithmetic curve
given by elliptic curves with certain ``special endomorphisms''. These endomorphisms
only occur in positive characteristic and the cycles $\calZ(m)$ are always supported
in the fiber above a unique prime $p$, which is non-split in $\kb$.
They showed that the degree generating series
\begin{equation}
  \label{eq:eisholdeg}
  \sum_{m \in \frac{1}{\abs{D}}\Z_{>0}} \widehat\deg\, \calZ(m) e(m\tau) (\phi_{m}+\phi_{-m}) + 2 \Lambda'( \chi_{D}, 0) \phi_0,
\end{equation}
is the holomorphic part of $-h_\kb \calE_P(\tau)$.
Here, we simply wrote $\phi_m$ for $\phi_\mu$ if $m \in \Z$ is an integer with 
${Q(\mu) + m \in \Z}$ and $\Lambda'(\chi_{D}, s)$ denotes the derivative of the completed Dirichlet $L$-function $\Lambda(\chi_{D},s) = \abs{D}^{\frac{s}{2}} \pi^{-\frac{s+1}{2}} \Gamma\left(\frac{s+1}{2}\right) L(\chi_D,s)$ for the character $\chi_{D} = \kronecker{D}{\cdot}$. 
If $m$ is not represented by $-Q$ modulo $\Z$, then the corresponding coefficient vanishes.
The modularity of the generating series \eqref{eq:eisholdeg} is proven in \cite{kry-tiny} by
explicitly computing the Fourier coefficients of $\calE_P(\tau)$ and the arithmetic degrees separately.

As no ``explicit'' construction of the forms $\pre{\Theta}_{P}(\tau,h)$ is known, we have to resort to a different method
to show a relation to the cycles $\calZ(m)$.
To state this second result in more detail, we write the pushforward (which has to be appropriately normalized) 
of the cycle $\calZ(m)$ to $\Spec \calO_{\Hkb}$ as an Arakelov divisor with vanishing archimedean contribution as
\[
  \calZ(m) = \sum_{\frakP \subset \calO_{\Hkb}} \calZ(m)_{\frakP} \frakP,
\]
where the sum runs over all nonzero prime ideals of $\calO_{\Hkb}$.
The arithmetic degree above is obtained as
\[
  \widehat\deg\, \calZ(m) = \sum_{\frakP \subset \calO_{\Hkb}} \calZ(m)_{\frakP} \log\N_{\Hkb/\Q}(\frakP).
\]

We obtain the following result which follows from Theorem \ref{thm:pre-pvals} in Section \ref{sec:arith-pullback}.
\begin{theorem}
  \label{thm:intro-pre-pvals}
  For every $[\fraka] \in \Clk$ there is a harmonic weak Maa\ss{} form
  $\pre{\Theta}_{\fraka^2}(\tau) \in \calH_{1,P^{-}}$ with holomorphic part
  \[
  \prep{\Theta}_{\fraka^2}(\tau) = \sum_{\substack{m \gg -\infty \\ m \in \frac{1}{\abs{D}}\Z}} c^{+}(\fraka^2,m) e(m \tau)  (\phi_{m} + \phi_{-m})
  \]
  satisfying the following properties.
  \begin{enumerate}
  \item We have $\xi_1(\pre{\Theta}_{\fraka^2}(\tau)) = \Theta_\fraka(\tau)$ for every $\fraka$ and the harmonic weak Maa\ss{} form
    \begin{equation}
      \label{eq:intro-eis-pre}
      \pre{E}_P(\tau) := \frac{1}{h_\kb}\sum_{[\fraka] \in \Clk} \pre{\Theta}_{\fraka^{2}}(\tau)
    \end{equation}
    satisfies $\xi_1(\pre{E}_{P}(\tau)) = E_P(\tau)$, where $E_P(\tau)$ is the Eisenstein series attached to $P$ and the Fourier coefficients of negative index of $\pre{E}_{P}^+(\tau)$ vanish.
  \item For all $m \in \frac{1}{\abs{D}}\Z$ with $\chi_{D}(\abs{D}m)=1$, we have $c^{+}(\fraka^2,m) = 0$.
  \item If $\chi_{D}(\abs{D}m) \neq 1$, then
    \[
      c^{+}(\fraka,m) = - \frac{2}{r} \log\abs{\alpha(\fraka^2, m)},
    \]
   where $\alpha(\fraka,m) \in \calO_{\Hkb}$ and $r \in \Z_{>0}$ depends only on $D$.
  \item 
    Furthermore, we have for $m > 0$ that
    \begin{equation*}
      \label{eq:shimrepintro}
      \ord_{\mathfrak{P}}(\alpha(\fraka^2,m)) = 2 r \calZ(m)_{\mathfrak{P}^{\sigma}}
    \end{equation*}
    for all prime ideals $\frakP \subset \calO_{\Hkb}$, where $\sigma = \sigma(\fraka^{-1}) \in \Gal(\Hkb/\kb)$
    corresponds to the image of the ideal class of $\fraka^{-1}$ under the Artin map of class field theory.
  \item If $m < 0$ with $\chi_{D}(-Dm) \neq 1$, then $\alpha(\fraka,m) \in \calO_{\Hkb}^{\times}$.
  \end{enumerate}
\end{theorem}
Note that the theorem implies that we can write
$c^{+}(\fraka,m) = \log(\beta(\fraka,m))$, where $\beta(\fraka,m)$ is contained in a finite
field extension of $\Hkb$.

The idea of the proof of Theorem \ref{thm:intro-pre-pvals} (or Theorem \ref{thm:pre-pvals} in the body of the paper) 
is to use certain seesaw dual reductive pairs \cite{Kudla-seesaw}.
The following simple identity is central for our argument:
\begin{equation}
  \label{eq:seesawintro1}
  \langle f(\tau) \otimes_\C g(\tau), \overline{\Theta_P(\tau, h) \otimes_\C \Theta_N(\tau)} \rangle = \Phi_P(z_U^\pm,h,f) \cdot \langle g(\tau), \overline{\Theta_N(\tau)} \rangle,
\end{equation}
for $f$ valued in $S_P$ and $g$ valued in $S_N$ and where $h \in T(\adeles_f)$.
We use this together with Theorem \ref{thm:cmval-intro} to relate the coefficients of $\prep{\Theta}_{P}(\tau,h)$ (occurring via Theorem \ref{thm:cmval-intro} in the first factor on the right-hand side of \eqref{eq:seesawintro1}) to Borcherds products of weight zero on modular curves (the left-hand side of \eqref{eq:seesawintro1} in a special case).
These Borcherds products are modular functions with special divisors. The second factor on the right-hand side can be shown to be constant for a special choice of $g$
giving the desired relation.
We then apply the main result of \cite{ehlen-intersection} which relates the CM values of Borcherds products to the special cycles.

Duke and Li \cite{DukeLiMock} independently obtained related results
on the coefficients of such harmonic weak Maa\ss{} forms 
in the case of a prime discriminant using different methods.
Based on numerical evidence, they formulated a conjecture on the prime factorization of
the algebraic numbers $\alpha(\fraka,m)$.
To state the conjectured formula, which we prove in this paper, we need a bit more notation.
For $m \in \Q_{>0}$, we define a set of rational primes by
\begin{equation*}
  \label{eq:Diff-intro}
  \Diff(m) = \{ p < \infty\ \mid\ (-m, D)_p = -1 \},
\end{equation*}
where $(\cdot, \cdot)_p$ denotes the $p$-adic Hilbert symbol \cite{serrearith}.
Moreover, let
\begin{equation*}
  \nu_{p}(m) =
  \begin{cases}
    \frac{1}{2}(\ord_{p}(m)+1), & \text{if $p$ is inert in $\kb$},\\
    \ord_{p}(m \abs{D}), & \text{if $p$ is ramified in $\kb$},
  \end{cases}
\end{equation*}
and, finally $o(m) = 1$ if $\ord_{l}(m\abs{D})>0$ and $o(m) = 0$, otherwise.
The following formula has been conjectured in \cite[p. 46]{DukeLiMock} (stated slightly differently). We fix the embedding of $\kb$ into $\C$, such that $\Im(\sqrt{D})>0$ and identify $\Hkb = \kb(j)$, where $j$ is equal to the $j$-invariant $j\left( \tfrac{1+\sqrt{D}}{2} \right) = j(\calO_\kb)$.
\begin{theorem}\label{thm:dukeliconj}
Assume that $l>3$ and let $m \in \frac{1}{\abs{D}}\Z_{>0}$.
Then $\alpha(\fraka, m) \in \calO_{\Hkb}^\times$ unless $\abs{\Diff(m)} = 1$.
Thus, assume that $\Diff(m) = \{p\}$ and let $\frakP_{0} \mid p$ be the unique prime above $p$
fixed by complex conjugation. If $\frakP = \frakP_{0}^{\sigma(\frakb)}$ for any fractional ideal $\frakb$ of $\kb$, then
\[  
   \ord_{\frakP}(\alpha(\fraka^2, m)) = 2^{o(m)} \cdot r \cdot \sum_{n \geq 1} \rho\left( \frac{m\abs{D}}{p^{n}}, [\fraka]^2[\frakb]^{-2}\right) 
                                                              = 2^{o(m)} r \cdot \nu_{p}(m) \rho(m\abs{D}/p, [\fraka]^{2}[\frakb]^{-2}),
\]
where $\rho(m, \calC)$ is the number of integral ideals of $\calO_\kb$ of norm $m$ in the class $\calC \in \Clk$.
\end{theorem}
Theorem \ref{thm:dukeliconj} follows from the new relation \emph{(iv)} of the coefficients to the special cycles, 
together with Proposition \ref{prop:Zprime-intro} below.
We note that Duke and Li also conjectured that $r \mid 24 h_\kb h_{\Hkb}$, where $h_{\Hkb}$ is the class number of $\Hkb$, which does not
follow from our proof (we do get a bound on $r$ but it might be larger than what they conjecture).
See Section \ref{sec:conjecture-duke-li} for details on the scalar-valued case as in \cite{DukeLiMock} and Corollary \ref{cor:scalar-pval} for a generalization to composite fundamental discriminants.

As a corollary of our results (see Theorem \ref{thm:arithgen}), we are also able to show that \eqref{eq:eisholdeg} 
is (up to a constant) the holomorphic part of a harmonic weak Maa\ss{} form $\pre{E}_{P}(\tau)$ 
without using any explicit formulas and therefore give a completely new proof of the modularity of the (degree) generating series of these cycles.
This follows essentially from \emph{(i)} together with \emph{(iv)} in Theorem \ref{thm:intro-pre-pvals} because taking the sum as in \emph{(i)}
in fact computes the degree of the cycles by \emph{(iv)}. See also Section \ref{sec:consequences}.
In Section \ref{sec:br-mod}, we also consider a holomorphic analogue of the generating series \eqref{eq:eisholdeg},
which can be seen as the generating series of the cycles $\calZ(m)$ equipped with an automorphic Green function,
and show that is modular, as well.

In \cite{ehlen-intersection}, we gave formulas for the multiplicities of
the cycles $\calZ(m)$ which provide a prime ideal decomposition of the algebraic numbers
$\alpha(\fraka,m)$ in Theorem \ref{thm:intro-pre-pvals}. In the case of a prime discriminant
the formulas are particularly simple and explicit.

\begin{proposition}[Proposition 2 in \cite{ehlen-intersection}]
  \label{prop:Zprime-intro}\ 
  \begin{enumerate}
  \item We have $\calZ(m)_{\frakP} = 0$ unless $\abs{\Diff(m)} = 1$.
  \item Assume that $\Diff(m) = \{p\}$. 
   Let $\frakP_0 \mid p$ be the unique prime ideal that is fixed by complex conjugation, $\overline\frakP_{0} = \frakP_{0}$.
    For $\frakP = \frakP_{0}^{\sigma}$, where $\sigma = \sigma(\fraka) \in \Gal(\Hkb/\kb)$
    corresponds to the ideal class of $\fraka$ under the Artin map, we have
    \[
      \calZ(m)_{\frakP} = 2^{o(m)-1} \nu_{p}(m) \rho(m\abs{D}/p, [\fraka]^{-2}).
    \]
  \end{enumerate}
\end{proposition}

In Section \ref{sec:compr-example}, we walk through the proof of Theorem \ref{thm:pre-pvals}
in a concrete example for $D=-23$ and give in fact a finite formula for the coefficients
of the holomorphic part in this case. It should be possible to generalize the results obtained in
the example and we will come back to this in a sequel to this paper.

Some remarks regarding the relation of our work to
\cite{MarynaInnerProd} and \cite{DukeLiMock} are in order.
We make use of the same seesaw identity as in \cite{MarynaInnerProd}, where it has been used to obtain formulas for regularized Petersson inner products of weakly holomorphic modular forms
of weight one and the theta function $\Theta_{P}(\tau,1)$ for the lattice $P = \calO_\kb$ as above
in the case of a prime discriminant. We use the same seesaw identity (essentially \eqref{eq:seesawintro1})
but focus on the coefficients of harmonic Maa\ss{} forms and applications in arithmetic geometry, which do not appear in \cite{MarynaInnerProd} and requires some extra work.
Duke and Li \cite{DukeLiMock} make use of the Rankin-Selberg method and the construction of (higher level)
automorphic Green functions as in \cite{grosszagier-singularmoduli} and \cite{grosszagier},
whereas we entirely rely on the (regularized) theta lift machinery. 
In particular, Theorem 1.2 of \cite{DukeLiMock} is a special case of Theorem \ref{thm:cmval-intro} (or rather Theorem \ref{thm:value-phiz} in the body of the paper) in signature $(2,1)$.
We refer the reader to \cite[Section 8]{cmvals-preprint1-ehlen} for this particular case.

\subsection{Acknowledgements}
This article contains the (improved) main results of my thesis \cite{ehlen-diss}.
I am greateful to my advisor Jan Hendrik Bruinier for introducing me to the subject, for his constant support,
for helpful discussions and his comments on earlier versions of this article and my thesis.
I thank Claudia Alfes, Eric Hofmann, Yingkun Li, Fredrik Str\"{o}mberg, Maryna Viazovska and Tonghai Yang for their helpful comments.
Moreover, I would like to thank the anonymous referees for reading the paper carefully, for pointing out errors and
providing many useful comments and suggestions that helped to improve the exposition
and the results of this paper.

\endgroup

\section{Preliminaries}
\label{chap:prelims}
We introduce some notation for the rest of the article.
For a place $\frakp$ of a field $\kb$ we let $\kb_{\frakp}$ denote the completion
of $\kb$ with respect to the valuation $v_{\frakp}$ corresponding to $\frakp$.
For non-archimedean $\frakp$, we denote by $\calO_{\frakp} \subset \kb_{\frakp}$ the corresponding valuation ring.
We consider the adeles $\adeles_{\kb} = \rprod_{\frakp} \kb_{\frakp}$ over
$\kb$ and the finite adeles are denoted by $\adeles_{\kb,f} = \rprod_{\frakp \nmid \infty} \kb_{\frakp}$.
\nomenclature[Ad]{$\adeles$, $\adeles_\kb$}{The adeles over $\Q$ and $\kb$, respectively}

Moreover, we denote by $\adeles_{\kb}^{\times}$ and $\adeles_{\kb,f}^{\times}$ the groups of (finite) ideles over $\kb$.
If $\kb=\Q$, we denote by $\adeles$ the adeles over $\Q$ and by $\adeles_{f}$ the finite adeles.
\nomenclature[Adf]{$\adeles_f$, $\adeles_{\kb,f}$}{The finite adeles over $\Q$ and $\kb$, respectively}
\nomenclature[Adsi]{$\adeles^\times$, $\adeles_{\kb}^\times$}{The ideles over $\Q$ and $\kb$}
\nomenclature[Adsfi]{$\adeles_{f}^\times$, $\adeles_{\kb, f}^\times$}{The finite ideles over $\Q$, $\kb$}

Throughout, we let $n$ be a non-negative integer and let $V$ be a quadratic space over $\Q$ of type $(2, n)$ with a non-degenerate quadratic form $Q$. 
\nomenclature[V]{$V$}{A rational quadratic space of signature $(2,n)$}
\nomenclature[n]{$n$}{a non-negative integer, s.t. $V$ has signature $(2,n)$, sometimes also used as index in sums}

\subsection{Shimura varieties attached to $\GSpin_V$}
\label{sec:shimura-varieties}
\label{sec:symmdom-shimura}
As in the introduction we abbreviate $H: = \GSpin_V$.
\nomenclature[H]{$H$}{Frequently, $H = \GSpin_V$, the general spin group}
One has the following exact sequence of algebraic groups:
\begin{equation*}
  \label{eq:Gspinseq}
  1 \longrightarrow \G_m \longrightarrow \GSpin_{V} \longrightarrow \SO_V \longrightarrow 1.
\end{equation*}
Here, $\G_{m}$ denotes the multiplicative group.

\begin{remark}
  Our setup is basically the same as in \cite{kudla-integrals, schofer, bryfaltings}.
  However, we warn the reader that we are working with a quadratic space of type $(2,n)$,
  whereas these references mostly use type $(n,2)$ quadratic spaces.
\end{remark}
Let $\bm{K} \subset \SO_V(\R)$ be a maximal compact subgroup of $\SO_V(\R)$.
Since $V$ is a quadratic space over $\Q$ of type $(2,n)$, the quotient $\SO_V(\R)/\bm{K}$ is a symmetric space with a complex structure. There are several ways to realize $\SO_V(\R)/\bm{K}$.

Consider the \emph{Grassmannian} $\domain$ of oriented two-dimensional positive definite subspaces of $V(\R) = V \otimes_{\Q} \R$. That is, we let
\begin{equation*}
  \domain: = \{z^{\pm} \, \mid \, z \subset V(\R), \dim z=2, \ Q|_z > 0 \}.
\end{equation*}
\nomenclature[D1]{$\domain$}{The symmetric domain attached to $\SO_V(\R)$, realized as Grassmannian}
Here, for each 2-dimensional positive definite subspace $z \subset V(\R)$,
we write $z^{+}$ and $z^{-}$ for $z$ together with one of the two possible choices of orientation.
The group $H(\R)$ acts naturally on $\domain$ and this action is transitive by Witt's theorem.
The Grassmannian has two connected components and each of them is
isomorphic to the symmetric space $\SO_V(\R)/\bm{K}$.

Let $K \subset H(\adeles_f)$ be a compact open subgroup. We write
$X_K$ for the associated \emph{Shimura variety} with complex points
\begin{equation*}
  X_K(\C) = H(\Q) \backslash ( \domain \times H(\adeles_f)/K).
\end{equation*}
\nomenclature[XK]{$X_K$}{A Shimura variety}

\subsection{Special divisors}
\label{sec:heegner-divisors}
There is a natural family of divisors on Shimura varieties of orthogonal type that will play an important
role. We also refer to and \cite{kudla-cycles-O}, \cite{kudla-integrals}, and \cite{bryfaltings},.
Let $L \subset V(\Q)$ be an even lattice and let $K \subset H(\adeles_f)$
be an open compact subgroup such that $KL \subset L$ and $K$ acts trivially
on $L'/L$. We will make these assumptions throughout this section.
In this situation, we consider the group
\[
\Gamma_{K} = H(\Q) \cap K,
\]
which is an arithmetic subgroup of $H(\Q)$.

Let $x \in V(\Q)$ be a vector of negative norm and
denote the orthogonal complement ${x^{\perp} \subset V(\Q)}$
by $V_{x}$. We let $H_{x}$ be the stabilizer of $x$ in $H$.
Then $H_{x} \cong \GSpin(V_{x})$ and the Grassmannian
\[
  \domain_{x} = \{ z \in \domain\ \mid\ z \perp x \} \subset \domain
\]
defines an analytic set of codimension one in $\domain$.

Let $h \in H(\adeles_{f})$ and consider
\begin{equation}
  \label{eq:Zxh}
  H_{x}(\Q) \bs \domain_{x} \times H_{x}(\adeles_{f}) / (H_{x}(\adeles_{f})\cap hKh^{-1}) \longrightarrow X_{K}
\end{equation}
given by
\[
(z,h_{1}) \mapsto (z,h_{1}h).
\]
The image of this map defines a divisor $Z(x,h)$ on $X_{K}$ which is rational over $\Q$ \cite{kudla-cycles-O}.
For $m \in \Q_{<0}$ consider the quadric $\Omega_{m} \subset V$ given by
\[
\Omega_m = \{ x\in V\ \mid\ Q(x) = m\}.
\]
By Witt's theorem, if $\Omega_{m}(\Q) \neq \emptyset$, the orthogonal group acts transitively and thus
for every $x_0 \in \Omega_{m}(\Q)$ we have $\Omega_{m}(\Q) = H(\Q)x_{0}$
and $\Omega_{m}(\adeles_{f}) = H(\adeles_{f})x_{0}$.
Here,
\[
\Omega_{m}(\adeles_f) = \left( \prod_{p \nmid \infty} \Omega_{m}(\Q_{p}) \right) \cap V(\adeles_f)
\]
and $\Omega_{m}(\Q_{p}) = \{ x \in V(\Q_{p})\, \mid\, Q(x) = m \}$.
Moreover, for any compact open subgroup $K \subset H(\adeles_f)$,
we have $\Omega_m(\adeles_f) = K \Omega_{m}(\Q)$ (see Lemma 5.1 of \cite{kudla-cycles-O}).

We let $S(V(\adeles))$ be the space of Schwartz(-Bruhat) functions on $V(\adeles)$.
That is, the space $S(V(\R))$ is the usual space of Schwartz (rapidly decreasing)
functions on $V(\R)$ and $S(V(\Q_p))$ is the space of locally constant functions
$V(\Q_{p}) \rightarrow \C$ with compact support and we let
\[
  S(V(\adeles_{f})) = \bigotimes_{p < \infty} S(V(\Q_{p}))
\]
and $S(V(\adeles)) = S(V(\adeles_f)) \otimes S(V(\R))$.

Let $L$ be an even lattice and $\mu \in L'/L \cong \hat{L}'/\hat{L}$, where
$\hat{L} = L \otimes_\Z \Zhat$ with $\Zhat = \prod_{p<\infty} \Z_p$.
\nomenclature[Zh]{$\Zhat$}{$=\prod_{p<\infty} \Z_p$}
\nomenclature[Lh]{$\hat{L}$}{$=L\otimes_{\Z} \Zhat$}
We let $\phi_\mu \in S(V(\adeles_f))$ be the characteristic function of $\mu$.
\nomenclature[phimu]{$\phi_{\mu}$}{The characteristic function of $\mu + L$}
We consider the finite dimensional subspace
\begin{equation*}
  S_L = \bigoplus_{\mu \in L'/L} \C\phi_\mu \subset S(V(\adeles_f)).
\end{equation*}
\nomenclature[SL]{$S_L$}{The span of the characteristic functions $\phi_{\mu}$ in $S(V(\adeles_{f}))$}

\begin{definition}
  \label{def:Zm}
  For a Schwartz function $\varphi \in S_{L}$ write
  \[
  \supp(\varphi) \cap \Omega_m(\adeles_f) = \bigsqcup_{j} K \xi_{j}^{-1} x_{0},
  \]
  where $\xi_{j} \in H(\adeles_{f})$. We define the special divisor
  \[
  Z(m,\varphi) := \sum_{j} \varphi(\xi_j^{-1}x_0)Z(x_0,\xi_{j}).
  \]
  For $\mu \in L'/L$, briefly write $Z(m,\mu) := Z(m,\phi_{\mu})$.
\end{definition}
\nomenclature[Zm2]{$Z(m,\mu)$}{A special divisor}

In this context, we also let
\[
  L_{m} := \Omega_{m} \cap L' \quad \text{and} \quad L_{m,\mu} := L_{m} \cap (L + \mu).
\]
\nomenclature[Lmu]{$L_{m,\mu}$}{Norm $m$ elements in $L + \mu$}

\subsection{The Weil representation}
\label{sec:weil-representation}
Let $L \subset V$ be an even lattice and let $\mu \in L'/L$.

We write $\langle \phi, \chi \rangle$ for the standard bilinear pairing between $S(V(\adeles))$ and its dual $S(V(\adeles))^{\vee}$.
In particular
\[
  \langle a \phi_{\mu}, b \phi_{\nu} \rangle = ab\, \delta_{\mu,\nu}
\]
for $a,b \in \C$ and $\mu,\nu \in L'/L$, where we identify $S_{L}$ with its dual.
\nomenclature[()]{$\langle \cdot,\cdot\rangle$}{The $\C$-bilinear pairing between $S(V(\adeles))$ and its dual}

\begin{remark}
  We note that the space $S_{L}$ can be identified with the group ring
  $\C[L'/L]$ of the finite abelian group $L'/L$ via
  $\phi_{\mu} \mapsto \frake_{\mu}$, if $\{\frake_{\mu}\, \mid\, \mu \in L'/L\}$
  is the standard basis for $\C[L'/L]$.
  In the latter space the corresponding scalar product is conjugate-linear in the second argument.
\end{remark}

We write $\widetilde{\Gamma} := \Mp_2(\Z)$ for the two-fold metaplectic cover
of $\SL_2(\Z)$. The elements of $\widetilde{\Gamma}$ are pairs $(A,\phi)$,
where $A = \smallabcd \in \SL_2(\Z)$ and $\phi: \uhp \to \C$ is a holomorphic function
with $\phi^2(\tau) = c\tau + d$.

There is a representation $\rho_{L}: \widetilde{\Gamma} \rightarrow \Aut S_L$,
usually called the Weil representation associated with $L$.
This representation can be described explicitly as follows.

The group
$\widetilde{\Gamma}$ is generated by
\begin{equation}
  S = \left( \Smatrix, \sqrt \tau \right), \quad T = \left( \Tmatrix, 1 \right).
\end{equation}
\nomenclature[rhoL]{$\rho_L$}{The Weil representation associated with $L$}
\nomenclature[S]{$S$}{The matrix $\smallSmatrix \in \SL_2(\Z)$ or the element $(\smallSmatrix, \sqrt{\tau}) \in \Mp_2(\Z)$}
\nomenclature[T0]{$T$}{The matrix $\smallTmatrix \in \SL_2(\Z)$ or the element $(\smallTmatrix, 1) \in \Mp_2(\Z)$}
For these generators, we have
\begin{equation}
  \begin{aligned}
    \rho_L(T)\phi_\mu &= e(Q(\mu))\phi_\mu, \\
    \rho_L(S)\phi_\mu &= \frac{e(-\sgn(V)/8)}{\sqrt{|L^\prime/L|}}
    \sum_{\nu \in L^\prime/L}e(-(\mu,\nu))\phi_\nu.
  \end{aligned}
\end{equation}
Here, $\sgn(V)$ denotes the signature of $V$, which is equal to $2-n$ in our case.

\subsection{Harmonic weak Maa\ss{} forms}
\label{sec:harmonic-weak-maass}
The main reference for this section is the fundamental
article by Bruinier and Funke \cite{brfugeom}.

Let $(V,Q)$ be a rational quadratic space
and let $L \subset V$ be an even lattice. Moreover, let $k \in \frac{1}{2}\Z$.
For $(\gamma,\phi) \in \tilde\Gamma$, we define the \emph{Petersson slash operator} on functions
$f: \uhp \rightarrow S_L$ by
\begin{equation*}
  \left( f \mid_{k,L} (\gamma,\phi)\right) (\tau) = \phi(\tau)^{-2k} \rho_L((\gamma,\phi))^{-1} f(\gamma \tau).
\end{equation*}
\nomenclature[-kL]{$\mid_{k,L}$}{The Petersson slash operator on vector valued functions}
\begin{definition}
  A twice continuously differentiable function $f:\uhp \to S_L$
  is called a {\em harmonic weak Maa\ss{} form} (of weight $k$ with
  respect to $\tilde\Gamma$ and $\rho_L$) if it satisfies:
  \begin{enumerate}
  \item[(i)]
    $f \mid_{k,L} \gamma = f$ for all $\gamma \in \tilde\Gamma$,
  \item[(ii)]
    there is a $C>0$ such that $f(\tau)=O(e^{C v})$ as $v\to \infty$
    (uniformly in $u$, where $\tau=u+iv$),
  \item[(iii)]
    $\Delta_k f = 0$, where
    \begin{align*}
      \Delta_k := -v^2\left( \frac{\partial^2}{\partial u^2}+
        \frac{\partial^2}{\partial v^2}\right) + ikv\left(
        \frac{\partial}{\partial u}+i \frac{\partial}{\partial v}\right)
    \end{align*}
    is the hyperbolic Laplace operator in weight $k$.
  \end{enumerate}
\end{definition}
We denote the space of harmonic weak Maa\ss{} forms of weight $k$ with
respect to $\rho_L$ by $\calH_{k,L}$ and
write $M^!_{k,L}$ for the subspace of weakly holomorphic modular forms.
Moreover, we write $S_{k,L}$ and $M_{k,L}$ for the subspaces of cusp forms
and holomorphic modular forms.
As usual, elements of all three spaces $S_{k,L} \subset M_{k,L} \subset M^{!}_{k,L}$ are assumed to
be holomorphic on the upper half-plane;
weakly holomorphic modular forms are allowed to have a pole at the cusp,
holomorphic modular forms are required to be holomorphic at the cusp and
cusp forms are holomorphic modular forms that vanish at the cusp.
\nomenclature[Hkl]{$\calH_{k,L}$}{The space of harmonic weak Maa\ss{} forms of weight $k$ and representation $\rho_L$}
\nomenclature[Mkl]{$M_{k,L}$}{The space of holomorphic modular forms of weight $k$ and representation $\rho_L$}
\nomenclature[Mkla]{$M^{"!}_{k,L}$}{The space of weakly holomorphic modular forms of weight $k$ and representation $\rho_L$}
\nomenclature[Skl]{$S_{k,L}$}{The space of cusp forms of weight $k$ and representation $\rho_L$}

We write the Fourier expansion of $f \in \calH_{k,L}$ as
\begin{equation}
  f(\tau) = \sum_{\mu \in L'/L} \sum_{n \in \Q} c_{f}(n, \mu, v) q^n\phi_\mu.
\end{equation}

Since $f$ is harmonic with respect to the weight $k$ Laplace operator,
the coefficients $c_f(n,\mu,v)$ satisfy $\Delta_k c_f(n, \mu, v)\exp(2\pi n (u+iv)) = 0$.
Computing a basis for the space of solutions to this differential equation
gives rise to a unique decomposition 
of the Fourier expansion of $f$ into a \emph{holomorphic part} $f^+$ and
a \emph{non-holomorphic part} $f^-$. 
If $n \neq 0$, we write accordingly $c_f(n,\mu,v) = c_f^+(n,\mu) + c_f^-(n,\mu)W_k(2\pi n v)$,
where $W_k(a) = \int_{-2a}^\infty e^{-t}t^{k-2}\, dt$.
\nomenclature[f+]{$f^{+}$}{The holomorphic part of $f$}
\nomenclature[f-]{$f^{-}$}{The non-holomorphic part of $f$}

In weight one, which is of particular interest for us,
the expansion of the non-holomorphic part is of the form
\begin{equation}
  f^-(\tau) = \sum_{\mu\in L'/L} \left(
    c_{f}^-(0,\mu) \log(v)
    + \sum_{\substack{n\in \Q\\ n \neq 0}} c_{f}^-(n,\mu) W_1(2\pi nv) q^n
  \right) \phi_\mu, \label{f-k1}
\end{equation}
where $W_1(a) = \Gamma(0,-2a)$ for $a<0$.

We recall a few more facts that can all be found in \cite[Section 3]{brfugeom}.
We denote by $L^-$ the lattice given by the $\Z$-module $L$ together with the quadratic form $-Q$.
\nomenclature[L-]{$L^{-}$}{The quadratic module given by $L$ together with the quadratic form $-Q$}
There is an antilinear differential operator
$\xi := \xi_k: \mathcal{H}_{k,L} \to M^{!}_{2-k,L^-}$, defined by
\begin{equation}
  \label{defxi} f(\tau)\mapsto \xi(f)(\tau)
  :=v^{k-2} \overline{L_k f(\tau)} = R_{-k} v^k\overline{ f(\tau)}.
\end{equation}
\nomenclature[xi]{$\xi$}{Usually, $\xi=\xi_{k}$, a differential operator}

Here $L_k$ and $R_k$ are the Maa\ss{} lowering and raising operators,
\begin{equation*}\label{def:RkLk}
  L_k = -2iv^2\frac{\partial}{\partial\overline{\tau}} \quad \text{and} \quad 
  R_k = 2i \frac{\partial}{\partial\tau} + k v^{-1}.
\end{equation*}
We let $H_{k,L} \subset \calH_{k,L}$ be the subspace of forms with cuspidal ``shadow'',
  \[
  H_{k,L} := \{ f \in \calH_{k,L}\ \mid\ \xi(f) \in S_{k,L^{-}} \}.
  \]
Alternatively, we could define this to be the space
of $f \in \calH_{k,L}$, such that there is a Fourier polynomial
\[
P_{f}(\tau) = \sum_{\mu \in L'/L} \sum_{m<0} c_{f}^+(m,\mu) e(m\tau), \text{ with } f - P_{f}(\tau) = O(1)
\]
as $\Im(\tau) \rightarrow \infty$.
The Fourier polynomial $P_{f}(\tau)$ is also called the \emph{principal part} of $f$.
Note that this space was denoted by $H^{+}_{k,L}$ in \cite{brfugeom}.
\nomenclature[HkL]{$H_{k,L}$}{Harmonic weak Maa\ss{} forms $f \in \calH_{k,L}$ such that $\xi_k(f) \in S_{2-k,L^-}$}
\nomenclature[Pf]{$P_{f}(\tau)$}{The principal part of $f$ (not including the constant term)}

The kernel of $\xi_k$ is equal to $M^!_{k,L}$ and
by \cite[Corollary~3.8]{brfugeom}, the sequences
\begin{gather}
  \label{ex-sequ}
  \xymatrix@1{ 0 \ar[r] & M_{k,L}^! \ar[r] & \calH_{k,L} \ar[r]^-{\xi_{k}} & M^!_{2-k,L^-} \ar[r] & 0} \\
  \xymatrix@1{ 0 \ar[r] & M_{k,L}^! \ar[r] & H_{k,L} \ar[r]^-{\xi_{k}} & S_{2-k,L^-} \ar[r] & 0} \label{ex-sequS}
\end{gather}
are exact.

For $f \in S_{k,L}$ and $g \in M_{k,L}$, we define the \emph{Petersson inner product}
of $f$ and $g$ as
\[
  (f,g) = \int_{\SL_2(\Z) \bs \uhp} \langle f(\tau), \overline{g(\tau)} \rangle v^{k}  d\mu(\tau).
\]
\nomenclature{$(f,g)$}{The Petersson inner product of $f$ and $g$ (vector valued modular forms)}

We denote by $\del$ and $\delbar$ the usual Dolbeault operators, such that we have
$d = \del + \delbar$ for the exterior derivative on differential forms on $\uhp$.
\begin{lemma}
  \label{lem:xidifform}
  In terms of differential forms, we have
  \[
    \bar\del(fd\tau) = -v^{2-k} \overline{\xi_k(f)} d\mu(\tau) = -L_{k}f d\mu(\tau).
  \]
\end{lemma}

Using the Petersson inner product and the operator $\xi_k$,
we obtain a bilinear pairing between $g \in M_{2-k,L^{-}}$
and $f \in H_{k,L}$ via
\begin{equation}
  \label{eq:pairing}
  \{g, f\} := (g, \xi_k(f) )_{2-k} = \int_{\SL_{2}(\Z) \bs \uhp} \langle g, \overline{\xi_k(f)} \rangle v^{2-k} d\mu(\tau)
  = \int_{\SL_{2}(\Z) \bs \uhp} \langle g, L_kf \rangle d\mu(\tau).
\end{equation}

Using Lemma \ref{lem:xidifform},
the following result is essentially an application of Stokes' theorem
(see \cite[Proposition 3.5]{brfugeom}).
\begin{lemma}
  \label{lem:pairing}
  Let $f \in H_{k,L}$ and $g \in M_{2-k,L^{-}}$. Then
  \[
    \{ g, f \} = \sum_{\mu \in L'/L} \sum_{n \leq 0} c^{+}(n,\mu)b(-n,\mu),
  \]
  which implies that the pairing only depends on the principal part of $f$ (and on $g$).
  The exact sequence \eqref{ex-sequ} implies that the pairing between $S_{2-k,L^{-}}$ and
  $H_{k,L}/M^{!}_{k,L}$ is non-degenerate.
\end{lemma}

\subsection{Regularized theta lifts}
\label{sec:theta}
We let $\Theta_{L}(\tau,z,h)$ be the Siegel theta function associated with $L$
as in \cite{bryfaltings}, where $\tau \in \uhp$, $z \in \domain$ and $h \in H(\adeles_{f})$, where $H=\GSpin_V$.
If $V$ is positive definite and $\domain = \{z^\pm\}$ consists
of only two points, we frequently drop $z$ from the notation and just write $\Theta_L(\tau,h)$.
The following theorem can be found (in a ``more classical'' language) in \cite{boautgra}. 
A reference that uses our (adelic) setup is Kudla's seminal paper \cite{kudla-integrals}.
\begin{theorem}
  If $K \subset H(\adeles_f)$ is an open compact subgroup preserving $L$ and acting trivially on $L'/L$,
  then the Siegel theta function $\Theta_L(\tau,z,h_f)$ defines a function on the Shimura variety $X_K$ (in $(z,h_f)$).
  Moreover, as a function in $\tau \in \uhp$, it is a non-holomorphic vector-valued modular form
  of weight $(2-n)/2$, that is, for
  $\gamma = \left(\smallabcd,\phi(\tau)\right) \in \widetilde{\Gamma}$,
  we have
  \[
  \Theta_L(\gamma\tau, z, h) = \phi(\tau)^{2-n}\rho_L(\gamma) \Theta_L(\tau,z,h).
  \]
\end{theorem}

Let $\calF := \{\tau \in \uhp; \abs{\tau} \geq 1,\, -1/2 \leq \Re(\tau) \leq 1/2 \}$
the standard fundamental domain for the action of $\SL_2(\Z)$ on $\uhp$ 
and let $\calF_{T}:=\{\tau \in \calF;\ \Im(\tau) \leq T\}$.
\nomenclature[F2]{$\calF_{T}$}{A truncated fundamental domain}
\nomenclature[F1]{$\calF$}{The standard fundamental domain for $\SL_2(\Z)$}
For $f \in H_{1-n/2,L}$, we consider the regularized theta integral
\[
  \Phi_L(z,h,f) = \int_{\Gamma \backslash \uhp}^{\reg}
   \langle f(\tau), \overline{\Theta_L(\tau,z,h)} \rangle v^k d\mu(\tau) := \CT_{s=0} \left[
  \lim_{T\rightarrow \infty}
   \int_{\calF_T} \langle f(\tau), \overline{\Theta_L(\tau,z,h)} \rangle v^{k-s} d\mu(\tau) \right].
\]
Here, $\CT\limits_{s=0}$ denotes the constant term in the Laurent expansion at $s=0$
of the meromorphic continuation of the function enclosed in $\left[ \cdot \right]$, which is initially defined for $\Re(s)$ large enough.
\nomenclature[CT]{$\CT_{s=0}[A(s)]$}{The constant term in the Laurent series expansion of $A(s)$ at $s=0$}

Associated with $f$ is the divisor
\begin{equation}
  \label{eq:Zf}
  Z(f) = \sum_{\mu \in L'/L} \sum_{m<0} c^+(m,\mu) Z(m,\mu).
\end{equation}
\nomenclature[Zf]{$Z(f)$}{A special divisor attached to $f$}
\begin{theorem}[\cite{boautgra}, Theorem 13.3, cf. Theorem 1.3 in \cite{kudla-integrals}]
  \label{thm:borcherds}
Let $f \in M_{(2-n)/2,L}^{!}$ with $c(m,\mu) \in \Z$ for all $m<0$
and $c(m,\mu) \in \Q$ for all $m \in \Q$.
There is a function $\Psi_L(z,h,f)$ on $\domain\times H(\adeles_f)$, such that:
\begin{enumerate}
\item $\Psi_L(z,h,f)$ is a meromorphic modular form for $H(\Q)$ of
  weight $c_f(0,0)/2$ and level $K$, with some unitary multiplier system of
  finite order,
\item the divisor of $\Psi_L(z,h,f)^2$ on $X_{K}$ is given by $Z(f)$.
\item and we have
      $\Phi_L(z,h,f) =  -2 \log\lVert \Psi_L(z,h,f) \rVert^{2} - c_{f}(0,0)(\log(2\pi) +\Gamma'(1))$,\\
      where $\lVert \Psi_L(z,h,f) \rVert^{2} = \abs{\Psi_L(z,h,f)}^2 \abs{y}^{c_{f}(0,0)}$.
\end{enumerate}
\end{theorem}
The modular form $\Psi_L(z,h,f)$ admits an expansion as an infinite product, giving it the name
\emph{Borcherds product}.
We will only use the product expansion in a special case in Section \ref{sec:coeff-holom-part}.
Therefore, we omit the general case and refer to Theorem 13.3 of \cite{boautgra}.

Bruinier extended the space of input functions of the theta lift to include harmonic weak Maa\ss{} forms and this extension will play an important role for us.
Recall that a function $G(z)$ on $\domain$ has a logarithmic singularity along a divisor $D$,
if every point in $\domain$ has a small neighborhood $U$, such that
for any meromorphic function $g$ locally defining $D$, we have that $G - \log|g|$ extends to a smooth function on $U$.
\begin{theorem}[\cite{brhabil}, Theorem 2.12, Theorem 4.7]\ \\
  Let $f \in H_{1-n/2,L}$. Then the following holds.
\begin{enumerate}
\item The function $\Phi_L(z,h,f)$ is smooth on $X_K \backslash Z(f)$ and has a logarithmic
  singularity along $-2 Z(f)$.
\item The differential $dd^c \Phi_L(z,h,f)$ extends to a smooth $(1,1)$ form on $X_K$.
As a current on $X_K$, we have 
\[
  dd^c  \left[ \Phi_L(z,h,f) \right] + \delta_{Z(f)} = \left[dd^c  \Phi_L(z,h,f)\right].
\] 
\end{enumerate}
\end{theorem}

\subsection{Operations on vector-valued modular forms}
Let $A_{k,L}$ be the space of $S_L$-valued functions that are invariant under the weight $k$ slash operator.
Let $M \subset L$ be a sublattice of finite index. Then if $f \in A_{k,L}$, it can be naturally viewed as an element of $A_{k, M}$. Indeed, we have the
inclusions $M \subset L \subset L'\subset M'$ and therefore
\[
L/M \subset L'/M\subset M'/M.
\]
We have the natural map $L'/M \to L'/L$, $\mu\mapsto \bar \mu$.
\begin{lemma}
  \label{sublattice} There are  two natural maps
  $$
  \res_{L/M}: A_{k,L} \rightarrow  A_{k,M},
  \quad f\mapsto f_M
  $$
  and
  $$  \tr_{L/M}: A_{k,M}\rightarrow  A_{k,L},
  \quad g \mapsto g^L
  $$
  such that for any $f \in A_{k,L}$ and $g\in A_{k,M}$
  $$
  \langle f, \bar g^L\rangle =\langle f_M, \bar g \rangle.
  $$
  They are given as follows. For  $\mu\in M'/M$ and $f \in
  A_{k,L}$,
  \[
  (f_M)_\mu = \begin{cases} f_{\bar\mu},&\text{if $\mu\in L'/M$,}\\
    0,&\text{if $\mu\notin L'/M$.}
  \end{cases}
  \]
  For any $\bar\mu \in L'/L$, and $g \in A_{k,M}$, let $\mu$ be a
  fixed preimage of $\bar\mu$ in $L'/M$. Then
  $$
  (g^L)_{\bar\mu} =\sum_{\alpha \in L/M} g_{\alpha +\mu}.
  $$
\end{lemma}
\nomenclature[fM]{$f_{M}$}{The image of $f$ under the restriction map $\res$}
\nomenclature[fL]{$f^{L}$}{The image of $f$ under the trace map $\tr$}
\begin{proof}
  See \cite[Proposition 6.9]{schehabil} for the map $\res_{L/M}$.
  The proof for $\tr_{L/M}$ is a similar calculation.
\end{proof}

\subsection{Jacobi forms and vector valued modular forms}
\label{sec:jacobi-forms}
In this section we recall the notion of a (weakly holomorphic) Jacobi form.
We will use Jacobi forms of scalar index
which have been intensively studied by Eichler and Zagier \cite{ez}
in Section \ref{sec:coeff-holom-part}.
\begin{definition}
  Let $k,m \in \Z$ and $\varphi: \uhp \times \C \rightarrow \C$ be a holomorphic function.
  Then $\varphi$ is called a \emph{holomorphic Jacobi form} of weight $k$ and index $m$
  if
  \begin{enumerate}
  \item $\varphi(\gamma\tau, \frac{z}{c\tau+d}) = (c\tau+d)^{k} e(mcz^2/(c\tau+d)) \varphi(\tau,z)$,
    for all $\gamma = \smallabcd \in \SL_2(\Z)$,
  \item $\varphi(\tau,z+r\tau+s) = e(-m(r^2\tau+2rz)) \varphi(\tau,z)$ for all $r,s \in \Z$, and
  \item $\varphi(\tau,z)$ is holomorphic at the cusp $\infty$.
  \end{enumerate}
\end{definition}
Note that such a $\varphi$ has a Fourier expansion of the form
\[
  \varphi(\tau,z) = \sum_{n,r \in \Z} c(n,r) q^n\zeta^{r},
\]
where $q=e^{2\pi i \tau}$ and $\zeta = e^{2\pi i z}$.
The last condition in the definition means that
$c(n,r)=0$ if the discriminant $4nm-r^{2}$ is negative.

A \emph{weakly holomorphic Jacobi form} satisfies all the preceding conditions
except that we only require it to be meromorphic at $\infty$.
This means in terms of the Fourier expansion that there are only finitely many non-vanishing
Fourier coefficients with negative discriminant.
We denote the space of holomorphic Jacobi forms of weight $k$ and index $m$ by $J_{k,m}$
and by $J_{k,m}^{!}$ the space of weakly holomorphic Jacobi forms of weight $k$ and index $m$.

Using the definitions, it is easy to check that
\begin{enumerate}
\item If $f \in J_{k_1,m_1}^{!}$ and $g \in J_{k_2,m_2}^{!}$, then $fg \in J^{!}_{k_1+k_2, m_1+m_2}$.
\item If $f \in M_{k_1}^{!}(\SL_2(\Z))$ and $\varphi \in J_{k_2,m}^{!}$, then $f\varphi \in J_{k_{1}+k_2,m}^{!}$.
\item If $\varphi \in J_{k,m}^{!}$, then $\varphi(\tau, 0) \in M_k^{!}$.
\end{enumerate}

For $r \in \Z/2m\Z$ we write
\[
  \theta_{r}(\tau,z) = \sum_{\substack{n \in \Z \\ n \equiv r \bmod{2m}}} q^{\frac{n^{2}}{4m}}\zeta^{n},
\]
for the corresponding theta function.
\begin{proposition}
  \label{prop:jac-theta-exp}
  Let $\varphi \in J_{k,m}^{!}$ be a weakly holomorphic Jacobi form.
  Then $\varphi(\tau,z)$ has a theta expansion of the form
  \[
  \varphi(\tau,z) = \sum_{r \bmod{2m}} \varphi_{r}(\tau) \theta_{r}(\tau,z).
  \]
  Moreover, let $L = \Z$ be the lattice with quadratic form $Q(x) = -mx^{2}$.
  Then
  \[
  \Phi(\tau) = \sum_{r \bmod{2m}} \varphi_r(\tau)\phi_{r}
  \]
  is a vector valued modular form contained in $M_{k-1/2,L}^{!}$.
  Here, $\phi_{r}$ is the characteristic function of the coset $r + 2m\Z$.
  This correspondence establishes an isomorphism
  \[
  M_{k-1/2,L}^{!} \cong J_{k,m}^{!}.
  \]
\end{proposition}
\begin{proof}
  For holomorphic Jacobi forms, this is Theorem 5.1 of \cite{ez}.
  It is straightforward to extend this to weakly holomorphic forms.
  See also \cite[Section 8]{zagiertraces}.
\end{proof}

\begin{lemma} \label{lem:jac-theta-pairing}
  With the same notation as in Proposition \ref{prop:jac-theta-exp},
  let $\varphi \in J^{!}_{k,m}$ be a weakly holomorphic Jacobi form.
  We let $\Theta_{L^{-}}(\tau)$ be the theta function
  \[
  \Theta_{L^{-}} (\tau) = \sum_{n \in \Z} e\left( \frac{n^{2}}{4m} \tau\right) \phi_{n} =
  \sum_{r \in \Z/2m\Z} \Theta_{L^{-},r}(\tau) \phi_{r} \in M_{\frac{1}{2},L^{-}}
  \]
  associated with the lattice $L^{-}$.
  Then we have
  \[
  \langle \Phi(\tau), \Theta_{L^{-}}(\tau) \rangle = \varphi(\tau,0).
  \]
\end{lemma}
\begin{proof}
  This follows directly from Proposition \ref{prop:jac-theta-exp}
  using
  \[
  \langle \Phi(\tau), \Theta_{L^{-}}(\tau) \rangle = \sum_{r \in \Z/2m\Z} \varphi_{r}(\tau)\Theta_{L^{-},r}(\tau)
  = \sum_{r \in \Z/2m\Z} \varphi_{r}(\tau)\theta_{r}(\tau,0).
  \]
\end{proof}

From the properties stated above, it follows that the bi-graded ring
\[
  J_{ev,\ast}^{!} = \bigoplus_{k,m \in \Z} J_{2k,m}^{!}
\]
of weakly holomorphic Jacobi forms of even weight
is a module over the graded ring
\[
M_{\ast}^{!} = M_{\ast}^{!}(\SL_2(\Z)) = \bigoplus_{k \in \Z} M_k^{!}(\SL_2(\Z)).
\]
(Note that $M_{k}^{!}(\SL_2(\Z)) = \{0\}$ for $k$ odd.)
\begin{proposition}[See \cite{zagiertraces}]
  \label{Jevgens}
  The $M_{\ast}^{!}$-module $J_{ev,\ast}^{!}$
  of weakly holomorphic Jacobi forms of even weight
  is free of rank two and generated by
  \begin{align*}
    \FJ(\tau,z) = \phi_{-2,1}(\tau,z)
    &= (\zeta -2 +\zeta^{-1})+(-2 \zeta^2 + 8 \zeta -12 +8\zeta^{-1}-2\zeta^{-2})q + \ldots \in J^{!}_{-2,1}, \\
    \GJ(\tau,z) = \phi_{0,1} (\tau,z)
    &= (\zeta + 10 + \zeta^{-1})+(10 \zeta^2-64\zeta+108-64\zeta^{-1}+10\zeta^{-2})q + \ldots \in J^{!}_{0,1},
  \end{align*}
defined by Eichler and Zagier \cite{ez}.
\end{proposition}
\nomenclature[FJ]{$\FJ$}{A weak Jacobi form}
\nomenclature[GJ]{$\GJ$}{A weak Jacobi form}

\subsection{Special endomorphisms}
\label{sec:spec-endom}
\nomenclature[D]{$D$}{Usually an odd fundamental discriminant}
Let $D$ be a negative odd fundamental discriminant and let $\kb=Q(\sqrt{D})$ be the
imaginary quadratic field of discriminant $D$. We write $\OD$ for the ring of integers in $\kb$, $h_\kb$ for the class number of $h_\kb$, and $w_\kb$ for the number of roots of unity in $\kb$. Moreover, we let $\Hkb$ be the Hilbert class field of $\kb$
and $\calO_{\Hkb} \subset \Hkb$ be its ring of integers.
By class field theory, we have an isomorphism via the Artin map
\cite[II, Example 3.3]{silverman-advanced}
\[
  (\cdot, \Hkb/\kb): \Clk \rightarrow \Gal(\Hkb/\kb).
\]
We will use the convention that we write $\sigma(\fraka) = \sigma([\fraka])$
for $([\fraka],\Hkb/\kb)$ for the image of the class of the fractional ideal $\fraka$ under this map.
\nomenclature[sigmaa]{$\sigma(\fraka)$}{The image of $\fraka$ under the Artin map $(\cdot,\Hkb/\kb)$}
\nomenclature[Hcf]{$\Hkb$}{The Hilbert class field of the imaginary quadratic field $\kb$}

For any scheme $S$ we consider pairs $(E,\iota)$, where
$E$ is an elliptic curves over $S$ with complex multiplication $\iota: \OD \hookrightarrow \End(E)$.
The coarse moduli scheme of the corresponding moduli problem is isomorphic to $\Spec \calO_{\Hkb}$ \cite{kry-tiny}.
If $(E,\iota)$ is an elliptic curve with complex multiplication over $S$, we write $\calO_{E} = \End_{S}(E)$
and consider the lattice $L(E,\iota)$ of \emph{special endomorphisms}
\[
  L(E,\iota) = \{x \in \calO_{E}\ \mid\ \iota(\alpha)x
             = x\iota(\bar\alpha) \text{ for all } \alpha \in \OD \text{ and } \tr{x} = 0 \}
\]
as in Definition 5.7 of \cite{kry-tiny}.
It is equipped with the positive definite quadratic form given by $\N(x) = \deg(x) = -x^{2}$. 
For $S = \Spec \C$ or $S = \Spec \bar\F_{p}$ for a prime $p$ that is split in $\kb$, we have that $L(E,\iota)$ is zero.

For non-split primes, $L(E,\iota)$ is a positive definite lattice of
rank $2$ in $\calO_{E}$ and $(E,\iota)$ is supersingular.
In this case $\calO_{E}$ is a maximal order in the quaternion algebra $\B_p$
over $\Q$, which is ramified exactly at $p$ and $\infty$.
The quadratic form $\N(x)$ corresponds to the reduced norm on $\calO_E$.

We fix the embedding of $\kb=\Q(\sqrt{D})$ into $\C$ such that $\sqrt{D}$ has positive imaginary part
and let $\omega = (1+\sqrt{D})/2$ so that $\OD = \Z + \Z\omega$. We let $j_D = j(\omega) = j(\OD)$.
Let $p$ be a rational prime that is not split in $\kb$ and let $\frakp$ be the unique prime ideal of $\kb$ above $p$.
If $\frakP$ is a prime of $\Hkb:=\kb(j_D)$ above $\frakp$, then the image of $j_D$ under the reduction map $\calO_{\Hkb} \to \calO_{\Hkb}/\frakP$
is the $j$-invariant of an elliptic curve $(E_\frakP, \iota_\frakP)$  
with complex multiplication by $\OD$ over $\bar{\F}_p$, unique up to isomorphism.

Fix a fractional ideal $\fraka \subset k$ and let
$\mu \in \different{\kb}^{-1}\fraka/\fraka$ and $m \in \Q_{>0}$.
The following cycles arise from a moduli problem that has been studied in \cite{kry-tiny} and
\cite{bryfaltings} and generalized in \cite{ky-pullback}.
For $m \in \Q$, let $L(E,\iota, m,\fraka,\mu)$ be the set of all $x \in L(E,\iota)\different{\kb}^{-1}\fraka$, such that
\[
   \N(x) = m \N (\fraka), \text{ and } x + \mu \in \calO_{E}\fraka.
\]

For every positive rational number $m$,  we define an Arakelov divisor
\[
  \calZ(m,\fraka,\mu) = \sum_{\frakP \subset \calO_{\Hkb}} \calZ(m,\fraka,\mu)_\frakP \frakP
\]
on $\Spec \calO_{\Hkb}$.
\nomenclature[Zmam]{$\calZ(m,\fraka,\mu)$}{A certain arithmetic divisor on $\Spec \calO_{\Hkb}$}
Here, we let $\calZ(m,\fraka,\mu)_\frakP = 0$ if the rational prime $p$ below $\frakP$
is split in $\kb$ and otherwise
\[
   \calZ(m,\fraka,\mu)_\frakP = \frac{\nu_p(m)}{w_\kb} |L(E_\frakP,\iota_\frakP, m, \fraka,\mu)|
\]
with
\begin{equation*}
  \nu_{p}(m) =
  \begin{cases}
    \frac{1}{2}(\ord_{p}(m)+1), & \text{if $p$ is inert in $\kb$},\\
    \ord_{p}(m \abs{D}), & \text{if $p$ is ramified in $\kb$}.
  \end{cases}
\end{equation*}
If $\calZ(m,\fraka,\mu)$ is non-empty, then we have $m + Q(\mu) = m + \N(\mu)/N(\fraka) \in \Z$.
\nomenclature[ZmamP]{$\calZ(m,\fraka,\mu)_\frakP$}{see $\calZ(m,\fraka,\mu)$ \nomnorefpage}
The representation numbers $|L(E_\frakP,\iota_\frakP, m,\fraka,\mu)|$ can be determined
following \cite{kry-tiny} (completely explicit in the case of a prime discriminant and up to Galois conjugation in general).
We refer to \cite{ehlen-intersection} for details.

\section{CM cycles and CM values of regularized theta lifts} \label{sec:cmcycles}
In this section we fix a rational quadratic space $(V,Q)$ of type $(2,n)$.
We fix a compact open subgroup $K$ of $H(\adeles_f)$ and consider the Shimura variety $X_{K}$ as in Section \ref{sec:shimura-varieties}.

The type of CM cycles we consider are given as follows.
Let $U \subset V$ be a $2$-dimensional, positive definite rational subspace.
This determines a two-point subset $\{z_U^\pm\}\subset \domain$
given by $U(\R)$ with the two possible choices of orientation.
Denote by $V_- = U^\perp \subset V$ the $n$-dimensional negative definite orthogonal complement
of $U$ over $\Q$. Then we have a \emph{rational} splitting
\begin{align}
  \label{split} V = U \oplus V_{-}.
\end{align}
We obtain a cycle $Z(U)_K \subset X_K$, which is called the \emph{CM cycle} in $X_K$ corresponding to $U$.
\nomenclature[U]{$U$}{A two-dimensional positive definite subspace of $V(\Q)$}
It is obtained by embedding a Shimura variety associated with $U$ into $X_K$,
which is given as follows.
Put $T = \GSpin_U$, which we view as a subgroup of $H$ acting trivially on $V_-$.
The group $K_T = K \cap T(\adeles_f)$ is a compact open subgroup of $T(\adeles_f)$.
We obtain a generically injective map
\begin{equation}
  \label{eq:ZU}
  Z(U)_K = T(\Q) \backslash (\{z_U^\pm\} \times T(\adeles_f) / K_T) \hookrightarrow X_K.
\end{equation}
\nomenclature[zU]{$z_{U}^{\pm} = z_U$}{The (two) point(s) corresponding to $U \subset V(\R)$}
\nomenclature[T1]{}{Also $T = \GSpin_U$}

Here, each point is counted with multiplicity
$\frac{2}{w_{K, T}}$, where we let  $w_{K,T}=\abs{(T(\mathbb Q) \cap K_T)}$.
\nomenclature[ZUK1]{$Z(U)_{K}=Z(U)$}{A CM cycle\nomrefeq}
\nomenclature[wKT]{$w_{K,T}$}{$=\abs{(T(\mathbb Q) \cap K_T)}$}

If the choice of $K$ is clear from the context, we will abbreviate $Z(U) = Z(U)_{K}$.
\begin{lemma}\label{lem:ZUCl}
  Suppose that $U$ is isomorphic as a rational quadratic space
  to an imaginary quadratic field $\kb$ and let $\calO_\kb \subset k$
  be its ring of integers.
  If $K_T = \hat{\calO}_{k}^{\times}$, then $Z(U)$ is isomorphic to two
  copies of the ideal class group $\Clk$ of $\kb$,
  that is, $Z(U) \cong \Clk \times \{ z_U^{\pm} \}$.
\end{lemma}
\begin{proof}
  We have $T(\adeles_f) \cong \adeles_{\kb,f}^\times$ (see, for instance \cite{kitqf} or \cite[Section 2.2]{ehlen-binary}) and then the claim follows from \cite[VI. Satz 1.3]{neukirchalgzt}.
\end{proof}

\subsection{The average value}
\label{sec:average-value}
Fix an even lattice $L \subset V$ and we abbreviate $\Phi(z,h,f) := \Phi_L(z,h,f)$.
Schofer \cite{schofer}, Bruinier and Yang \cite{bryfaltings} studied the CM value
\begin{equation}
  \label{eq:cmvaldef}
  \Phi(Z(U),f) = \frac{2}{w_{K,T}} \sum_{(z,h) \in \supp{Z(U)_{K}}} \Phi(z,h,f).
\end{equation}
We review their main results.

The splitting \eqref{split} yields two lattices, $P$ and $N$, defined by
\begin{align*}
  P = L \cap U, \quad N = L \cap V_{-}.
\end{align*}
The direct sum $P \oplus N$ is a sublattice of $L$ of finite index.
\nomenclature[P2]{}{Often obtained as $P = L \cap U$}
\nomenclature[N2]{}{Also a negative definite lattice}

For $z=z_U^\pm$ and $h\in T(\adeles_f)$,
the Siegel theta function $\Theta_{P \oplus N}(\tau,z,h)$ splits as a product
\begin{align}
  \label{splittheta}
  \Theta_{P \oplus N}(\tau,z_U^\pm,h)= \Theta_P(\tau,z_{U}^{\pm},h) \otimes_\C \Theta_N(\tau).
\end{align}
Here, $\Theta_N(\tau)=\Theta_N(\tau,1)$ is the
$S_N$-valued theta function of weight $n/2$ associated to the
negative definite lattice $N$. Note that $v^{-n/2}\overline{\Theta_N(\tau)}$
is the holomorphic theta function corresponding to the positive definite lattice
$N^{-}$. Moreover, we identified $S_{P \oplus N}$ with the tensor product
$S_P \otimes_\C S_N$.

Attached to $P$ there is a so-called incoherent Eisenstein series $\hat E_P(\tau,s)$
of weight 1 transforming with representation $\overline{\rho_{P}} = \rho_{P^{-}}$ \cite{kry-tiny,kryderivfaltings}.
Here, the term ``incoherent'' refers to the fact that it is built
from local data at each place which does not correspond to a
quadratic space over $\Q$.
\nomenclature[EPh]{$\hat E_{P}(\tau,s)$}{An incoherent Eisenstein series attached to $P$}

Its central value at $s=0$ vanishes but it is the value
of the derivative $\frac{\partial}{\partial s}\hat E_P(\tau,s)$ at $s=0$
that carries the arithmetic data which contributes to the CM values.
The function
\begin{equation}
  \label{eq:calE}
  \calE_P(\tau) = \frac{\partial}{\partial s}\hat E_P(\tau,s)\mid_{s=0}
\end{equation}
is a harmonic weak Maa\ss{} form of weight $1$ with respect to $\overline{\rho_{P}}$.
\nomenclature[EPc]{$\calE_{P}(\tau)$}{A harmonic weak Maa\ss{} form, given by $\hat E_{P}'(\tau,0)$}

If $S(q)=\sum_{n\in \Z} a_n q^n$ is a Laurent series in $q$, we write $\CT(S)=a_0$
for the constant term in the $q$-expansion.
\nomenclature[CTq]{$\CT$}{Also the constant term in a Laurent series in $q$}

\begin{theorem} \label{thmBrYCM}
  Let $f \in H_{k,L}$ with $k= 1-n/2$.
  The value of the theta lift $\Phi(z,h,f)$ at the CM cycle $Z(U)_{K}$ is given by
  \begin{align*}
    \Phi(Z(U),f)& = \deg(Z(U)) \left( \CT\left(\langle
        (f_{P \oplus N})^+(\tau),\, \Theta_{N^{-}}(\tau)\otimes \ \calE_P^{+}(\tau)\rangle\right)
      - L'(\xi_k(f), U,0)\right).
  \end{align*}
  Here, $L'(\xi_k(f),U,s)$ is the derivative with respect to $s$ of
  the $L$-function defined by the convolution integral
  \[
    L(\xi_k(f),U,s) =
    \int_{\SL_2(\Z) \bs \uhp} \langle \overline{\xi_k(f)(\tau)}
                        ,  \hat E_{P}(\tau,s) \otimes \Theta_{N^{-}}(\tau) \rangle\, v^{1+n/2}\frac{dudv}{v^{2}}.
  \]
\end{theorem}
\begin{proof}
  This is Theorem 4.7 in \cite{bryfaltings} with a corrected sign. \qedhere
\end{proof}
\nomenclature[rhot]{$\tilde\rho(n,\beta)$}{The coefficient of index $(n,\beta)$ of $E_P(\tau)$}

The proof involves the Siegel-Weil formula and the standard Eisenstein
series associated with $P$, which is defined as
\begin{equation}
  \label{eq:holeis}
  E_{P}(\tau,s)
  = \frac{1}{2} \sum_{\gamma \in \Gamma_{\infty} \bs \Gamma} (\Im(\tau)^s \phi_{0}) \mid_{1,P} \gamma,
\end{equation}
where $\phi_{0} \in S(V(\adeles_f))$ is the characteristic function of $\hat{P}$.
The series converges for $\Re(s)>1$ and has a meromorphic continuation
to the whole complex $s$-plane.
Note that we normalized $E_{P}(\tau,s)$, such that the constant term is equal to one.
\nomenclature[EP]{$E_{P}(\tau)$}{The holomorphic Eisenstein series in $M_{1,P}$, equal to $E_{P}(\tau,0)$}

As in \cite{bryfaltings}, we fix the Haar measure on
$\SO_{U}(\R) \cong \SO(2,\R)$ such that $\vol(\SO_{U}(\R))=1$.
This implies that we have $\vol(\SO_{U}(\Q)\bs\SO_{U}(\adeles_f))=2$.
Moreover, we use the usual Haar measure on $\adeles_{f}^{\times}$.
It satisfies $\vol(\Z_{p}^{\times}) = \vol(\hat\Z^{\times}) = 1$
and $\vol(\Q^{\times}\bs\adeles_f^{\times}) = 1/2$.

For later reference, we write the Fourier expansion of $E_P(\tau,0)$ as
\begin{equation}
  \label{eq:EPFourier}
  E_{P}(\tau) = E_P(\tau,0) = \phi_{0} + \sum_{\beta \in P'/P} \sum_{\substack{n \in \Q_{>0} \\ n \in Q(\beta) + \Z}} \tilde\rho(n,\beta) e(n\tau) \phi_{\beta}.
\end{equation}
A crucial fact is that $\calE_P(\tau)$ maps to $E_{P}(\tau)$ under the $\xi_{1}$-operator.
This has been stated by Bruinier and Yang \cite[Remark 2.4]{bryfaltings}
and follows directly from equation (2.19) in \cite{bryfaltings}.

The Fourier expansion of $\calE_P(\tau)$ can be determined using a very general result
by Kudla and Yang \cite{KudlaYangEisenstein} on the coefficients of Eisenstein series on $\SL_{2}$.
See Proposition 7.2 in \cite{KudlaYangEisenstein} and Schofer \cite{schofer}.

\subsection{The value of $\Phi(z,f)$ at an individual CM point}
\label{sec:cmvalue-single}
We are now interested in computing the value of the theta lift $\Phi(z,f)$ at a CM point (rather than averaging over the cycle).
Let $K_{P} \subset K_{T} \subset T(\adeles_f)$ be a compact open subgroup
such that $K_{P}$ preserves $P$ and acts trivially on $P'/P$. Consider the
Shimura variety
\begin{equation*}
  Z(U)_{P,K} = T(\Q) \bs (\{ z_U^{\pm} \} \times T(\adeles_f) / K_{P}).
\end{equation*}
This is isomorphic to two identical copies of the ``class group''
\begin{equation}
  \label{eq:CPK-eq}
  \CPK = T(\Q) \bs T(\adeles_f) / K_{P}
\end{equation}
and defines a cover of the CM cycle $Z(U)_K$ with $[C_{K} : C_{P,K}]$ branches.
\nomenclature[CPK]{$\CPK$}{A certain cover of the CM cycle}

Since $K_{P}$ acts trivially on $P'/P$, the value $\Theta_P(\tau,h)=\Theta_P(\tau,z_{U}^{\pm},h)$ is well defined for an element $h \in C_{P,K}$. 
As a function of $\tau$, we have $\Theta_P(\tau,h) \in M_{1,P}$.

We would like to apply Stokes' theorem to compute $\Phi(z,h,f)$ for $(z,h) \in Z(U)_{P,K}$.
We use the existence of a preimage $\pre{\Theta}_P(\tau,h)$ under $\xi_1$ of each theta function 
$\Theta_P(\tau,h)$ for $h \in \CPK$, which is guaranteed the exact sequence \eqref{ex-sequ}. However, such a preimage is only unique up
to weakly holomorphic modular forms. Later, we will make specific choices but for now we can work
with an arbirtary preimage for every $h$.

Lemma \ref{lem:xidifform} becomes the following statement in our situation.
\begin{lemma}
  \label{lem:preimdifform}
  Let $\pre{\Theta}_P(\tau,h) \in \calH_{1,P^-}$, such that $\xi_1(\pre{\Theta}_P(\tau,h)) = \Theta_P(\tau,h)$.
  We have the equality of differential forms
  \[
    \bar\partial (\pre{\Theta}_{P}(\tau,h) d\tau) = - v \overline{\Theta_P(\tau,z_{U}^{\pm},h)} d\mu(\tau).
  \]
\end{lemma}
\nomenclature[ThetaP1]{$\pre{\Theta}_{P}(\tau,h)$}{A preimage of $\Theta_{P}(\tau,h)$ under the $\xi_1$ operator}

Note that we can always assume that $L$ splits as $L=P \oplus N$
without loss of generality because we can replace $f$ by $f_{P\oplus N}$, yielding
\begin{equation*}
  \langle f, \Theta_L \rangle =\langle f_{P\oplus N}, \Theta_P \otimes \Theta_N \rangle,
\end{equation*}
since $\Theta_{P\oplus N} = \Theta_P \otimes_\C \Theta_N$, and $\Theta_L = (\Theta_{P \oplus N})^L$ by \cite{bryfaltings},
identifying $S_P \otimes_\C S_N$ with $S_L$.

We express the theta integral in a way that
is convenient for the following calculations.
\begin{lemma}
  \label{lem:philim}
  We have
  \begin{equation*}
    \Phi(z_U^\pm,h,f) = \lim_{T \to \infty}
    \left( \int_{\calF_T}
      \langle f_{P \oplus N}(\tau), \Theta_{N^{-}}(\tau) \otimes \overline{\Theta_P(\tau,z_U^\pm,h)} \rangle v d\mu(\tau)
      - A_0 \log(T) \right),
  \end{equation*}
  where
  \begin{equation*}
    A_{0} = \CT\left( \langle f_{P \oplus N}^+(\tau),\, \Theta_{N^{-}}(\tau) \otimes \phi_{0+P} \rangle \right).
  \end{equation*}
We remark that this expression makes sense even if $(z_U^\pm,h)$ is contained in the support of $Z(f)$
and thus provides a (discontinuous) extension of $\Phi(z,h,f)$ to $Z(U)$ in case that $Z(f) \cap Z(U) \neq \emptyset$.
\end{lemma}
\begin{proof}
  This is Lemma 4.5 of \cite{bryfaltings}.
  Note that a similar statement in the case of signature $(2,0)$
  can be found in Lemma 2.19 of \cite{schofer}. The proof along the
  lines of the proof Proposition 2.5 of \cite{kudla-integrals} is straightforward.
\end{proof}
Using the same techniques as Bruinier and Yang \cite{bryfaltings}, we obtain the following
theorem which is central to all of our applications.
\begin{theorem}\label{thm:value-phiz}
  Let $f \in H_{k,L}$ with $k=1-n/2$.
  Then the value of $\Phi(z,h,f)$ for any $(z,h) \in Z(U)_{P,K}$ is given by
  \begin{align*}
    \Phi(z,h,f) &= \CT \left( \langle (f_{P \oplus N})^+(\tau), \Theta_{N^{-}}(\tau) \otimes \prep{\Theta}_{P}(\tau,h) \rangle \right) \\
    &\quad - \int_{\SL_2(\Z) \bs \uhp}^{\reg} \langle \overline{\xi_k(f_{P \oplus N})}(\tau), \Theta_{N^{-}}(\tau) \otimes \pre{\Theta}_{P}(\tau,h) \rangle v^{1+n/2} d\mu(\tau).
  \end{align*}
  Here, the integral is regularized by taking the limit
  \[
    \lim_{T \rightarrow \infty} \int_{\calF_{T}} \langle  \overline{\xi_k(f_{P \oplus N})}(\tau),\Theta_{N^{-}}(\tau) \otimes \pre{\Theta}_{P}(\tau,h) \rangle v^{1+n/2} d\mu(\tau).
  \]
\end{theorem}

\begin{remark}
  Note that for $f \in M^{!}_{k,L}$ the second summand does not occur since $\xi_k(f)=0$
  in that case. Moreover, we should remark that the regularized integral
  can also be written as
  \begin{align*}
    &\int_{\SL_2(\Z) \bs \uhp}^{\reg} \langle L_{1-n/2}(f_{P \oplus N}),\Theta_{N^{-}}(\tau) \otimes \pre{\Theta}_{P}(\tau,h) \rangle d\mu(\tau) \\
    &= \int_{\SL_2(\Z) \bs \uhp}^{\reg} \langle \overline{\xi_k(f_{P \oplus N})}, \Theta_{N^{-}}(\tau) \otimes \pre{\Theta}_{P}(\tau,h) \rangle v^{1+n/2} d\mu(\tau).
  \end{align*}
\end{remark}

\begin{proof}[Proof of Theorem \ref{thm:value-phiz}]
  Assume again that $L = P \oplus N$.
  According to Lemma \ref{lem:philim}, we write
  \begin{equation} \label{eq:philim}
    \Phi(z,h,f) = \lim_{T \to \infty} \left( I_T(z,h,f) - A_0 \log(T) \right),
  \end{equation}
  where
  \begin{align*}
    I_T(z,h,f) &= \int_{\calF_T}
      \langle f(\tau), \Theta_{N^-}(\tau) \otimes \overline{\Theta_P(\tau,z_U^\pm,h)} \rangle v d\mu(\tau)\\
    &= -\int_{\calF_T}
      \langle f(\tau), \Theta_{N^-}(\tau) \otimes \delbar \pre{\Theta}_{P}(\tau,h) \rangle d\tau \\
    &= - \int_{\calF_T}
         d \left( \langle f(\tau), \Theta_{N^-}(\tau) \otimes \pre{\Theta}_{P}(\tau,h) \rangle d\tau \right) \\
    &\quad
      + \int_{\calF_T} \langle \delbar f(\tau), \Theta_{N^-}(\tau) \otimes \pre{\Theta}_{P}(\tau,h)\rangle  d\tau .
  \end{align*}
  Here, we have used Lemma \ref{lem:preimdifform}.

  For the first integral, we apply Stokes' theorem and obtain
  \begin{align*}
    \int_{\calF_T}
    d \left( \langle f(\tau), \Theta_{N^-}(\tau) \otimes \pre{\Theta}_{P}(\tau,h) \rangle \right) d\tau
    &= \int_{\del \calF_T}
    \langle f(\tau), \Theta_{N^-}(\tau) \otimes \pre{\Theta}_{P}(\tau,h) d\tau \rangle \\
    &= -\int_{iT}^{iT+1} \langle f(\tau), \Theta_{N^-}(\tau) \otimes \pre{\Theta}_{P}(\tau,h) d\tau \rangle,
  \end{align*}
  since the integrand is an $\SL_2(\Z)$-invariant differential form and thus the integral over the
  equivalent pieces of $\del\calF_T$ cancel.
  We split this into three pieces, insert this splitting into \eqref{eq:philim} and regroup to obtain
  \begin{align}
    \Phi(z,h,f) &= \lim_{T \to \infty}
    \int_{iT}^{iT+1}  \langle f^+(\tau), \Theta_{N^-}(\tau) \otimes\pre{\Theta}_{P}^{+}(\tau,h) \rangle d\tau \label{int1}\\
    &\quad + \lim_{T \to \infty}
    \left( \int_{iT}^{iT+1}  \langle f^+(\tau), \Theta_{N^-}(\tau) \otimes \pre{\Theta}_{P}^{-}(\tau,h) \rangle d\tau
      - A_0 \log(T) \right) \label{int2}\\
    &\quad + \lim_{T \to \infty}
      \int_{iT}^{iT+1} \langle f^-(\tau), \Theta_{N^-}(\tau) \otimes \pre{\Theta}_{P}(\tau,h) \rangle  d\tau \label{int3}\\
    &\quad + \lim_{T \to \infty} \int_{\calF_T}
      \langle \delbar f(\tau), \Theta_{N^-}(\tau) \otimes \pre{\Theta}_{P}(\tau,h) \rangle d\tau \label{int4}.
  \end{align}
  Each of the limits above exist.
  
  The limit in \eqref{int3} is equal to zero due to the exponential decay of $f^-(\tau)$.
  We write the Fourier expansion of the integrand as
  \[
    \langle f^{-}(\tau), \Theta_{N^-}(\tau) \otimes \pre{\Theta}_{P}(\tau,h) \rangle  = \sum_{n \in \Z} a(n,v) e(n\tau).
  \]
  and insert this to obtain
  \begin{align*}
    &\int_{0}^{1} \langle f^-(u+iT), \Theta_{N^-}(u+iT) \otimes \pre{\Theta}_{P}(u + iT,h) \rangle du \\
    & = \sum_{n \in \Z} a(n,iT) e(inT) \int_{0}^{1} e^{2\pi i n u} du.
  \end{align*}
  The integral above is equal to $0$ for all $n \in \Z \setminus \{0\}$ and is equal to $1$ for $n = 0$.
  Consequently,
  \begin{align*}
    &\lim_{T \to \infty} \int_{iT}^{iT+1} \langle f^-(\tau), \Theta_{N^-}(\tau) \otimes \pre{\Theta}_{P}(\tau,h) \rangle d\tau \\
    &= \lim_{T \rightarrow \infty} a(0,iT) = \lim_{T \to \infty}\sum_{\mu \in L'/L} \sum_{m \in \Q_{>0}} c_f^{-}(-m,\mu)W_k(-2\pi m T) c_{g}(m,\mu,T),
  \end{align*}
  where $g(\tau) = \Theta_{N^-}(\tau) \otimes \pre{\Theta}_{P}(\tau,h)$.
  Using the standard growth estimates for the Whittaker function and
  the Fourier coefficients of $f$ and $g$, we obtain that there is an $N \in \Z_{>0}$
  and a constant $C>0$, such that for all $m \geq N$, we have
  \[
    c_f^{-}(-m,\mu)W_k(-2\pi m T) c_{g}(m,\mu,T) = O(e^{-mCT}).
  \]
  Thus, for every $T>0$, the constant term in the Fourier expansion of the function
  \[
    \langle f^-(u+iT), \Theta_{N^-}(u +iT) \otimes \pre{\Theta}_{P}(u+iT,h) \rangle
  \]
  can be bounded by
  \[
    \abs{a(0,iT)} \leq c \frac{r(T)}{1-r(T)} \text{ with } r(T) = e^{-CT},
  \]
  where $c,C>0$ are constants.
  Therefore, in the limit $T \rightarrow \infty$, we have
  \[
    \lim_{T \rightarrow \infty} \abs{a(0,iT)} = 0,
  \]
  which finally shows that \eqref{int3} vanishes.

 To show that the limit in \eqref{int2} is equal to zero is analogous.

 The contribution from \eqref{int1} is given by the constant term in the Fourier expansion of
  \[
    \langle f^+_{P \oplus N}(\tau), \Theta_{N^{-}}(\tau) \otimes \prep{\Theta}_{P}(\tau,h) \rangle.
  \]
  
  Finally, by Lemma \ref{lem:xidifform}, we see that \eqref{int4} is equal to
  \begin{align*}
    &\lim_{T \to \infty} \int_{\calF_T}
      \langle \delbar f(\tau), \Theta_{N^-}(\tau) \otimes \pre{\Theta}_{P}(\tau,h) \rangle d\tau\\
    &= - \lim_{T \to \infty} \int_{\calF_T}
        \langle L_{1-n/2}f(\tau), \Theta_{N^-}(\tau) \otimes \pre{\Theta}_{P}(\tau,h) \rangle d\mu(\tau).
  \end{align*}
  This is exactly the definition of the regularized integral in the statement of the theorem.
  We still have to justify that this limit exists.
  However, this now follows from the vanishing of \eqref{int2} and \eqref{int3} and the fact that
  $\Phi(z,h,f)$ is defined at $(z,h)$. That is, we have shown that
  \begin{gather*}
        \lim_{T \to \infty} \int_{\calF_T}
        \langle \overline{\xi_k(f)}, \Theta_{N^-}(\tau) \otimes \pre{\Theta}_{P}(\tau,h) \rangle v^{1+n/2} d\mu(\tau) \\
    = -\Phi(z,h,f) + \CT \left( \langle f^+_{P \oplus N}(\tau), \Theta_{N^{-}}(\tau) \otimes \prep{\Theta}_{P}(\tau,h) \rangle \right)
  \end{gather*}
  and therefore, the limit exists.
\end{proof}

\begin{remark}
  Note that the formula holds for \emph{any} preimage of $\Theta_{P}(\tau,h)$ under $\xi_1$.
\end{remark}

\section{The holomorphic part of $\pre{\Theta}_{P}$}
\label{sec:coeff-holom-part}
An important ingredient for Theorem \ref{thm:value-phiz}
was the existence of a harmonic weak Maa\ss{} form $\prep{\Theta}_{P}(\tau,h)$,
such that $\xi(\prep{\Theta}_{P}(\tau,h)) = \Theta_P(\tau,h)$
for a two-dimensional positive definite lattice $P$.
\nomenclature[P]{$P$}{Usually a two-dimensional positive definite even lattice}
In this section, we will prove our main result, Theorem \ref{thm:pre-pvals}, which
gives detailed information about the Fourier coefficients of the holomorphic
part of appropriately normalized $\pre{\Theta}_P(\tau,h)$ that appear on the right-hand side
of the formula for the CM value in Theorem \ref{thm:value-phiz}.
Theorem \ref{thm:intro-pre-pvals} in the introduction is a special case of this result.

The proof exploits a certain seesaw identity that will allow us to
express the coefficients of the holomorphic part of $\prep{\Theta}_{P}(\tau,h)$
essentially as special values of Borcherds products on modular curves.
We will relate each of these coefficients to a CM value of a meromorphic modular form of weight zero given by a Borcherds product.
Then, we apply the results of \cite{ehlen-intersection} to determine their prime ideal factorization.

\subsection{Embedding into a modular curve}
\label{sec:embedding}
In this section, we basically use the same setup as in Section 7.1 of \cite{bryfaltings},
except that Bruinier and Yang work in signature $(1,2)$.

Let $N$ be a positive integer and consider the congruence
subgroup $\Gamma_0(N) \subset \SL_2(\Z)$, defined by
  \[
  \Gamma_0(N) = \left\lbrace \abcd \in \SL_2(\Z)\ \mid\ c \equiv 0 \bmod{N} \right\rbrace.
  \]
\nomenclature[Ga0N]{$\Gamma_0(N)$}{The congruence subgroup $\Gamma_0(N) \subset \SL_2(\Z)$}
The modular curve $Y_0(N) := \Gamma_0(N) \backslash \uhp$ can be
obtained as a Shimura variety as follows.
\nomenclature[X0N]{$X_0(N)$}{The compactification of $Y_0(N) = \Gamma_0(N) \backslash \uhp$}

  Consider the vector space
  $V: = \{x \in M_2(\Q) \, \mid \, \tr(x)=0\}$ and define the quadratic form
  by $Q(x) = -N\det(x)$. The corresponding bilinear form is $(x,y) = N\tr(xy)$.
  The space $(V,Q)$ has signature $(2,1)$.

  The even part of the Clifford algebra is $C^0(V) = M_2(\Q)$
  and $H = \GSpin_{V} \cong \GL_2$.  The action of
  $\gamma \in H$ on $x \in V$ is given by
  \[
  \gamma.x = \gamma x \gamma\inverse.
  \]
  We have an isomorphism
  $\uhp \cup \overline \uhp \to \domain$ via
  \begin{equation*}
    z = x+iy \mapsto
    \R \Re \zmatrix
    \oplus \R \Im \zmatrix.
  \end{equation*}
  The action of $\gamma \in \GL_2 = \GSpin_V$ is explicitly given by
  \begin{equation*}
    \gamma. \zmatrix =
    \frac{(cz+d)^2}{\det(\gamma)} \begin{pmatrix}
      \gamma z & -(\gamma z)^2 \\ 1 & -\gamma z
    \end{pmatrix},
  \end{equation*}
  where $\gamma z$ is the action via linear fractional transformations on $\uhp \cup \bar\uhp$.
  
  For a prime $p$, let
  \[
  K_p = \{\smallabcd \in \GL_2(\Z_p) \ \big | \  c \in N\Z_p\}
  \]
  and
  \begin{equation*}
    K = \prod_p K_p.
  \end{equation*}
  Then $K$ is a compact open subgroup of the adelic group $H(\adeles_f)$ and
  by strong approximation \cite[Theorem 3.3.1]{bumpautreps}, we have
  $H(\adeles_f) = H(\Q)K$ and $H(\adeles) = H(\Q)H(\R)^+ K$.
This implies that $ X_K \cong \Gamma \backslash \uhp$ where $\Gamma$ is given by
  $\Gamma = H(\Q) \cap H(\R)^+ K  \cong \Gamma_0(N)$.
  The isomorphism is explicitly given by
  \begin{equation}
    \label{Y0shiso}
    Y_0(N) \rightarrow X_K,\ \Gamma_0(N)z \mapsto H(\Q)(z,1)K.
  \end{equation}

In $V$ we have the even lattice
\[
  L = \left\lbrace
    \begin{pmatrix}
      b & -\frac{a}{N} \\ c & -b
    \end{pmatrix} \ \mid\ a,b,c \in \Z
      \right\rbrace.
\]
The dual lattice of $L$ is given by
\[
  L' = \left\lbrace
    \begin{pmatrix}
      \frac{b}{2N} & -\frac{a}{N} \\ c & -\frac{b}{2N}
    \end{pmatrix} \ \mid\ a,b,c \in \Z
    \right\rbrace.
\]
The discriminant group $L'/L$ is cyclic of order $2N$
and we can identify the corresponding finite quadratic module with
the group $\Z/2N\Z$ together with the quadratic form $x^{2}/4N$,
valued in $\frac{1}{4N}\Z/\Z \subset \Q/\Z$. The isomorphism
of finite quadratic modules is explicitly given by
\[
  \Z/2N\Z \rightarrow L'/L,\ \quad r \mapsto \mu_{r} =
    \begin{pmatrix}
      \frac{r}{2N} & 0 \\ 0 & -\frac{r}{2N}
    \end{pmatrix} =: \diag(r/2N,-r/2N).
\]
\nomenclature[diag]{$\diag(a_1,\ldots,a_n)$}{The $n \times n$ diagonal matrix with diagonal $(a_1,\ldots,a_n)$}
\nomenclature[mur]{$\mu_r$}{$= \diag(r/2N,-r/2N)$}

It is easy to check that the group $K$ preserves the lattice $L$ and acts trivially on $L'/L$.

Recall that for $m \in \Q$ and $\mu \in L'/L$, we use the notation
\[
  L_{m,\mu} := \Omega_{m}(\Q) \cap \supp(\phi_{\mu}) = \{x \in \mu + L\, \mid\, Q(x)=m \}.
\]
Let $m \in \Q_{<0}$ and $\mu \in L'/L$ such that $m - Q(\mu) \in \Z$
and let $r \in \Z$ with $\mu \equiv \mu_{r} \bmod{L}$.
Then $D = 4Nm \in \Z$ is a negative discriminant such that $D \equiv r^2 \bmod{4N}$.
Using this notation, we put
\begin{equation}
  \label{xdef}
  \lambda_{r} =
  \begin{pmatrix}
    \frac{r}{2N} & \frac{1}{N} \\
    \frac{D-r^2}{4N} & -\frac{r}{2N}
  \end{pmatrix} \in L_{m,\mu_{r}}.
\end{equation}
The subspace $U = \lambda_r^{\perp} \subset V(\Q)$ is two-dimensional and positive definite.
\nomenclature[lr]{$\lambda_r$}{A certain element in $L_{m,\mu_r}$}

As in section 7.1 of \cite{bryfaltings}, we obtain a positive definite, two-dimensional lattice
\begin{equation}
  \label{LP}
  \calP := L \cap U = \Z
  \begin{pmatrix}
    1 & 0 \\ -r & -1
  \end{pmatrix}
  \oplus \Z
  \begin{pmatrix}
    0 & -\frac{1}{N} \\ \frac{D-r^2}{4N} & 0,
  \end{pmatrix}
\end{equation}
and a negative definite, one-dimensional lattice
\begin{equation}
  \label{LN1}
  \calN := L \cap \Q \lambda_{r} = \Z \frac{2N}{t} \lambda_{r} \text{ with dual } \calN^{'} = \Z \frac{t}{D}\lambda_{r}.
\end{equation}
Here, $t = (r,2N)$.
\nomenclature[PL02]{$\calP$}{A two-dimensional positive definite lattice given by $\calP = L \cap U$}
\nomenclature[N]{$\calN$}{A one-dimensional negative definite lattice given by $\calN = L \cap U^{\perp}$}

From now on, assume that $D < 0$ is a \emph{fundamental} discriminant and let 
$\kb = \kb_{D} = \Q(\sqrt{D})$.
The following lemma is easy to prove and identifies the lattice $\calP$ as an integral ideal
(see Lemma 7.1 in \cite{bryfaltings}).
\begin{lemma}
  \label{lem:Pab-general}
  With the same notation as above, we have an isometry of lattices
  \[
    (\calP,Q) \cong \left(\frakn, \frac{\N(x)}{\N(\frakn)}\right) \text{ with } \frakn = \left(N, \frac{r+\sqrt{D}}{2}\right) \subset \OD.
  \]
It is explicitly given by
\[
\begin{pmatrix}
  1 & 0 \\ -r & 1
\end{pmatrix}
 \mapsto N
\quad \text{and} \quad 
\begin{pmatrix}
  0 & \frac{-1}{N} \\ \frac{D-r^2}{4N} & 0
\end{pmatrix}
\mapsto \frac{r+\sqrt{D}}{2}.
\]
\end{lemma}

We note that the two points $z_{U}^{\pm}$ corresponding to $\lambda_{r}$
satisfy the quadratic equation
\[
  \frac{r^{2}-D}{4}\tau + r\tau + 1 = 0
\]
in terms of coordinates of $\uhp \cup \bar{\uhp}$.
In the following, we will often simply write $z_U$ for any of the two points $z_U^\pm$.

Recall that we write $T = \GSpin_U$ as in Section \ref{sec:cmcycles}
and we consider $T$ as a subgroup of $H = \GSpin_{V}$, acting trivially on $U^{\perp}$.
An explicit calculation shows the following lemma.
\begin{lemma}
  \label{lem:Tkisom}
  We have that $T = \GSpin_U \cong \kb_D^{\times}$ via
  \[
  1 \mapsto \twomat{1}{0}{0}{1}, \quad \sqrt{D} \mapsto \twomat{r}{-2}{\frac{r^2-D}{2}}{-r}.
  \]
  and for $K$ as above, we have
  \[
    K_{T} := K \cap T(\adeles_f) \cong \ODhatt.
  \]
\end{lemma}
Recall the setup from Section \ref{sec:cmvalue-single}.
Note that $K_T$ acts trivially on $\calP'/\calP$
and thus we can take $K_\calP = K_T$ and we have $\CPK = T(\adeles_f)/K_{T}$,
which is isomorphic to $I_\kb/\ODhatt \cong \Clk$ for $\kb = \kb_{D}$.

Now we specialize the setup for our application.
From now on and for the rest of this section we fix
an integral ideal $\fraka \subset \OD$ of the
imaginary quadratic field $\kb = \kb_{D}$ of discriminant $D$.
We assume throughout that $D$ is an odd fundamental discriminant.
Consider the positive definite lattice
$(P,Q) := \left( \fraka, \frac{\N(x)}{\N(\fraka)} \right)$ and
write $\fraka = \left(A, \frac{B + \sqrt{D}}{2} \right)$
with $A,B \in \Z$, $A>0$. That is, the ideal $\fraka$ is generated by
$A$ and $\frac{B + \sqrt{D}}{2}$.
\nomenclature[P1]{}{From Section \ref{sec:embedding} on, $P = \fraka \subset \calO_\kb$ with quadratic form $\N(x)/\N(\fraka)$, corresponding to $[A,B,C]$, s.t. $D=B^2-4AC$}
This is equivalent to saying that $P$ (or $\fraka$)
corresponds to the positive definite integral binary quadratic form
$[A,B,C]$ of discriminant $D = B^2 - 4AC$, with $C \in \Z$
determined by $A,B$ and $D$.
We will use the construction above
to embed the lattice $P$ into the lattice $L$ for $N = A \abs{D}$.
\nomenclature[N1]{$N$}{A positive integer, $N = A\abs{D}$ in Section \ref{sec:coeff-holom-part}}
\nomenclature[A]{$A$}{Starting from Section \ref{sec:embedding}: a prime which does not divide $D$, such that $[A,B,C]$ corresponds to $P = \fraka$ with $D = B^2-4AC$}
\nomenclature[B]{$B$}{Starting from Section \ref{sec:embedding}: a fixed integer, see $A$}
\nomenclature[C]{$C$}{Starting from Section \ref{sec:embedding}: a fixed integer, see $A$}
\nomenclature[D1]{}{From Section \ref{sec:embedding} on: $D=B^2-4AC$}

\begin{assumption}
  Without loss of generality, we will assume that $(A,D)=1$.
  If $A$ is not coprime to $D$, we can replace $[A,B,C]$
  with an equivalent form.
  An integral binary quadratic form of discriminant $D$ represents infinitely many primes
  (cf. Theorem 9.12 of \cite{cox-primes}).
  Thus, we may in fact choose $A$ to be a prime not dividing $D$.
\end{assumption}
Under these assumptions, there are $E,F \in \Z$ with $2AE + BF = 1$.
(Note that $(A,D) = 1$ implies $(2A,B) = 1$ because $D$ is odd.)
Using this, we have for $R := FD$ that
\begin{equation}
  \label{r2}
  R^2 \equiv D \bmod{4A\abs{D}}.
\end{equation}
Indeed, we have $R^2 \equiv 0 \bmod{D}$ and
$F^2D^2 = F^2D \cdot D = F^2(B^2-4AC)D \equiv D \bmod{4A}$.
\nomenclature[R]{$R$}{A fixed integer, $R=FD$ with $R^{2} \equiv D \bmod{4A\abs{D}}$, also see $A$}

\begin{warning}
  Note that the definition of $R$ depends of course on the ideal $\fraka$ we started with.
\end{warning}

With this setup, we put
$M := \frac{-1}{4A} = \frac{D}{4A\abs{D}}$ and let $\lambda_{R}$ as in \eqref{xdef}.
Note that we obtain in this special case
\begin{equation}
  \label{LN}
  \calN = L \cap \Q \lambda_{R} = \Z 2A \lambda_{R} \text{ with dual } \calN^{'} = \Z \lambda_{R}.
\end{equation}
In contrast to the general case, in our situation the lattice $L$
splits as $L = \calP \oplus \calN$.
This follows from the fact that the discriminant group $\calN'/\calN$
is isomorphic to $\Z/2A\Z$ and the following lemma.
\begin{lemma}
  \label{lem:Pab}
  With the same notation as above, we have an isometry of lattices
  \[
  (\calP,Q) \cong \left(\frakb, \frac{\N(x)}{\N(\frakb)}\right)
  \]
  with
  \[
  \frakb = \left(A\abs{D}, \frac{R+\sqrt{D}}{2}\right) = \different{\kb} \fraka.
  \]
  In particular, $[\frakb] = [\fraka] \in \Clk$.
  The isometry is given explicitly in Lemma \ref{lem:Pab-general}.
\end{lemma}
Throughout, we fix the isometry $\fraka \to \frakb$ given by
$x \mapsto \sqrt{D}x$  and $P=\fraka \to \frakb \to \calP$
given by applying $x \mapsto \sqrt{D}x$ and then the isometry of Lemma \ref{lem:Pab}.
We also tacitly identify the theta functions $\Theta_P(\tau,h)$ with $\Theta_\calP(\tau,h)$
using this isometry and write $P$ instead of $\calP$ for simplicity whenever this does not make a difference.

We will now see that the special divisor $Z(M,\mu_{R})$ (Definition \ref{def:Zm})
is given by the CM cycle $Z(U)$. For $D$ coprime to $N$, this was stated as Proposition 7.2 in \cite{bryfaltings}
but it also holds in our special situation.
\begin{proposition}
  With the same notation as above we have
  \[
    Z(U) = Z(M,\mu_{R}).
  \]
\end{proposition}

It is important for us to understand the action of $T(\adeles_f) \cong \adeles_{k,f}^{\times}$
on modular functions precisely. It is given by the following Lemma which can be easily shown using Theorem 6.31 of \cite{shimauto}.
\begin{lemma}
  \label{lem:idelesact}
  Let $h \in \adeles_{k,f}^{\times}$ and let $f \in \Q(X_0(N))$ be a rational
  modular function. Let $t \in T(\adeles_{f})$ be the image of $h$
  and write $t = \gamma k$ for $\gamma \in H(\Q)$ and $k \in K$.
  Let $z \in Z(M,\mu_{R})$ be a CM point.
  Then we have
  \[
    f(z)^{\sigma(h)} = f(\gamma^{-1} z).
  \]
\end{lemma}

\subsection{A seesaw identity}
\label{sec:seesaw-identity}
The motivation for this and the next section is the following.
Let us write the Fourier expansion of $\prep{\Theta}_{P}(\tau,h)$ as
\[
  \prep{\Theta}_{P}(\tau,h) = \sum_{\beta \in P'/P} \sum_{m \gg \infty} c_{P}^{+}(h,m,\beta) e(m\tau) \phi_{\beta}.
\]
\nomenclature[cP]{$c_{P}^+(h,m,\beta)$}{The coefficient of index $(m,\beta)$ of $\prep{\Theta}_{P}(\tau,h)$}

Suppose that there is a weakly holomorphic modular form $f \in M_{1,P}^{!}$,
with principal part
\[
  P_{f} = q^{-m}(\phi_{\beta} + \phi_{-\beta})
\]
for $m>0$.
(Note that such a form often does \emph{not} exist but we will deal
with this problem in the next section.)
On the one hand, we obtain by Theorem \ref{thm:value-phiz} that
\begin{align*}
  \Phi_{P}(h,f) &= \int_{\Gamma \bs \uhp} \langle f(\tau), \overline{\Theta_{P}(\tau,h)} \rangle v d\mu(\tau) \\
               &= \CT \langle f(\tau), \prep{\Theta}_{P}(\tau,h) \rangle
               = c_{P}^{+}(h, m,\beta) + c_{P}^{+}(h, m,-\beta) + \text{``error term''},
\end{align*}
where the ``error term'' is a contribution of the pairing of the principal part of $\pre{\Theta}_{P}$
with the coefficients of positive index of $f$. Let us ignore this term for the moment.

On the other hand, there is a different expression for the
theta lift $\Phi_{P}(h,f)$ in terms of a CM value of the theta lift for the lattice $L$.
The basic principle goes back to Kudla \cite{Kudla-seesaw} who
realized that many previously mysterious identities between theta lifts
can be understood in the context of ``seesaw dual reductive pairs''.

Explicitly, we will use the embedding defined above to obtain an expression
for the regularized theta lift in the case of signature $(2,0)$ as a CM value
of a modular function on $X_0(N)$ for $N = A \abs{D}$.
Consider the theta lift $\Phi_{L}$ corresponding to $L$ defined as
\begin{equation*}
  \Phi_{L}(z,h,g) = \int^{\reg}_{\Gamma \bs \uhp} \langle g ,
                    \overline{\Theta_{L}(\tau,z)} \rangle v^{\frac{1}{2}} d\mu(\tau),
\end{equation*}
for $(z,h) \in X_{K}$.
Since the lattice $L$ splits as $L = \calP \oplus \calN$, the Weil representation
$\rho_{L}$ is isomorphic to the tensor product $\rho_{\calP} \otimes \rho_{\calN}$.
In particular, we have that $M_{1,\calP}^{!} \otimes M_{-\frac{1}{2},\calN}^{!}$
is a subspace of $M_{\frac{1}{2},L}^{!}$.
At a point $(z_{U}^{\pm},h) \in Z(U)$ corresponding to the splitting $\calP \oplus \calN$
we obtain for an element $f \otimes_\C \varphi$ in
$M_{1,\calP}^{!} \otimes_\C M_{-\frac{1}{2},\calN}^{!}$ that
\begin{align*}
  \label{eq:tensor}
  \Phi_{L}(z_U,h,f \otimes \varphi) &=
  \int^{\reg}_{\Gamma \bs \uhp} \langle (f \otimes_\C \varphi) (\tau) ,
  (\overline{\Theta_{\calP} \otimes_\C \Theta_{\calN}})(\tau,h) \rangle v^{\frac{1}{2}} d\mu(\tau)\\
  &= \int^{\reg}_{\Gamma \bs \uhp} \langle f(\tau) , \overline{\Theta_{\calP}(\tau,h)} \rangle
  \langle \varphi(\tau), \Theta_{\calN^{-}}(\tau) \rangle v d\mu(\tau).
\end{align*}

Now we choose a specific function $\varphi$.
Since $\calN'/\calN = \Z/2A\Z$ with quadratic form $-x^{2}/4A$, the space
$M_{k,\calN}^{!}$ is isomorphic to the space of weakly holomorphic Jacobi forms $J_{k+\frac{1}{2},A}^{!}$
of weight $k+\frac{1}{2}$ and index $A$. See also Section \ref{sec:jacobi-forms}.
In particular, $M_{-\frac{1}{2},\calN}^{!}$ is isomorphic to $J_{0,A}^{!}$.
It follows from Proposition \ref{Jevgens},
that $J_{0,A}^{!}$ is generated as a $\C$-vector space by elements of
the form
\begin{equation}
  \label{eq:whjac0}
  \sum_{j=0}^{A} \psi_{j} \FJ^{j}\GJ^{A-j},
\end{equation}
where $\psi_{j} \in M_{2j}^{!}(\SL_{2}(\Z))$ and $\FJ,\GJ$ are the generators of the ring of weak Jacobi forms of
even weight as in Proposition \ref{Jevgens}.
We will frequently identify these forms as vector valued modular forms
in $M^{!}_{-5/2,\calN}$ and $M^{!}_{-1/2,\calN}$
via the theta development of Jacobi forms (see Section \ref{sec:jacobi-forms}).
Under this identification, the following relation is easy to obtain.
\begin{lemma}
  \label{lem:FGT}
  We have $\left\langle \FJ^{j}(\tau)\GJ^{A-j}(\tau),\Theta_{\calN^{-}}(\tau) \right\rangle = 
  \begin{cases}
    12^{A}, & \text{if } j = 0,\\
    0 & \text{otherwise.}
  \end{cases}
$
\end{lemma}

Using the lemma, we obtain that
\begin{equation}
  \label{eq:tensor1}
  \sum_{j=0}^A\langle \psi_{j} \FJ^{j}\GJ^{A-j}(\tau), \Theta_{\calN^{-}}(\tau) \rangle = 12^{A} \psi_{0}(\tau).
\end{equation}
Consequently,
\begin{align}
  \label{eq:seesaw}
  \Phi_{L}(z_U,h,f \otimes \sum_{j=0}^{A} \psi_{j} \FJ^{j}\GJ^{A-j} )
  &= 12^{A} \int^{\reg}_{\Gamma \bs \uhp}
    \psi_0(\tau) \langle f(\tau) , \overline{\Theta_{P}(\tau,h)} \rangle v d\mu(\tau)\\
  &= 12^{A} \Phi_{P}(h,f \cdot \psi_{0}).\notag
\end{align}
This allows us to relate $c_{P}^{+}(h, m,\beta)$
to the special value $\Phi_{L}(z_U,h, g)$ for a weakly holomorphic modular form $g$, which is equal to $\log \abs{\Psi((z_U, h, g)}$ by Borcherds' theorem (Theorem \ref{thm:borcherds}) if the constant term of the input function $g$ vanishes.
In this case, $\Psi(z,g)$ is a meromorphic modular form of weight zero.
By choosing $g$ carefully, we can assure that those special values of $\Psi(z,g)$ lie in the Hilbert class field $\Hkb$ and determine their prime ideal factorization using the results of \cite{ehlen-intersection}.

\begin{remark}
For $D$ a prime discriminant, a similar setup is used in \cite{MarynaInnerProd}.
The main result of \cite{MarynaInnerProd} expresses $\Phi_{P}(1,f)$ essentially as the logarithm of
an algebraic integer given by a special value of a Borcherds product on a modular curve of level $\abs{D}$ and gives 
a formula for the prime ideal decomposition using work of Gross.
We make use of Viazovska's idea but our setup is slightly different and combines some of the ideas of \cite{MarynaInnerProd} 
with those in Section 7 of \cite{bryfaltings}. This allows for square-free discriminants and 
by working adelically, we can keep track of all theta functions $\Theta_P(\tau,h)$ at once 
and determine the action of $\Gal(\Hkb/\kb)$ on the preimages $\pre{\Theta}_P(\tau,h)$.
Moreover, harmonic Maa\ss{} forms are not considered in \cite{MarynaInnerProd}
and to obtain a result that applies to the coefficients $c_P^+(h,m,\beta)$ requires extra work
as we shall see below.
Finally, to establish the relation to the special cycles (without the use of explicit formulas),
and to apply our results to arithmetic geometry,
we use the results of \cite{ehlen-intersection} which build on the arithmetic pullback formula
in \cite{bryfaltings} (and for the explicit formulas, of course on Gross' result, as well).
\end{remark}

\subsection{Weakly holomorphic modular forms of weight $1$}
\label{sec:special-basis}
We will now define a set of weakly holomorphic modular forms $f_{m,\beta}$ in $M_{1,P}^{!}$
serving as a replacement for a form with principal part $q^{-m}(\phi_\beta + \phi_{-\beta})$ in case
such a form does not exist.

Recall that, according to the exact sequence
\[
  \xymatrix@1{ 0 \ar[r] & M_{1,P}^! \ar[r] & H_{1,P} \ar[r]^-{\xi_{1}} & S_{1,P^-} \ar[r] & 0},
\]
the space $S_{1,P^-}$ is the space of obstructions for the existence of a weakly holomorphic modular form with prescribed principal part.
Moreover, note that if $M_{1,P} \neq \{ 0 \}$ only fixing the principal part is not enough to uniquely determine $f_{m,\beta}$,
which is why we have to take this space into account as well.

We write $S_{k,L}(\Q)$ for the space of cusp forms of weight $k$, representation $\rho_L$ 
with only rational Fourier coefficients. By \cite{McGrawBasis}, we have $S_{k,L}(\Q) \otimes_\Q \C = S_{k,L}$.
The following lemma gives a basis for $S_{k,L}(\Q)$ (for any $L$ and $k$) that has a ``simple'' structure,
similar to a $q$-expansion basis starting with increasing powers of $q$ for
scalar valued forms. 
The proof is straightforward.
\begin{lemma}
\label{lem:specialB}
Let $L$ be an even lattice.
Then there is a basis $\{g_{1}(L), \ldots, g_{d(L)}(L)\}$ of $S_{k,L}(\Q)$,
where $d(L) = \dim S_{k,L}$ with the following property.
There are rational numbers $n_{1}(L) \leq \ldots \leq n_{d(L)}(L)$
and elements $\mu_{1}(L),\ldots,\mu_{d(L)}(L) \in L'/L$,
such that we have for their Fourier coefficients
      \[
        c_{g_{j}}(n_{l},\mu_{l}) = \delta_{j,l}.
      \]
\end{lemma}

\begin{definition}
  \label{def:special_basis}
  From now on, we let $p_j:=n_j(P)$ and $\pi_j = \mu_j(P)$ be a set of
  indexes as in Lemma \ref{lem:specialB} for the space $S_{1, P}(\Q)$.
  We also write $d^+ = d(P)$.
  \nomenclature[pj]{$p_j$}{$=n_j(P)$}
  \nomenclature[pij]{$\pi_j$}{$=\mu_j(P)$}
  \nomenclature[dp]{$d^+$}{$= d(P) = \dim(S_{1,P})$}

  Moreover, we fix $n_{j}:=n_j(P^-)$,
  $\beta_{j} := \mu_j(P^-)$, and   $g_j := g_j(P^-)$
  as in Lemma \ref{lem:specialB} for $S_{1,P^{-}}(\Q)$.
  We write
  \[
    g_{j} (\tau) = \sum_{\beta \in P'/P} \sum_{m \in \Q_{>0}} a_j(m,\beta) e(m\tau)\phi_\beta
  \]
  for the Fourier expansion of $g_{j}$ and we write $d^- = d(P^-)$.    
  \nomenclature[nj]{$n_j$}{$=n_j(P^-)$}
  \nomenclature[betaj]{$\beta_j$}{$=\mu_j(P^-)$}
  \nomenclature[dm]{$d^-$}{$= d(P^-) = \dim(S_{1,P^-})$}
\end{definition}

\begin{proposition}
  \label{prop:fbm}
  For $m \in \Q_{>0}$ and $\beta \in P'/P$ with $m + Q(\beta) \in \Z$
  and $(m,\beta) \neq (n_{j}, \beta_{j})$ for all $j \in \{1, \ldots, d^-\}$,
  there is a weakly holomorphic modular form $f_{m,\beta} \in M_{1,P}^{!}$ 
 with only rational Fourier coefficients, having a Fourier expansion of the form
  \begin{equation}
    f_{m,\beta}(\tau) = q^{-m}(\phi_{\beta}+\phi_{-\beta}) - \sum_{j=1}^{d^-} a_{j}(m,\beta) q^{-n_{j}}(\phi_{\beta_{j}} + \phi_{-\beta_{j}}) +O(1)
  \end{equation}
  and with $c_{f_{m,\beta}}(0,0) = 0$ and $c_{f_{m,\beta}}(p_j, \pi_j) = 0$ for all $j \in \{1, \ldots, d^+\}$.
\end{proposition}
\nomenclature[fmb]{$f_{m,\beta}$}{An element of $M_{1,P}^{"!}$}
\begin{proof}
  The existence of $\tilde{f}_{m,\beta} \in H_{1,P}$ with the principal part as above is clear by Proposition 3.11 of \cite{brfugeom}.
  However, we have
  \[
  \{ \tilde{f}_{m,\beta},  g \} = 0
  \]
  for all $g \in S_{1,P^-}$ by construction and Lemma \ref{lem:pairing}, where $\{ \cdot, \cdot \}$ is the pairing in \eqref{eq:pairing}.
  Thus, $\xi_1(\tilde{f}_{m,\beta}) = 0$ and by \eqref{ex-sequ} we have $\tilde{f}_{m,\beta} \in M_{1,P}^!$.
  
  If we multiply any such form $\tilde{f}_{m,\beta}$ 
  with $\Delta(\tau)^{m_{0}}$, where $m_{0}=\max\{m, n_1,\ldots,n_{d^-} \}$, to obtain an element of $M_{1+12m_0,P}$.
  Since this space has a basis of forms with integral Fourier coefficients \cite{McGrawBasis},
  there is an element $g \in M_{1+12m_0,P}(\Q)$ with only rational Fourier coefficients,
  such that $\tilde{f}_{m,\beta} \Delta^{m_0} - g = O(q^{m_0})$.
  Consequently, we have $\tilde{g} = g/\Delta^{m_0} \in M_{1,P}^{!}$ and $\tilde{g}$
  has the correct principal part and only rational Fourier coefficients.

  Finally, we can subtract suitable multiple of $E_P$ to obtain the vanishing of the Fourier coefficient of index $(0,0)$
  and multiples of the special basis elements of $S_{1,P}$ to ensure that $c_{f_{m,\beta}}(p_j, \pi_j) = 0$ for all $j \in \{1, \ldots, d^+\}$.
\end{proof}

\begin{lemma}
\label{lem:Fmb}
  Let $f \in M_{1,P}^!$.
  There exists a weakly holomorphic modular form $\bm{f} \in M_{1/2,L}^!$, 
  such that
  \begin{equation}
    \label{eq:Fmb}
        \langle \bm{f}(\tau), \Theta_{\calN^{-}}(\tau) \rangle = 12^{A} f(\tau)
  \end{equation}
and the constant term $c_{\bm{f}}(0,0)$ of $\bm{f}$ vanishes.
If $f$ has rational Fourier coefficients, the same is true for $\bm{f}$
and if $f$ has only integral Fourier coefficients, then
the Fourier coefficients of $\bm{f}$ are rational with denominators bounded by the
maximum of
\[
  h_\kb \cdot \binom{2A}{A} \text{ and } 12h_\kb + 2w_\kb.
\]
\end{lemma}
\begin{proof}
  For the proof, recall the definition of $\FJ$ and $\GJ$ from Proposition~\ref{Jevgens}.
  Let us take $\bm{f} = f \otimes \GJ^{A}$, the simplest choice.
  Note that $\GJ^A$ has integral Fourier coefficients.
  Then we consider two cases.

  The first case is $A=1$.
  There is a $g \in M_{2}^{!}(\SL_{2}(\Z))$ with only integral Fourier coefficients
  such that we have $g(\tau) = q^{-1} + O(q)$.
  The weakly holomorphic modular form $E_{P} \otimes \FJ g \in M_{1/2,L}^{!}$
  has a non-vanishing constant term
  \[
    c_{E_{P}}(1,0)c_{\FJ}(0,0) + c_\FJ(1,0).
  \]
  In fact, $c_{E_P}(1,0) = w_\kb/h_\kb$ and we have $c_\FJ(0,0) = -2$ and $c_{\FJ}(1,0) = -12$.
  
  Therefore, replacing $\bm{f}$ by
  \[
    \bm{f} + \frac{c_{\bm{f}}(0,0)h_\kb}{12 h_\kb + 2w_\kb} E_{P} \otimes \FJ g.
  \]
  eliminates the constant term and does not change \eqref{eq:Fmb}.

  In the other case, when $A > 1$,
  we can take $g \in M_{2A}(\SL_{2}(\Z))$ with $g(\tau) = 1 + O(q)$ and subtract a multiple of $E_P \otimes \FJ^{A}g$.
  Note that the constant term of $\FJ^A$ is given by the constant term of $(\zeta - 2 +\zeta^{-1})^A$, which
  can be easily seen to be equal to
  \[
    (-1)^A\binom{2A}{A}
  \]
  and is in particular nonzero.
  Thus, we can replace $\bm{f}$ by
  \[
    \bm{f} - \frac{c_{\bm{f}}(0,0)}{(-1)^A\binom{2A}{A}} E_P \otimes \FJ^{A}g
  \]
  to obtain a vanishing constant term without changing \eqref{eq:Fmb}.
\end{proof}

We denote by $F_{m,\beta} := \bm{f_{m,\beta}}$ the weakly holomorphic modular form corresponding to 
$f_{m,\beta}$ via Lemma~\ref{lem:Fmb}.
\nomenclature[Fmb]{$F_{m,\beta}$}{The element in $\bm{f_{m,\beta}} \in M_{1/2,L}^{"!}$ corresponding to $f_{m,\beta}$ via Lemma \ref{lem:Fmb}}

It is easy to see that $f_{m,\beta}$ has Fourier coefficients with bounded denominators
using the result of McGraw \cite{McGrawBasis} that the space of vector valued
modular forms has a basis of modular forms with only integral Fourier coefficients.
\begin{definition}\label{def:cmb}
For each $\beta \in P'/P$ and $m \in \Q_{>0}$ with $m + Q(\beta) \in \Z$ we let
$c_{m,\beta} \in \Z$, such that $c_{m,\beta} F_{m,\beta}$ has only integral Fourier coefficients.
\end{definition}
\nomenclature[cmb]{$c_{m,\beta}$}{An integer defined in Definition \ref{def:cmb}}

We have that
\begin{equation}
  \label{eq:coeff1}
  12^{A} c_{m,\beta}\, \int^{\reg}_{\Gamma \bs \uhp} \langle f_{m,\beta}(\tau) , \overline{\Theta_{P}(\tau,h)} \rangle v\, d\mu(\tau)
    =  \Phi_{L}(z_U, h, c_{m,\beta} F_{m,\beta}).
\end{equation}
By the theorem of Borcherds (Theorem \ref{thm:borcherds}),
the value on the right hand side is essentially the logarithm of a special value of
a rational function on $Y_{0}(N)$ (or rather on its compactification $X_0(N)$) which is defined over $\Q$, as long as the coefficients of the input function (in our case given by $c_{m,\beta} f_{m,\beta}$) are all integers.
Moreover,
\begin{equation}
  \label{eq:coeff2}
  \int^{\reg}_{\Gamma \bs \uhp} \langle f_{m,\beta}(\tau) , \overline{\Theta_{P}(\tau,h)} \rangle v\, d\mu(\tau)
  =  - \frac{2}{12^{A}\,c_{m,\beta}} \log \abs{\Psi_{L}(z_U, h, c_{m,\beta} F_{m,\beta})}^{2}.
\end{equation}
By CM theory (Lemma \ref{lem:Psirat-idelesact}),
an integral power of the value $\Psi_{L}(z_U, h, c_{m,\beta} F_{m,\beta})$
on the right is contained in the Hilbert class field $\Hkb$ of $\kb$.
On the left hand side, we essentially obtain the coefficient $c_{P}^{+}(h,m,\beta)$ we are interested in
and some ``error terms''.

\subsection{The arithmetic pullback and the coefficients of the holomorphic part}
\label{sec:arith-pullback}
In this section we will determine the prime ideal factorization of the algebraic number
given by the special value $\Psi_{L}(z_U, h, c_{m,\beta} F_{m,\beta})$.

Our basic setup is the following.
Given $P \cong \fraka$, we let $\frakb = \different{\kb} \fraka$
as in Lemma \ref{lem:Pab} and put $N=A\abs{D}$,
where $(P,Q)$ corresponds to the integral binary quadratic form $[A,B,C]$.
As before, we assume that $(A,D)=1$.
Moreover, we let $M,R$ be as defined on page~\pageref{r2}.

\begin{definition}
  For a harmonic weak Maa\ss{} form $f \in H_{\frac{1}{2},L}(\Z)$, we write
  \[
    \calZ(f) = \sum_{\mu \in L'/L} \sum_{m<0} c_f^{+}(m,\mu) \calZ(m,\mu)
  \]
  for the divisor associated with $f$ on $\calY_{0}(N)$.
  Here, $\calZ(m, \mu)$ is the extension of the Heegner divisor
  $Z(m,\mu)$ to the stack $\calX_0(N)$ as in \cite{bryfaltings} and \cite{ehlen-intersection}.
  \nomenclature[ZmmuN]{$\calZ(m,\mu)$}{A special divisor on $\calX_{0}(N)$}
  \nomenclature[ZfN]{$\calZ(f)$}{A special divisor on $\calX_{0}(N)$}
  \nomenclature[X0N1]{$\calX_0(N)$}{A certain integral model of $X_0(N)$}
  It is equal to the flat closure of $Z(f)$, defined analogously.
\end{definition}
For $f \in H_{\frac{1}{2},L}$ with integral principal part 
the pair $\widehat{\calZ}^{c}(f)=(\calZ^{c}(f), \Phi_{L}(\cdot,g))$
defines an arithmetic divisor on $\calX_{0}(N)$.
Here, $\calZ^{c}(f) = \calZ(f) + C(f)$ is a suitable extension of $\calZ(f)$ to $\calX_{0}(N)$
where $C(f)$ is supported at the cusps.

\begin{lemma}
  Let $f \in M_{1/2,L}^{!}$ with constant coefficient $c_f(0,0) = 0$ and 
   $c_f(m,\mu) \in \Q$ for all $m \in \Q$ and $\mu \in L'/L$.
  Then there exists an integer $M_{f}$, such that the Borcherds product
  $\Psi_{L}(z,h,M_{f} \cdot f)$ defines a meromorphic modular function contained in $\Q(j,j_{N})$.
\end{lemma}
\begin{proof}
  See Lemma 8.1 of \cite{ehlen-intersection}.
\end{proof}
The lemma implies that
the arithmetic divisor $(\divisor(\Psi_{L}(\cdot,M_{f}f)), -\log\abs{\Psi(\cdot,M_{f}f)}^{2})$
associated with the Borcherds lift of $f$ is principal.
\begin{lemma}
  \label{lem:Psirat-idelesact}
  The CM value $\Psi_{L}(z_{U}^{\pm}, h, M_{f} \cdot f)$ 
  is contained in the Hilbert class field $\Hkb$ of $\kb$ for every $(z_{U}^{\pm},h) \in Z(U)$
  and we have
  \[
    \Psi_{L}(z_{U}^{\pm},h, M_f \cdot f) = \Psi_L^{\sigma(h)}(z_{U}^{\pm},1, M_f \cdot f).
  \]
\end{lemma}
\begin{proof}
  If we write $h = \gamma k \in H(\Q)K$, then we have according to Lemma \ref{lem:idelesact} that
  \[
    \Psi_{L}(z_{U}^{\pm},h,f) = \Psi_{L}(\gamma^{-1}z_{U}^{\pm},1,f) = \Psi_{L}^{\sigma(h)}(z_{U}^{\pm},1,f). \qedhere
  \]
\end{proof}

The following Theorem shows that $\divisor(\Psi_L(z,h,f))$ is a horizontal divisor.
\begin{theorem}
  \label{thm:bopintegral}
  Let $f \in M_{1/2,L}^{!}$ be a weakly holomorphic modular form
  with only integral Fourier coefficients and assume that $N$ is square-free.
  Suppose that the multiplier system of $\Psi_L(z,h,f)$ is trivial.
  Then the divisor of the rational function defined by $\Psi_L(z,h,f)$ on $\calY_{0}(N)$ is equal to $\calZ(f)$,
  the flat closure of $Z(f)$.
\end{theorem}
\begin{proof}
  See Theorem 8.2 of \cite{ehlen-intersection}.
\end{proof}

The following Proposition is crucial for the proof of the main
theorem in this section. It shows that we do not have to deal with bad intersection
in order to obtain arithmetic information about the coefficients of $\prep{\Theta}_{P}(\tau,h)$.
\begin{proposition}
  \label{prop:properint}
  Let $f \in M_{1,P}^{!}$ with $c_f(0,0) = 0$
  and $H \in M_{1/2,L}^{!}$, both with integral principal parts, such that
  \[
    \langle H(\tau), \phi_{0+P} \otimes \Theta_{\calN^{-}}(\tau) \rangle = \langle f(\tau), \phi_{0 + P} \rangle.
  \]
  Then $\calZ(H)$ and $\calZ(M,\mu_{R})$ (see page \pageref{r2}) intersect properly.
\end{proposition}
\begin{proof}
  The divisors $\calZ(M,\mu_R)$ and $\calZ(n,\nu)$ do not intersect properly if 
  they have complex points in common, which can only occur of
  $Dd$ is a perfect square for $D = 4NM$ and $d = 4Nn$.

  This means that improper intersection might occur here if there is a coefficient
  $c_{H}(m,\mu)$ with $4Nm = Dn^{2}$ for some $n \in \Z$ and $\mu = n\cdot\mu_R$
  and $\calZ(m,\mu)(\C) \cap \calZ(M,\mu_R)(\C)$
  is given by points coming from non-primitive elements $n\lambda \in L_{m,\mu}$ with $\lambda \in L_{M,\mu_R}$.
  Thus, we obtain
  \[
     \calZ(H)(\C) \cap \calZ(M,\mu_R)(\C) 
     = \sum_{n \in \Z} c_H\left(-\frac{n^2}{4A}, n \cdot \mu_R \right) \calZ(M,\mu_R)(\C)
  \]
  and we need to show that the sum on the right-hand side is zero.  
  On the one hand, we have by assumption 
  $\CT(\langle H(\tau), \phi_{0} \otimes \Theta_{\calN^{-}}(\tau) \rangle) = c_f(0,0) = 0$.
  On the other hand, recalling that $\calN'$ is spanned by $\lambda_R \equiv \mu_R \bmod{L}$, we obtain
  \begin{equation}
    \label{eq:10}
    \CT(\langle H(\tau), \phi_{0+P} \otimes \Theta_{\calN^{-}}(\tau) \rangle)
    = \sum_{n \in \Z} c_{H}\left( \frac{-n^{2}}{4A}, n\cdot\mu_R \right),
  \end{equation}
  by the definition of the theta function.

  Therefore, the total multiplicity of improper intersection is zero.
\end{proof}

\begin{remark}
  An alternative approach to avoid improper intersection would
  be to show a moving lemma, i.e.
  that it is always possible to subtract a weakly holomorphic modular
  form $H \in M_{1/2,L}^{!}$, such that $\tilde{F}_{m,\beta} = F_{m,\beta} - H$
  satisfies
  \[
    c_{\tilde{F}_{m,\beta}}\left(\frac{-n^2}{4A}, n \cdot \mu_R \right) = 0
  \]
  for all $n \in \Z$. This is in fact possible but tedious
  and, as the proposition shows, completely unnecessary for our applications.
\end{remark}

The following propositon expresses the prime ideal factorization of CM values of Borcherds products
that appear in the seesaw identity \eqref{eq:seesaw} in terms of the special cycles $\Spec \calO_{\Hkb}$ defined in Section \ref{sec:spec-endom}.
\begin{proposition} \label{prop:pullback-ord}
  Let $f \in M_{1,P}^{!}$ and $g \in M_{-\frac{1}{2},\calN}$ with
  \[
    g = \sum_{j=0}^{A}\psi_{j}\FJ^{j}\GJ^{A-j},
  \]
  where $\psi_{j} \in M_{2j}^{!}(\SL_2(\Z))$ for every $j$.
  Suppose that $f \otimes g$ has only integral Fourier coefficients,
  $c_{f\otimes g}(0,0) = 0$,
  and the multiplier system
  of $\Psi_{L}((z,h), f \otimes g)$ is trivial. Then we have
  \begin{equation}
  \label{eq:pord1}
  \ord_{\mathfrak{P}}(\Psi_{L}(z_U, h, f \otimes g))
  = 12^{A}\frac{w_\kb}{2}\sum_{\beta \in P'/P} \sum_{m > 0} c_{f\psi_{0}}(-m,\beta) \calZ(m,h.\fraka,h.\beta)_{\mathfrak{P}}
\end{equation}
for every prime ideal $\frakP$ of the Hilbert class field $\Hkb$.
\end{proposition}
\begin{proof}
  By Lemma \ref{lem:FGT}, we have
  \[
    \sum_{\beta \in P'/P} \langle f \otimes g, \phi_{\beta}\otimes \Theta_{\calN^{-}}\rangle \phi_{\beta}
      = 12^{A} f \psi_0.
  \]
  Therefore,
  \[
    12^{A}c_{f\psi_{0}}(-m,\beta) = \sum_{n \in \Z} c_{f \otimes g}\left( -m - \frac{n^{2}}{4A}, \beta + \mu_{\tilde{n}} \right).
  \]
  Here $\tilde{n} = -Dn$ so that $\mu_{\tilde{n}}$ corresponds to $n \bmod{2A}$.

  By Theorem 7.4 of \cite{ehlen-intersection}, we have
  \begin{align*}
    \ord_{\mathfrak{P}}(\Psi_{L}(z_U, 1, f \otimes g))
      &= \sum_{\beta \in P'/P}\sum_{\nu \in N'/N} \sum_{m_{1} <0} c_{f \otimes g}(m_{1}, \beta + \nu) \\
          &\phantom{=}\times \sum_{\substack{n \equiv F r_{1} \bmod{2A} \\ n^{2} \leq \abs{D_{1}/D}}}
          \calZ\left( \abs{m_{1}} - \frac{n^{2}}{4A}, \fraka, r_{1} \left(\frac{1+F\sqrt{D}}{2\sqrt{D}}\right) \right)_\frakP.
  \end{align*}
  Here, $r_1$ corresponds to $\mu_{r_{1}} = \beta + \nu$.
  Note that 
  \[
     r_{1} \left(\frac{1+F\sqrt{D}}{2\sqrt{D}}\right) \in \diffinv{k}\fraka/\fraka
  \]
  does not depend on $\nu$. Moreover, 
  the generator $A/\sqrt{D}$ of $\different{\kb}^{-1}\fraka$ maps to
  $2A \bmod{2N}$ using the isometry in Lemma \ref{lem:Pab} and our identification of $r \bmod{2N}$ with $\mu_{r}$.
  Consequently,
  \[
    2A \frac{1+F\sqrt{D}}{2 \sqrt{D}} = \frac{A}{\sqrt{D}} + AF \equiv \frac{A}{\sqrt{D}} \bmod{\fraka}.
  \]
  Thus, we obtain 
  \begin{align*}
    \ord_{\mathfrak{P}}(\Psi_{L}(z_U, 1, f \otimes g)) &=
      \sum_{\beta \in P'/P}
        \sum_{m>0}
         \sum_{n \in \Z}
            c_{f \otimes g}\left(-m - \frac{n^{2}}{4A},\beta+\mu_{\tilde{n}}\right)
            \calZ\left(m, \fraka, \beta\right) \\
          &= 12^{A} \sum_{m > 0} \sum_{\beta \in P'/P}c_{f\psi_{0}}(-m,\beta) \calZ\left( m, \fraka, \beta \right)_\frakP.
  \end{align*}
The result for $h \neq 1$ follows from the action of the Galois group $\Gal(\Hkb/\kb)$
described in Lemma \ref{lem:idelesact} above for the left-hand side and Proposition 3.5 of \cite{ehlen-intersection} for the cycles (note that the numbers in the published version have been accidentally shifted by the publisher and Proposition 3.5 is contained in Section 5 op. cit.)
\end{proof}

We are now in a position to state our main result.
\begin{theorem}
  \label{thm:pre-pvals}
  Assume that the lattice $P$ is given by an integral ideal $\fraka \subset \calO_\kb$
  with quadratic form $Q(x) = \N(x)/\N(\fraka)$.
  
  For every $h \in \CPK \cong \Clk$ there is a harmonic weak Maa\ss{} form
  $\pre{\Theta}_P(\tau,h) \in \calH_{1,P^{-}}$ with holomorphic part
  \[
    \prep{\Theta}_P(\tau,h) = \sum_{\beta \in P'/P} \sum_{m \gg -\infty} c_{P}^{+}(h,m,\beta) e(m \tau)  \phi_{\beta}
  \]
  satisfying the following properties:
  \begin{enumerate}
  \item We have $\xi(\pre{\Theta}_P(\tau,h)) = \Theta_P(\tau,h)$
        and
    \begin{equation}
      \label{eq:11}
       \pre{E}_P(\tau) := \frac{1}{h_{k}}\sum_{h \in \CPK}\pre{\Theta}_P(\tau,h)
    \end{equation}
    satisfies $\xi_1(\pre{E}_{P}(\tau)) = E_P(\tau)$ and the principal part of
    $\pre{E}_{P}(\tau)$ vanishes.
\nomenclature[EP1]{$\pre{E}_{P}(\tau)$}{A preimage of $E_{P}(\tau)$ under the $\xi_1$-operator}
  \item For all $\beta \in P'/P$ and all $m \in \Q$, $m \neq 0$
    with $m \equiv -Q(\beta) \bmod{\Z}$, we have
    \begin{equation}
      \label{eq:cprma}
      c_{P}^{+}(h,m,\beta) = - \frac{2}{r} \log\abs{\alpha(h,m,\beta)},
    \end{equation}
    where $\alpha(h,m,\beta) \in \calO_{\Hkb}$ and $r \in \Z_{>0}$ only depends on $D$.
  \item Moreover, if $m > 0$, then
    \[
      \ord_{\mathfrak{P}}(\alpha(h,m,\beta)) = r \cdot w_{\kb}\cdot \calZ(m,h.\fraka,h.\beta)_{\mathfrak{P}}
    \]
    for all prime ideals $\frakP \subset \calO_{\Hkb}$.
  \item For $m < 0$, we have
    $\alpha(h,m,\beta) \in \calO_{\Hkb}^{\times}$ and $c_P^+(h,m,\beta)=0$ if $m < -p_{d^+}$ \\ (see Definition \ref{def:special_basis}).
  \end{enumerate}
\end{theorem}

\subsection{Proof of Theorem \ref{thm:pre-pvals}}
\label{sec:proof-main}
The proof is structured as follows.
In the next subsection, we will prove \emph{(i)}, \emph{(ii)} for $m<0$ and \emph{(iv)}.
After that, in Section \ref{sec:coeff-posit-index} we will turn to the coefficients of index $m>0$ and show \emph{(iii)}.
We will first show that \emph{(iii)} holds with integers $r(m,\beta)$ in place of $r$ and finally
show in Section \ref{sec:independence}, that we can choose $r(m,\beta)$ to not depend on $m$ and $\beta$.

\subsubsection{The principal part and compatibility with the Siegel-Weil formula}
\label{sec:princ-parts-preim}
Since $P$ is positive definite and $P'/P$ is anisotropic, we have the decomposition
\begin{equation}
  \label{eq:thPdec}
  \Theta_P(\tau,h) = E_P(\tau,0) + g_{P}(\tau,h),
\end{equation}
where for each $h \in \CPK \cong \Clk$ the form $g_{P}(\tau,h) \in S_{1,P}$ is a cusp form of weight $1$
which has rational Fourier coefficients.

Let $h_1, \ldots, h_{d^+} \in S_{1,P}$ be a basis of $S_{1,P}$ as constructed in Section \ref{sec:special-basis}.
There are harmonic weak Maa\ss{} forms
$H_1, \ldots, H_{d^+} \in H_{1,P^-}$ with $\xi_1(H_{i}) = h_{i}$ and such that the principal part of $H_i$ is of the form
\[
    P_{H_i}(\tau) = \frac{1}{2}\sum_{j=1}^{d^+} q^{-p_j} (\frake_{\pi_j} + \frake_{-\pi_j})(h_i, h_j).
\]
Indeed, if we define $H_i$ to have this principal part, then the pairing defined in Lemma \ref{lem:pairing} yields 
$(\xi_1(H_i), h_j) = \{H_i, h_j\} = (h_i, h_j)$ for all $j$ and thus $\xi_1(H_i) = h_i$ by the non-degeneracy of
the Petersson inner product on cusp forms.
We define coefficients $a_{i}(h) \in \Q$ by $\sum_{i=1}^{d^+} a_i(h) h_{i} = g_{P}(\tau,h)$.
Using the same set of coefficients, we put
\begin{equation}
  \label{eq:pregdef}
  \pre{g}_{P}(\tau,h) := \sum_{i = 1}^{d^+}a_i(h) H_{i}
\end{equation}
\begin{remark}
After our normalization and fixing a basis of $S_{1,P}$ as in Section \ref{sec:special-basis}, 
the forms $\pre{g}_P(\tau,h)$ are uniquely determined up to holomorphic modular forms in $M_{1,P^-} = S_{1,P^-}$.

Moreover, it is possible to show that the orthogonal complement of the space of
theta functions in $M_{1,P}$ has a basis with rational Fourier coefficients.
This implies that the Petersson inner products in the principal part will in fact
only involve inner products of cusp forms coming from theta functions.
In \cite{ehlen-binary}, we give an explicit formula
for these Petersson inner products in terms of CM values of $\log\abs{v^{1/2}\eta^2(\tau)}$.
\end{remark}

  In order to ensure \eqref{eq:11}, we will first just pick a preimage $\pre{E}_P(\tau)$ of $E_P(\tau)$ to \emph{define} a preimage of $\Theta_P(\tau,h)$ by 
  $\pre{\Theta}_P(\tau,h) := \pre{E}_P(\tau) + \pre{g}_P(\tau,h)$.
  First note that such a preimage exists by the surjectivity of $\xi_1$, see \eqref{ex-sequ}. 
  It is convenient to pick a preimage that has a vanishing principal part.
  To see that this is possible, note that $\{\pre{E}_{P}, g\} = (E_P, g) = 0$ for all $g \in S_{1,P}$. 
  Thus, using the fact that there is a harmonic Maa\ss{} form in $f \in H_{1,P^-}$, such that
  $\pre{E}_P - f$ has vanishing principal part
  (Proposition 3.11 of \cite{brfugeom}), we see that $0 = \{\pre{E}_{P}, g\} = \{f,g\}$.
  By the exactness of the sequence \eqref{ex-sequS} we obtain 
  $f \in M_{1,P}^!$, and thus $\xi_1(\pre{E}_P - f) = E_P$. Later, in Lemma \ref{lem:pos-index-1}, we will pick a specific $\pre{E}_P$.
  The following proposition summarizes our normalization.
\begin{proposition} \label{SumG=E}
  Let $\pre{E}_{P} \in H_{1,P^-}$ with $\xi_1(\pre{E}_{P}) = E_{P}$ and such that the principal part of $\pre{E}_P(\tau)$
  vanishes.
  For $h \in \CPK$ define
  \[
    \pre{\Theta}_{P}(\tau,h) := \pre{E}_P(\tau) + \pre{g}_P(\tau,h),
  \]
  where $\pre{g}_P(\tau,h)$ has been defined in \eqref{eq:pregdef}.
  Then $\xi_1 (\pre{\Theta}_{P})(\tau,h) = \Theta_P(\tau,h)$ and we have
  \[
    \frac{1}{h_\kb}\sum_{h \in \CPK} \pre{\Theta}_{P}(\tau,h)  = \pre{E}_P(\tau).
  \]
\end{proposition}

 For the proof, we quote the following Lemma\footnote{Note that the
factor $2/w_{P}$ is missing in \cite{schofer}.} of Schofer \cite[Lemma 2.13]{schofer}.
\begin{lemma}
  \label{lem:schoferSumInt}
  Let $B(h)$ be a function on $T(\adeles_f)$ depending only on the image
  of $h$ in $\SO_{U}(\adeles_f)$. 
Assume that $B$ is invariant under $K_T$ and $T(\Q)$. 
Then
  \[
  2 \frac{\vol(K_P)}{w_{T}} \sum_{h \in \CPK} B(h)
  = \int_{\SO_{U}(\Q) \bs \SO_{U}(\adeles_f)} B(h) dh,
  \]
where $w_T = \abs{T(\Q) \cap K_T}$.
\end{lemma}
\begin{proof}[Proof of Proposition \ref{SumG=E}]
  Setting $B(h) = \Theta_P(\tau,(z_U,h))$ in Lemma \ref{lem:schoferSumInt}
  we get
  \[
  \sum_{h \in \CPK} \Theta_P(\tau,(z_U,h))
  = \frac{w_{P}}{2\vol(K_P)}
  \int_{\SO_{U}(\Q) \bs \SO_{U}(\adeles_f)} \Theta_P(\tau,(z_U,h)) dh.
  \]
  The latter integral is equal to $2E_P(\tau,0)$ by the Siegel-Weil formula
  (Theorem 2.1 of \cite{bryfaltings}).
  Therefore, since $\Theta_P(\tau,z_{U}^{\pm},h) = E_P(\tau,0) + g_{P}(\tau,h)$,
  we have indeed
  \begin{equation*}
    \sum_{h \in \CPK} g_{P}(\tau,h) = 0.
  \end{equation*}
  Consequently,
  \begin{equation*}
    \label{eq:2}
    \sum_{h \in \CPK} a_{i}(h) = 0
  \end{equation*}
  for all $i$, since the $h_{i}$ are linearly independent.
  It follows that
  \begin{align*}
    \label{eq:3}
    \xi_1 (\pre{g}_P(\tau,h)) &=  g_{P}(\tau,h)\ \text{from the definition of }  \pre{g}_P(\tau,h) \text{ and}\\
    \sum_{h \in \CPK} \pre{g}_P(\tau,h) &= 0.
  \end{align*}
  Thus, we obtain
  \[
    \sum_{h \in \CPK} \pre{\Theta}_{P}(\tau,h) = \abs{C_{P,K}} \cdot \pre{E}_P(\tau).
  \]
  Using Lemma \ref{lem:schoferSumInt} with $B(h)=1$
  shows that the factor $\frac{\vol(K_T)}{w_{T}} = \abs{\CPK} = h_\kb$ in our case.
\end{proof}

\begin{lemma}
\label{lem:principal-part}
  For $\pre{\Theta}_{P}(\tau,h)$ as in Proposition \ref{SumG=E} and $m<0$, we have
\[
  c_P^+(h,m,\beta) = -2 r(m,\beta)^{-1} \log\abs{\alpha(h,m,\beta)}
\]
with $\alpha(m,\beta) \in \calO_{\Hkb}^\times$ and $r(m,\beta) \in \Z_{>0}$.
\end{lemma}
\begin{proof}
  By construction every coefficient in the principal part
  of $\pre{\Theta}_{P}(\tau,h)$ is of the form
  \[
    (g,g_{P}(\tau,h)) = (g,\Theta_{P}(\tau,h)),
  \]
  for some $g \in S_{1,P}(\Q)$, a cusp form with rational Fourier coefficients.
  Therefore, fixing $h \in \CPK$, it is enough to show that
  \begin{equation}
    \label{eq:rla1}
    (g, \Theta_P(\tau,h)) = -\frac{2}{c} \log\abs{\alpha},
  \end{equation}
  with $c \in \Z_{>0}$ and $\alpha \in \calO_{\Hkb}^{\times}$.
  This follows easily from the fact that
  \[
    (g, \Theta_P(\tau,h)) = \Phi_{P}(h,g)
  \]
 and that this quantity is equal to
  \[
     12^{A} \Phi_{P}(h,g) = - 2 \log \abs{\Psi_{L}(z_U, h, \bm{g})}^{2}
  \]
  by \eqref{eq:seesaw}.
  Thus, replacing $g$ by $s \cdot g$ for an appropriate positive integer $s$ if necessary, we have by \eqref{eq:pord1} that
  \[
    \alpha := \Psi_{L}(z_U, h, s \cdot \bm{g})^{2} \in \calO_{\Hkb}^{\times}.
  \]
  Thus, we obtain \eqref{eq:rla1} with $c = 12^{A}s$ and this finishes the proof.
\end{proof}

\subsubsection{The coefficients of positive index}
\label{sec:coeff-posit-index}

\begin{lemma}
  \label{lem:pos-index-1}
  We can choose the functions $\pre{\Theta}_P(\tau,h)$, such that:
  \begin{enumerate}
  \item Proposition \ref{SumG=E} holds,
  \item for $j \in \{1,\ldots,d^-\}$, we have
  \[
    c_{P}^{+}(h,n_j,\beta_j) = -\frac{2}{r_0} \log\abs{\alpha(h,n_j,\beta_j)}
  \]
  with
  \[
    \ord_{\mathfrak{P}}(\alpha(h,n_j, \beta_j))
        = r_0 \cdot w_{\kb} \cdot \calZ(n_{j},h.\fraka,h.\beta_j)_{\mathfrak{P}}.
  \]
  \end{enumerate}
  Here, the indices $n_j, \beta_j$ for $j=1,\ldots,d^-$ belong to our special basis
  of $S_{1,P^-}$ as in Lemma \ref{lem:specialB} and $r_0$ does not depend on $j$.
\end{lemma}
\begin{proof}
  We choose $r_0 \in \Z_{>0}$ minimal such that for all primes $\frakP$ of $\calO_{\Hkb}$
  with $\frakP \mid p$ for any $p \in \bigcup_{j=1}^{d} \Diff(n_{j})$
  there is an element $\alpha(1,n_j,\beta_j) \in \calO_{\Hkb}$ with
  \[
    \ord_{\frakP}(\alpha(1,n_j,\beta_j)) = r_0 \cdot w_{\kb}\cdot \calZ(n_j,\fraka,\beta_j)_{\frakP}
  \]
  for all $\frakP$. 
  Note that $r_0$ is a divisor of $h_{\Hkb}$, the class number of $\Hkb$.
  Using this, we let
  \[
    \alpha(h,n_j,\beta_j) =  \alpha(1,n_j,\beta_j)^{\sigma(h)},
  \]
  where $\sigma(h)$ corresponds to the class of the idele $h$ under the Artin map.
  
  Replacing $\pre{\Theta}_P(\tau,h)$ by
  \[
    \pre{\Theta}_P(\tau,h) -
    \sum_{j=1}^{d^-}\left(\frac{2}{r_0}\log\abs{\alpha(h,n_{j},\beta_{j})}
    +c_P^+(h,n_j,\beta_j)\right)g_{j}(\tau)
  \]
  and accordingly $\pre{E}_P(\tau)$ by
  \[
    \pre{E}_P(\tau) - \sum_{h \in \CPK} \sum_{j=1}^{d^-}\left(\frac{2}{r_0}\log\abs{\alpha(h,n_{j},\beta_{j})}
    +c_P^+(h,n_j,\beta_j)\right)g_{j}(\tau),
  \]
  we obtain that the coefficients $c_{P}^{+}(h,n_{j},\beta_{j})$ satisfy the assertion for all $j$ and all $h$,
 Proposition \ref{SumG=E} is still satisfied and the principal part is unaltered.
 \end{proof}
  
  We can now finish the proof of Theorem \ref{thm:pre-pvals}.
  So far, we have shown \emph{(ii)} for $m<0$, \emph{(ii)} and \emph{(iii)} for $(m,\beta) = (n_j,\beta_j)$
  and also \emph{(iv)}.
  We now show \emph{(ii)} and \emph{(iii)} for all $m>0$ and $h=1$.
  The case $h \neq 1$ then follows by the reciprocity laws in Lemma \ref{lem:idelesact}
  and the action of $h$ on the special cycles as described in Sections 4 and 5 of \cite{ehlen-intersection}.

  Let $F_{m,\beta} \in M_{-1/2,L}^{!}$ as in Lemma \ref{lem:Fmb} and recall Equation \eqref{eq:coeff2}:
  \begin{equation}
    \label{eq:coeff2_2}
    \int^{\reg}_{\Gamma \bs \uhp} \langle f_{m,\beta}(\tau) , \overline{\Theta_{P}(\tau,h)} \rangle v\, d\mu(\tau)
      =  - \frac{4}{12^{A}\, c_{m,\beta}} \log \abs{\Psi_{L}(z_U, h, c_{m,\beta} F_{m,\beta})}.
  \end{equation}
  By Theorem \ref{thm:value-phiz}, the left hand side is equal to
  \begin{align}
    \label{eq:liftlhs}
     \{\pre{\Theta}_{P}(\tau,h),f_{m,\beta} \}
      &= 2 c_{P}^{+}(h,m,\beta) \\
      &\quad + \sum_{\gamma \in P'/P} \sum_{n > 0} c_{P}^{+}(h, -n, \gamma)c_{f_{m,\beta}}(n,\gamma)
      - 2\sum_{j=1}^{d^-} c_{P}^{+}(h, \beta_{j}, n_{j}) a_{j}(m, \beta). \notag
  \end{align}
Note that the first sum in the second line vanishes because $c_{P}^{+}(h, -n, \gamma) = 0$ for $(n,\gamma) \not\in \{(p_1,\pi_1), \ldots, (p_{d^+},\pi_{d^+})\}$ and 
$c_{f_{m,\beta}}(p_j, \pi_j) = 0$ for all $j$.

  We let $M_{m,\beta} \in \Z_{>0}$ with $c_{m,\beta} \mid M_{m,\beta}$,  
  such that  $\Psi_L(z,h, M_{m,\beta} \cdot r_0 \cdot F_{m,\beta}) \in \Q(j,j_N)$, where $r_0$ has been defined in Lemma \ref{lem:pos-index-1}.

  We replace $c_{m,\beta}$ by $r_0 \cdot M_{m,\beta}$ and
  the right-hand side of \eqref{eq:coeff2_2} now becomes
  \begin{equation}
    \label{eq:20}
    -4 r(m,\beta)^{-1} \log\abs{\tilde\alpha(1,m,\beta)},
  \end{equation}
  with $\tilde\alpha(1,m,\beta) := \Psi(z_U, 1, M_{m,\beta} \cdot r_0 \cdot F_{m,\beta}) \in \Hkb$
  and $r(m, \beta) := 12^A \cdot r_0 \cdot M_{m,\beta}$.
  
  Using \eqref{eq:coeff2_2}, \eqref{eq:liftlhs}, \eqref{eq:20}
  and Lemma \ref{lem:pos-index-1}, we obtain
  \begin{align*}
    c_{P}^{+}(1,m,\beta) = &-\frac{2}{r(m,\beta)} \log\abs{\tilde\alpha(1,m,\beta)} \\
      &- \frac{2 \cdot 12^{A}M_{m,\beta}}{r(m,\beta)} \sum_{j=1}^{d^-} a_{j}(m, \beta) \log\abs{\alpha(1,n_j,\beta_j)}.
  \end{align*}
  This implies that
  \[
    c_{P}^{+}(1,m,\beta) = - \frac{2}{r(m,\beta)}
       \log\abs{\tilde\alpha(1,m,\beta)
          \prod_{j=1}^{d^-}\alpha(1,n_{j},\beta_{j})^{12^{A}M_{m,\beta} a_{j}(m,\beta)}}.
  \]
  Note that the number in the absolute value is algebraic and contained in $\Hkb$.
  To see this, note that the denominator of $a_{j}(m,\beta)$ is bounded by $c_{m,\beta}$.
  This implies that
  \[
    \alpha(1,n_{j},\pm\beta_{j})^{12^{A}M_{m,\beta} a_{j}(m,\beta)} \in \calO_{\Hkb}
  \]
  for all $j$.
  Finally, let
  \[
    \alpha(1,m,\beta) = \tilde\alpha(1,m,\beta)
                        \prod_{j=1}^{d^-}\alpha(1,n_{j},\beta_{j})^{12^{A}M_{m,\beta} a_{j}(m,\beta)} \in \Hkb.
  \]
  This shows \eqref{eq:cprma} for $m>0$ and $\alpha(1,m,\beta)$ satisfies
  \begin{align}
    \label{eq:21a}
        \ord_{\frakP}(\alpha(1,m,\beta)) = &\ord_{\frakP}(\tilde\alpha(1,m,\beta)) \\
      &+ 12^{A}w_{\kb} M_{m,\beta} r_0 \sum_{j=1}^{d^-} a_{j}(m,\beta) \calZ(n_{j},\fraka,\beta_{j})_{\frakP} \notag
  \end{align}
by Lemma \ref{lem:pos-index-1}.
  
  By Proposition \ref{prop:properint}, we know that
  $\calZ(F_{m,\beta})$ and $\calZ(M,\mu_{R})$ intersect properly
  and we have by \eqref{eq:pord1} that
  \begin{equation*}
    \ord_{\frakP}(\tilde\alpha(1,m,\beta)) =
        r(m,\beta) \frac{w_\kb}{2}\, \sum_{\gamma \in P'/P} \sum_{n >0} c_{f_{m,\beta}}(-n,\gamma) \calZ(n,\fraka,\gamma)_{\frakP}
  \end{equation*}
  which we can further expand to
  \begin{align}
    \label{eq:21b}
        \ord_{\frakP}(\tilde\alpha(1,m,\beta))
        &= r(m,\beta) w_{\kb}\, \calZ(m,\fraka,\beta)_{\frakP} \\
        &\quad - r(m,\beta) w_{\kb}\, \sum_{j=1}^{d^-} a_{j}(m,\beta)\calZ(n_{j},\fraka,\beta_{j})_{\frakP}, \notag
  \end{align}
  where we used that $\calZ(n,\fraka,-\gamma)_{\frakP} = \calZ(n,\fraka,\gamma)_{\frakP}$ and $c_{f_{m,\beta}}(n,\gamma) = c_{f_{m,\beta}}(n,-\gamma)$.
  
  Plugging \eqref{eq:21b} into \eqref{eq:21a} and using and $r(m,\beta) = 12^{A} M_{m,\beta} r(n_j,\beta_j)$ we conclude
  \[
    \ord_{\frakP}(\alpha(1,m,\beta)) = r(m,\beta) w_{\kb}\, \calZ(m,\fraka,\beta)_{\frakP},
  \]
  which proves \emph{(iii)}, but with $r(m,\beta)$ possibly depending on $m$ and $\beta$.

\subsubsection{A bound for $r(m,\beta)$}
\label{sec:independence}
We will now finish the proof by showing that the integers $r(m, \beta)$ in the theorem can be bounded
so that we can choose an integer $r$ only depending on the isomorphism class of $P$ (and thus can be chosen to only depend on $D$).

\begin{proposition}
  \label{prop:independence}
  There is an $r \in \Z_{>0}$, such that we can take $r(m,\beta) = r$ for all $m \in \Q$, $\beta \in P'/P$ and $h \in \CPK$.
\end{proposition}
\begin{proof}
  Recall that for square-free $N$ each cusp of $\Gamma_0(N)$ can be represented a fraction $1/c$ with $c \mid N$.
    
  For $f \in M_{1/2,L}^{!}(\Z)$, the
  Weyl vector associated with $f$ at the cusp $1/c$, is defined by
  \[
    \rho_{f,c} = \frac{\sqrt{N}}{8 \pi}\int_{\calF}^{\reg} \langle f(\tau), \overline{\Theta_{K_{c}}(\tau)} \rangle v^{1/2} d\mu(\tau),
  \]
  where $K_{c}$ is a certain one-dimensional positive definite lattice attached to the cusp $1/c$
  satisfying $K_{c}'/K_{c} \cong L'/L$.

  The significance of the Weyl vectors is that the divisor of $\Psi_{L}(z,f)$
  on $X_0(N)$ is given by $Z(f) + C(f)$, where $C(f) = \sum_{c \mid N} \rho_{f,c} \cdot (1/c)$.
  
  We assume that the constant term of $f$ vanishes,
  $c_f(m,\mu) \in \Q$ for all $m \in \Q$ and $\mu \in L'/L$ and $c_f(m,\mu) \in \Z$ for $m \leq 0$.
  Since the multiplier system $\sigma$ of $\Psi_{L}(z,f)$ is of finite order $M$,
  we have that $M \cdot (Z(f) + C(f))$ is the divisor of a rational function on $X_0(N)$.
  Note that this also implies that $\deg(Z(f) + C(f))=0$.
  Conversely, as in the proof of Theorem 6.2 in \cite{bronolderiv},
  we have that if $M \cdot (Z(f) + C(f))$ is the divisor of a rational function on $X_0(N)$
  then $\sigma^{M}$ is trivial.
  If $\rho_{f, c} \in \Z$ for all cusps, then $Z(f) + C(f)$ defines a rational point in the Jacobian $J$ of $X_0(N)$
  because the Heegner divisors are defined over $\Q$ and so are the cusps of $\Gamma_0(N)$ for $N$ square-free.
  The group of rational points $J(\Q)$ is a finitely generated abelian group by the Mordell-Weil theorem.
  Therefore, $M  \mid M(N)$, where $M(N)$ is the order of the torsion subgroup of $J(\Q)$.

  The idea is now, similar to the general theme of this paper,
  to construct a harmonic Maa\ss{} form $\pre{\Theta}_{K}(\tau) \in \calH_{3/2,K^-}$, such that
  $\xi_{3/2}(\pre{\Theta}_{K_{c}})(\tau) = \Theta_{K_{c}}(\tau)$ to compute the Weyl vectors
  and ensure that they are integral so that the above reasoning applies.

  In the case of a prime discriminant, the theta function $\Theta_{K_{c}}(\tau)$ is in fact an Eisenstein series.
  A preimage is given by a generalization of Zagier's Eisenstein series, constructed as the 
  Kudla-Millson lift of the constant functions $1$ as in
  Theorem 4.5 of \cite{brfu06} (see also Theorem 5.5 of \cite{alfesehlentraces}
  for the specialization to $\Gamma_0(N)$).

  For the general case, the construction is a bit more complicated and involves
  showing the existence of weakly holomorphic modular forms of weight $0$
  for $\Gamma_0(N)$ with certain prescribed principal parts. The details can be
  found in \cite{BS-pretheta12}; in particular, by Theorem 4.5 together
  with Theorem 5.1 (1) op. cit. we obtain that there is a constant $\kappa$ that only depends on $A\abs{D}=N$, such that $\kappa \cdot \rho_{f,c} \in \Z$ for all $c \mid N$ and all $f \in M_{1/2, L}^!$ with only integral Fourier coefficients. 

  Without loss of generality we can assume that $A$ is a prime and does not divide $D$.
  Recall the numbers $n_1, \ldots, n_{d^-}$ that belong to our special basis of $S_{1,P^{-}}$
  constructed in Section \ref{sec:special-basis}.
  Let $n > \max \{n_{1}, \ldots, n_{d^-}\}$ be a fixed integer and $S = \{ n + Q(\beta)\ \mid\ \beta \in P'/P \}$.
  For $x \in S$ and $\beta \in P'/P$, such that $x \equiv Q(\beta) \bmod{\Z}$, we let
  $F_{x,\beta} := \bm{f}_{x,\beta}$ as in Lemma \ref{lem:Fmb}.
  From our discussion above, it follows that there is an $r \in \Z_{>0}$, 
  such that the following properties are satisfied:
  \begin{enumerate}
  \item The functions $r \cdot f_{x,\beta}$ have integral Fourier coefficients for all $x \in S$;
  \item if $r \cdot f \in M_{1,P}^!$ has integral Fourier coefficients, then $r' \cdot \bm{f} \in M_{1/2,L}^!$ constructed as in Lemma \ref{lem:Fmb}
  also has integral Fourier coefficients with $r' = r/12^A$;
  \item and we have that $\rho_{r' \cdot \bm{f}, c} \in \Z$ for all $c \mid N$ and $\Psi_{L}((z,h), r' \cdot \bm{f}) \in \Q(j,j_{N})$.
  \item Moreover, for all $m \leq n$, we have
    \[
      c_{P}^{+}(h,m,\beta) = -\frac{2}{r} \log\abs{\alpha(h,m,\beta)}
    \]
    with $\alpha(h,m,\beta) \in \calO_{\Hkb}$ and
    $\ord_{\frakP}(\alpha(h,m,\beta)) = w_\kb \cdot r \cdot \calZ(m,h.\fraka,h.\beta)_{\frakP}$.
  \end{enumerate}
  
  Now let $m \in \Q$ and $\beta \in P'/P$, such that $m>n$ and $m \equiv Q(\beta) \bmod{\Z}$.
  Then $m = x + n'$ for some $x \in S$ and $n' \in \Z_{>0}$.
  We let $f(\tau) = j^{n'}(\tau)f_{x,\beta}(\tau)$, where $j(\tau)$ is the $j$-invariant.
  Moreover, we let $\bm{f}(\tau) \in M_{1/2,L}^{!}$ given by
  \[
    \bm{f} = f \otimes \GJ^A - C  E_{P} \otimes \FJ g, 
  \]
  for a suitable constant $C\in \Q$ and $g \in M_{2A}^!(\SL_2(Z))$ as in the proof of Lemma~\ref{lem:Fmb}, such that
  $c_{\bm{f}}(0,0)=0$.
  Note that $\GJ$ and $j^{n'}$ have integral Fourier coefficients.
  By (i), Lemma~\ref{lem:Fmb} and (ii) we have that $r' \cdot \bm{f}$ has integral Fourier coefficients and
  by (iii) the Weyl vectors of $r' \cdot \bm{f}$ are integral and $\Psi_{L}((z,h), r' \cdot \bm{f})$ is contained in $\Q(j,j_N)$.
  
  As before, we have on the one hand
  \[
    \Phi_{P}(h, f) = 2 c_{P}^{+}(h,m,\beta) + \sum_{\gamma \in P'/P} \sum_{m' < m} c_{P}^{+}(h,m',\gamma)c_{f}(-m',\gamma)
  \]
  and on the other hand
  \[
    \Phi_{P}(h,f) = \frac{-2}{r} \log\abs{\Psi_{L}(z_U, h,r' \cdot \bm{f})}^2.
  \]
  By our assumptions we have that $\Psi_{L}(z_U,h,r' \cdot \bm{f}) \in \calO_{\Hkb}$.
  The statement of the proposition now easily follows by induction on $n' \in \Z_{>0}$.
\end{proof}

\begin{remark}
  Let $l \geq 5$ be a prime, let $n$ be the numerator of $(l-1)/12$ and consider the Jacobian $J$ of $X_0(l)$.
  It follows from a Theorem of Mazur \cite{mazur-modcurves-eisenstein} that
  the torsion subgroup $J(\Q)_{tors}$ is cyclic of order $n$.
  Together with this explicit result, Proposition \ref{prop:independence}
  describes an algorithm to determine the number $r$ for a given prime discriminant $D=-l$
  explicitly.
\end{remark}

\subsection{Consequences}
\label{sec:consequences}
The following theorem shows the modularity of the generating series of the degrees of the special cycles, as mentioned in the introduction.
This result was proven by Kudla, Rapoport and Yang using an explicit comparison of Fourier coefficients.
\begin{theorem}
  \label{thm:arithgen}
  There exists a harmonic weak Maa\ss{} form $\pre{E}_{P}(\tau) \in M_{1,P^{-}}$
  with vanishing principal part, such that
  \[
    -\frac{h_{k}}{w_{\kb}} \prep{E}_{P}(\tau) = \sum_{\mu \in P'/P}\sum_{m > 0} \widehat\deg\,\calZ(m,\fraka,\mu) e(m\tau) \phi_{\mu} + c\phi_0,
  \]
  where $c \in \C$ is a constant.
\end{theorem}
\begin{proof}
  Consider the sum
  \[
    \pre{E}_{P}(\tau) = \frac{1}{h_{k}}\sum_{h \in C_{k}/K} \pre{\Theta}_P(\tau,h).
  \]
  We obtain by Theorem \ref{thm:pre-pvals} that
  the holomorphic part of $\pre{E}_{P}(\tau)$ has a vanishing principal part and we have
  \begin{align*}
    \frac{h_{k}}{w_{\kb}}\pre{E}_{P}^{+}(\tau) &=
    \frac{1}{w_{\kb}} \sum_{\beta \in P'/P} \sum_{m \geq 0} \sum_{h \in \CPK} c_{P}^{+}(h,m,\beta)e(m \tau) \phi_{\beta}\\
      &= -\frac{1}{w_{\kb}} \sum_{\beta \in P'/P}\sum_{m \geq 0} r^{-1} \sum_{h \in \CPK}
          \log{\abs{\alpha(h,m,\beta)}^2} e(m \tau) \phi_{\beta},
  \end{align*}
  where $\alpha(h,m,\beta) \in \calO_{\Hkb}$.
  
  Moreover, for $m>0$ and for every prime $\frakP$ of the Hilbert class field $\Hkb$
  of $\kb$, we have
  \[
    \ord_{\frakP}\left(\prod_{h} \alpha(h,m,\beta)\right) = r w_{\kb} \sum_{h} \calZ(m,h.\fraka,h.\beta)_\frakP.
  \]
  By Proposition 5.6 in \cite{ehlen-intersection}, we have for $\sigma = \sigma(h)$ that
  \[
    \calZ(m,\fraka,\beta)_{\frakP^{\sigma}} = \calZ(m,h^{-1}.\fraka,h^{-1}.\beta)
  \]
  and thus
  \[
     \sum_{h \in C_{k}/K} \calZ(m,h.\fraka,h.\beta)_\frakP = \sum_{\sigma \in \Gal(\Hkb/\kb)} \calZ(m,\fraka,\beta)_{\frakP^{\sigma}}.
  \]
  This shows in fact that
  \[
    \sum_{h} \log\abs{\alpha(h,m,\beta)}^2 = \log\N_{\Hkb/\Q}(\alpha(1,m,\beta)).
  \]
  But since $\calZ(m,\fraka,\beta)$ is supported at a unique prime $p$, we also have that
  \[
    \deghat\, \calZ(m,\fraka,\beta) = \sum_{\sigma \in \Gal(\Hkb/\kb)}\calZ(m,\fraka,\beta)_{\frakP^{\sigma}} \log \N_{\Hkb/\Q}(\frakP).
  \]
  Here, we used that $p$ is non-split in $\kb$.
  Then we find
  \begin{align*}
    \sum_{h \in C_{k}/K} \log{\abs{\alpha(h,m,\beta)}^{2}}
    &= w_\kb \cdot r \cdot \sum_{\sigma \in \Gal(\Hkb/\kb)} \calZ(m,\fraka,\beta)_{\frakP^{\sigma}} \log \N_{\Hkb/\Q}(\frakP)\\
    &= w_\kb \cdot r \cdot \deghat\, \calZ(m,\fraka,\beta),
  \end{align*}
  which implies the statement of the proposition.
\end{proof}

We collect some consequences of Theorem \ref{thm:pre-pvals}
in connection with the explicit formulas given in \cite{ehlen-intersection}.

For $m \in \Q_{>0}$, we define a set of rational primes by
\begin{equation*}
  \Diff(m) = \{ p < \infty\ \mid\ (-m\N(\fraka),D)_p = -1 \}.
\end{equation*}
A starting point is the following Corollary.
\begin{corollary}
The algebraic numbers in Theorem \ref{thm:pre-pvals} satisfy the following properties.
\begin{enumerate}
\item  If $\abs{\Diff(m)} \neq 1$, then $c_{P}^{+}(h,m,\beta) = -\frac{2}{r} \log\abs{\epsilon}$
    for $\epsilon \in \calO_{\Hkb}^{\times}$.
  \item If $\Diff(m)=\{p\}$, we have $\ord_{\frakP}(\alpha(h,m,\beta))=0$ for all primes $\frakP \nmid p$.
  \item The Shimura reciprocity
          \[
            \alpha(h,m,\beta) = \alpha(1,m,\beta)^{\sigma(h)}
          \]
          holds, where $\sigma(h)$ corresponds to $h$ under the Artin map.
\end{enumerate}
\end{corollary}
Therefore, we may assume from now on that $\Diff(m) = \{p\}$. In particular, $p$ is nonsplit in $\kb$.
We recall some notation needed to describe the results 
-- for details we refer the reader to \cite{ehlen-intersection}.
We let $p_{0} \in \Z$ be a prime with $p_{0} \nmid 2pD$ such that if $p$ is inert in $\kb$, we have
  \[
   (D,-pp_{0})_{v} =
   \begin{cases}
     -1, & v=p,\infty,\\
     1, & \text{otherwise,}
   \end{cases}
  \]
and if $p$ is ramified in $\kb$, we have
  \[
    (D,-p_{0})_{v} =
    \begin{cases}
     -1, & v=p,\infty,\\
     1, & \text{otherwise.}
   \end{cases}
  \]
With this choice, we let $\frakp_{0}$ be any fixed prime ideal of $\OD$ lying above $p_{0}$.
Here, $( \cdot, \cdot)_{v}$ denotes the $v$-adic Hilbert symbol.
If $p$ is inert in $\kb$, let $\frakc_{0} = \frakp_{0}\different{\kb}$.
If $p$ is ramified and $\frakp \subset \OD$ is the prime above $p$,
let $\frakc_{0} = \frakp_{0}\frakp^{-1}\different{\kb}$.

Moreover, we define for $m \in \Q$:
\begin{equation}
  \label{eq:nudef}
  \nu_{p}(m) =
      \begin{cases}
            \frac{1}{2}(\ord_{p}(m)+1), & \text{if $p$ is inert in $\kb$},\\
            \ord_{p}(m \abs{D}), & \text{if $p$ is ramified in $\kb$}.
      \end{cases}
\end{equation}
We let $\frakP_{0}$ be the prime ideal
corresponding to the elliptic curve with complex multiplication 
$(E_0, \iota_0)$ over $\bar{\F}_p$ such that the class $[L(E_0,\iota_0)]$ of the rank one
$\OD$-module $L(E_0,\iota_0)$ in $\Pic(\OD)$ is equal to $[\frakp_0]^{-1}$.
Finally, for $m \in \Q$, 
we let $o(m)$ denote the number of primes $q \mid D$ such that $\ord_{q}(m\abs{D})>0$.
In the followig corollary, we restrict our attention to the simpler cases.
\begin{corollary}\label{cor:coeffs-explicit}
\ \\
  \begin{enumerate}
  \item Let $L$ be the subfield of $\Hkb$
        fixed by all elements of order less or equal than two in $\Gal(\Hkb/\kb)$
        and let $\frakf \subset \calO_{L}$ be the prime ideal below the fixed prime $\frakP_{0}$.
        We have
        \begin{align*}
          \frac{\ord_{\frakf}(\N_{\Hkb/L}(\alpha(h,m,\beta)))}{r \cdot w_{\kb}} &= \calZ(m,h.\fraka,h.\beta)_{\frakf}\\
                       &= 2^{o(m)-1} \nu_{p}(m) \rho(m\abs{D}/p, [h]^2[\frakc_{0}\fraka]).
        \end{align*}
  \item If $D=-l$ is prime, then there is a unique prime $\frakP \mid p$ that is fixed by complex conjugation.
        We have
        \[
          \frac{\ord_{\frakP}(\alpha(h,m,\beta))}{r \cdot w_{\kb}} = 2^{o(m)-1} \nu_{p}(m) \rho(m\abs{D}/p, [h]^{2}[\fraka]^{2}).
        \]
  \end{enumerate}
\end{corollary}

Note that we did not use the known explicit formulas for $\calE_P(\tau)$ so far.
Employing these formulas \cite{KudlaYangEisenstein}, however, we obtain another corollary.
\begin{corollary}
  \label{cor:calE=preE}
  With the same notation as in Theorem \ref{thm:pre-pvals}, we have
  \[
    \pre{E}_P(\tau) = \calE_{P}(\tau),
  \]
  where $\calE_{P}(\tau) = \frac{\partial}{\partial s}\hat E_P(\tau,s)\mid_{s=0}$
  is defined in \eqref{eq:calE}.
\end{corollary}

Using the pairing \eqref{eq:pairing}, the following Corollary is immediate.
\begin{corollary}
  \label{cor:ct}
   The constant term of $\pre{\Theta}_{P}(\tau,h)$ is given by
    \[
      c_{P}^{+}(h,0,0) = -\sum_{\beta \in L'/L} \sum_{m > 0} c_{P}^{+}(h,-m,\beta)\tilde\rho(m,\beta)
                        - 2 \frac{\Lambda'(\chi_{D},0)}{\Lambda(\chi_{D},0)},
    \]
    where $\tilde\rho(m,\beta)$ is the coefficient of index $(m,\beta)$ of $E_P(\tau)$.
\end{corollary}

\section{The scalar valued case}
\label{sec:proof-thm2}
To complete the picture and give a slightly more classical description of our results,
we briefly describe the scalar valued case in this section.
Throughout, we let $D<0$ with $D \equiv 1 \bmod 4$ be a fundamental discriminant
  and let $\frakA \in \Clk/\Clks$ be a genus.
For every class $[\fraka] \in \frakA$, there is a theta function $\theta_\fraka(\tau)$ contained in $M_1(\abs{D}, \chi_D)$.
A special case of the Siegel-Weil formula states that
\[
   \sum_{\frakb \in \Clk} \theta_{\fraka\frakb^2}(\tau) = h_\kb E_\frakA(\tau),
\]
where $E_{\frakA}(\tau)$ is the normalized genus Eisenstein series attached to $\frakA$.
It is well known that
\[
E_\frakA(\tau) = 1 + \frac{w_\kb}{h_\kb}\sum_{n=1}^\infty \rho_\frakA(n) e(n \tau),
\]
where $\rho_\frakA(n)$ is equal to the number of integral ideals of norm $n$ in the genus $\frakA$.
Note that $E_\frakA(\tau)$ is the $\phi_0$-component of $E_P(\tau)$.

\begin{corollary}
  \label{cor:scalar-pre}
  For every $[\fraka] \in \frakA$, there is a
  $\pre{\theta}_\fraka \in \calH_{1}(\abs{D},\chi_D)$, where $\chi_{D}$ is given by the Kronecker symbol, with the following properties. We write $c_\fraka^+(n)$ for the coefficients of the holomorphic part of $\pre{\theta}_\fraka$.
  \begin{enumerate}
  \item We have $\xi(\pre{\theta}_{\fraka})=\theta_{\fraka}$ and
        \[
          \sum_{[\frakb] \in \Clk} \pre{\theta}_{{\frakb}^2\fraka} = h_\kb\, \pre{E}_{\frakA},
        \]
        where $\xi(\pre{E}_{\frakA}) = E_{\frakA}$ and the principal part of
        $\pre{E}_{\frakA}$ vanishes.
  \item For all $n > 0$ we have
    \[
      c_{\fraka}^{+}(n) = - \frac{2}{r} \log\abs{\alpha(\fraka,n)},
    \]
    where $\alpha(\fraka,n) \in \calO_{\Hkb}$ and $r \in \Z_{>0}$ is independent of $n$.
  \item Finally, for all $n < 0$, we have that
    \[
      c_{\fraka}^{+}(n) = - \frac{2}{r} \log\abs{\alpha(\fraka,n)},
    \]
    where $\alpha(\fraka,n) \in \calO_{\Hkb}^{\times}$.
  \end{enumerate}
\end{corollary}
\begin{proof}
  All statements follow from Theorem \ref{thm:pre-pvals} by considering
  \[
    \pre{\theta}_{\fraka\frakb^{2}}(\tau) = \pre{\Theta}_{P,0}(\tau,h) \in \calH_{1}(\abs{D},\chi_{D})
  \]
  where $P = \fraka$, $(h) = \frakb$ and
  \[
    \pre{\Theta}_{P}(\tau,h) = \sum_{\beta \in P'/P} \pre{\Theta}_{P,\beta}(\tau,h) \phi_{\beta}.
  \]
  Since we have $\theta_{\fraka}(\tau) = \Theta_{P,0}(\tau)$,
  we conclude $\xi(\pre{\theta}_{\fraka}) = \theta_{\fraka}(\tau)$.
\end{proof}

\begin{corollary}
  \label{cor:scalar-pval}
  We use the same notation as in Corollary \ref{cor:scalar-pre}.
  \begin{enumerate}
    \item We have $\ord_{\frakP}(\alpha(\fraka,n)) = 0$, unless $\abs{\Diff(n)} = 1$.
    \item
      If $\Diff(n) = \{p\}$, then
      \[
        \ord_{\frakP_0}(\alpha(\fraka,n))
          = r \cdot w_\kb \cdot \nu_p(n) \rho(n \abs{D} / p, [\fraka\frakc_0]),
      \]
      where $\frakP_{0} \mid p$ is the distinguished prime as above.
    \item The Shimura reciprocity
          \[
            \ord_{\frakP_0^{\sigma}}(\alpha(\fraka,n)) = \alpha(\frakb^{-2}\fraka,n),
          \]
          holds, where $\sigma = \sigma(\frakb)$.
    \item For the coefficients in the principal part, we have that $\alpha(\fraka,n) \in \calO_{\Hkb}^{\times}$.
  \end{enumerate}
\end{corollary}
\begin{proof}
  This follows from Theorem \ref{thm:pre-pvals} and the proof of Propositions 3.8 and 3.9 in \cite{ehlen-intersection}
  by setting $\mu = 0$ and noting that the condition $\lambda_\fraka x \in \fraka$ is superfluous if $m \in \Z$ (using the notation of loc. cit.).
\end{proof}

\begin{remark}
  Note that we are able to give a relatively simple formula for the $\frakP$-valuations
  of the coefficients in comparison to the vector valued case, where this is only possible after passing to the fixed field of elements of order two in the Galois group $\Gal(\Hkb/\kb)$. In fact, considering scalar valued theta series corresponds to ``forgetting'' the additional congruence conditions that the vector valued forms keep track of.
\end{remark}

\subsection{The conjecture of Duke and Li}
\label{sec:conjecture-duke-li}

In \cite{DukeLiMock}, W. Duke and Y. Li also study the scalar valued case.
Motivated by numerical experiments, the authors were led to formulate a conjecture, which we state now, adapted to our notation.

Let $\kb = \Q(\sqrt{-l})$ for a prime $l \equiv 3 \bmod{4}$, $l>3$. We consider the lattice $P = \calO_\kb$ given by the ring of integers in $\kb$,
which is sufficient because there is only one genus in this case. 
We note that the preimages are normalized differently in op. cit. than in the previous section, leading to a slightly different result
if we compare the conjecture with Corollary \ref{cor:scalar-pval}.

Now let $p$ a prime that is non-split in $\kb$, let $\frakP_{0} \mid p$ be the unique prime above $p$
fixed by complex conjugation, and $\frakP^{\sigma_{\frakb}} = \frakP_{0}$ for a fractional ideal $\frakb$.
\begin{conjecture}[\cite{DukeLiMock}] \label{con:dukelipval}
The functions $\pre{\theta}_\fraka(\tau)$ can be chosen such that for $n \in \Z_{>0}$ with $\chi_{p}(n) \neq 1$, we have
\[
  c^{+}_{\fraka^2}(n) = - \frac{2}{r} \log\abs{u(\fraka^2,n)}
\]
with $u(\fraka^2,n) \in \calO_{\Hkb}$ and, if $\Diff\left(\tfrac{n}{\abs{D}}\right) = \{p\}$, then
\begin{equation}
  \label{eq:dukelipval}
  \ord_{\frakP}(u(\fraka^2,n)) = 2r \sum_{m \geq 1} \rho\left( \frac{n}{p^{m}},[\fraka\frakb]^{2}\right),
\end{equation}
with $r \in \Z$ independent of $n$ and divides $24h_{\kb}h_{\Hkb}$.
\end{conjecture}
\begin{remark}
  1) Compare with the conjecture as stated in op. cit., 
   note that the normalization of the theta functions
  differs by a factor of $1/2$ from our normalization; this leads to the factor $2$ on the right-hand side above.
  2) We added the condition $\Diff\left(\tfrac{n}{\abs{D}}\right) = \{p\}$ because it is needed for the conjecture to hold
  and be compatible with the Siegel-Weil formula (which essentially corresponds to (iv) in Theorem 1.1 of op. cit.).
\end{remark}

We will see below that our results imply the conjecture regarding the valuations,
but we were not able to show the explicit bound $24h_{k}h_{\Hkb}$ for $r$.

The next lemma gives a simpler expression for the right-hand side of \eqref{eq:dukelipval},
relating it to our formula.
\begin{lemma} \label{lem:duklivsours}
  If $n \in \Z_{>0}$ with $\Diff\left(\tfrac{n}{\abs{D}}\right) = \{p\}$, then
  \[
     \sum_{m \geq 1} \rho\left( \frac{n}{p^{m}}, [\frakc] \right) = \nu_{p}\left(\tfrac{n}{\abs{D}}\right)\, \rho(n/p, [\frakc])
  \]
for any class $[\frakc] \in \Clk$.
\end{lemma}
\begin{proof}
Suppose that $p \neq l$. Then $\chi_{l}(-n) = -1$, $l \nmid n$ and
we only have a non-zero contribution if $\ord_{p}(n) > 0$ and $\ord_{p}(n)$ is necessarily odd.
More precisely, we have
\[
  \rho\left(\frac{n}{p^{m}}, [\frakc] \right) =
  \begin{cases}
    0, &\text{ if } m \text{ is even}, \\
    \rho\left( n/p, [\frakc] \right), &\text{ if } m \text{ is odd.}
  \end{cases}
\]
Since $p$ is inert, there are no ideals of norm $p$.
Therefore, the map $\frakc \mapsto p\frakc$ establishes a bijection between
ideals of norm $\N(\frakc)$ and ideals of norm $p^{2}\N(\frakc)$,
leaving the class invariant. Thus, we obtain the contribution
$\rho(n/p, [\frakc])$ times $\frac{1}{2}(\ord_{p}(n)+1)$, as in our formula.

The case $p = l$ is similar, but this time we obtain $\ord_{p}(n) \cdot \rho(n / p, [\frakc])$
because there is a unique ideal of norm $p$.
\end{proof}

\begin{theorem} \label{thm:dukeli}
  Conjecture \ref{con:dukelipval} is true.
\end{theorem}
\begin{proof}
  To prove the theorem, we take $\pre\theta_{\fraka^2}(\tau) := \sum_{\beta \in P'/P} \pre{\Theta}_{P,\beta}(\abs{D}\tau,h)$, where $(h) = \fraka$ instead of the zeroth component as in the proof of Corollary \ref{cor:scalar-pre}. First, note that by Equation (2) in \cite{BruinierBundschuh}
$\pre\theta_{\fraka^2}(\tau) \in \calH_{1}(\abs{D},\chi_D)$.
Moreover, it is easy to see that in fact
\[ 
  \sum_{\beta \in P'/P} \Theta_{P,\beta}(\abs{D}\tau,h) = \theta_{\fraka^2}(\tau).
\]
Indeed, multiplication by $\sqrt{D}$ gives a bijection between elements of norm $m$ in $\diffinv{\kb}\fraka^2$ and elements of norm $m\abs{D}$ in $\fraka^2$.
Hence, $\xi_1(\pre\theta_{\fraka^2}(\tau)) = \theta_{\fraka^2}(\tau)$.

By Theorem \ref{thm:pre-pvals} and Corollary \ref{cor:coeffs-explicit}, we have that
the $n$-th Fourier coefficient of $\pre\theta_{\fraka^2}(\tau)$ is equal to 
\[
  \frac{-2}{r}\sum_{\beta \in P'/P}\log\abs{\alpha(h,\tfrac{n}{\abs{D}},\beta)}
\]
with (setting $m:= n/\abs{D}$)
\[
  \ord_{\frakP_0}(\alpha(h,m,\beta)) = 2^{o(m)}r \cdot \nu_p\left(m\right)\rho(n/p, [\fraka]^2)
\]
if $m +\Q(\beta) \in \Z$.
Now, if $l \nmid n$, then $o(m)=0$ and if $\rho(n/p,[\fraka]^2) \neq 0$, then there are exactly two elements in $\beta \in P'/P$
with $m+Q(\beta) \in \Z$. If $l \mid n$, then $\beta = 0$ is the only element in $P'/P$ satisfying $m+Q(\beta) \in \Z$ and also $o(m)=1$.
In any case
\[
 \sum_{\beta \in P'/P}\ord_{\frakP_0}(\alpha(h,\tfrac{n}{\abs{D}},\beta)) = 2 r \cdot \nu_p\left(\tfrac{n}{\abs{D}}\right)\rho(n/p, [\fraka]^2),
\]
which agrees with the conjecture by Lemma \ref{lem:duklivsours}.
\end{proof}
Note that our theorem also covers the case $l=3$ and we provide a generalization to composite discriminants.

\begin{proof}[Proof of Theorem \ref{thm:dukeliconj}]
  Theorem \ref{thm:dukeliconj} is a reformulation of the conjecture in
the vector-valued case and follows directly from Theorem \ref{thm:pre-pvals}, and Corollary \ref{cor:coeffs-explicit} \emph{(ii)} together with Lemma \ref{lem:duklivsours}.
\end{proof}

\section[A comprehensive example]{A comprehensive example for $D=-23$}
\label{sec:compr-example}
In this section, we will give an example where we can give an explicit, finite formula
for all the coefficients of the forms $\pre{\Theta}_P(\tau,h)$
and follow the proof of Theorem \ref{thm:pre-pvals}.

The method of this section can be used to obtain such explicit formulas for all primes $p \equiv 3 \bmod{4}$
such that the modular curve $X_0^+(l)$ has genus zero. Here, $X_0^+(l)$ 
is the compactification of $\Gamma_0(l)^* \bs \uhp$, where $\Gamma_0(l)^*$ is
the extension of $\Gamma_0(p)$ by the Fricke involution.  It is known that there are
only finitely many such primes.
In order to keep the example as explicit as possible, we specialize to $l=23$.
It is the first prime congruent to $3$ modulo $4$ such that $\kb=\Q(\sqrt{-l})$ has class number $h_{\kb}>1$.
In this case, we have $h_{\kb}=3$.
The three ideal classes can be represented by the ideals
$\calO = (1), \fraka = \left(2,(2+\sqrt{-23})/2 \right)$ and $\bar{\fraka}$.

Note that $M_{1,P}$, for both $P=\calO$ and $P=\fraka$,
is isomorphic \cite{BruinierBundschuh} to $M_{1}^+(\Gamma_{0}(23),\chi_{-23})$, the subspace of
modular forms whose Fourier coefficients of index $n$ vanish whenever $\kronecker{-23}{n}=-1$.
Thus, for simplicity we will work with scalar valued modular forms in this example.

The two corresponding scalar valued theta functions have Fourier expansions starting with
\begin{align*}
  \theta_{\calO}(\tau) &= 1 + 2q + 2q^{4} + 4q^{6} + 4q^{8} + 2q^{9} + 4q^{12} + O(q^{15}) \\
  \theta_{\fraka}(\tau)=\theta_{\bar{\fraka}}(\tau) &= 1 + 2q^{2} + 2q^{3} + 2q^{4} + 2q^{6} + 2q^{8} + 2q^{9} + 4q^{12} + 2q^{13} + O(q^{15}).
\end{align*}
The difference
\begin{equation*}
  g(\tau) = \frac{1}{2} ( \theta_{\calO}(\tau) - \theta_{\fraka}(\tau) ) = q - q^{2} - q^{3} + q^{6} + q^{8} - q^{13} + O(q^{15})
\end{equation*}
can easily be identified as $\eta(\tau) \eta(23\tau)$.
We will write $c_g(n)$ for the coefficient of index $n$ of $g$
and $c_{\calO}(n)$ and $c_\fraka(n)$ for the coefficients of
the two theta series above. The cusp form $g$ is a normalized newform and does not have any zeroes on $\uhp$.

In fact, we have in our case $M_{1}^+(\Gamma_{0}(23),\chi_{-23}) = M_{1}(\Gamma_{0}(23),\chi_{-23})$
and this space is spanned by $g$ and the genus Eisenstein series
\begin{align*}
  E(\tau) &= \frac{1}{3} (\theta_{\calO}(\tau) +  2\theta_{\fraka}(\tau))\\
          &= 1 + \frac{2}{3}q + \frac{4}{3}q^{2} + \frac{4}{3}q^{3} + 2q^{4}
             + \frac{8}{3}q^{6} + \frac{8}{3}q^{8} + 2q^{9} + 4q^{12} + \frac{4}{3}q^{13} + O(q^{15}).
\end{align*}
We will write $c_E(n)$ for the coefficient of index $n$ of $E(\tau)$.
The vanishing of $M_{1}^-(\Gamma_{0}(23),\chi_{-23})$ corresponds to the fact that
$M_{1,P^-}=\{0\}$ for either choices of $P$.

We would like to determine the coefficients of $\pre{\theta}_\calO(\tau)$ and $\pre{\theta}_\fraka(\tau)$.
Part of our normalization is that $\pre{\theta}_\calO(\tau) = \pre{E}(\tau) + \frac{4}{3}\pre{g}(\tau)$ and
$\pre{\theta}_\fraka(\tau) = \pre{E}(\tau) - \frac{2}{3}\pre{g}(\tau)$, where $\xi_1(\pre{g}(\tau)) = g(\tau)$
and $\xi_1(\pre{E}(\tau))=E(\tau)$ with vanishing principal part.

An important consequence of the vanishing of $M_{1,P^-}$ is that there are no obstructions to
finding a weakly holomorphic modular form in $M_{1,P}^! \cong M_{1}^{!,+}(\Gamma_{0}(23),\chi_{-23})$.
Therefore, the technical part concerning the cusp forms in $S_{1,P^-}$ does not concern us in this simple example.
We can choose $\pre{g}(\tau)$ to have principal part $(g,g)_{\Gamma_0(23)}q^{-1}$,
where $(g, g)_{\Gamma_0(23)} = \int_{\Gamma_0(23)\bs \uhp}g(\tau)\bar{g}(\tau)\frac{dudv}{v}$ is the
Petersson norm of $g$.
This can be identified as the theta lift of $3g/4$ using the technique of \cite{ehlen-binary}. 
But it is also the example that Stark gives on page 91 of \cite{StarkLsII} and it turns out that
$(g,g)_{\Gamma_{0}(23)} = 3\log \abs{\alpha}$, where $\alpha \in H$ is the unit
which is the unique real root of the polynomial $X^{3}-X-1$ with complex embedding equal to $1.324717 \ldots$.

To determine the coefficients $c_{\pre{g}}^+(m)$ of $\pre{g}$ with positive index we first show that for every $m>0$ 
with $\kronecker{-l}{m} \neq 1$ (so that $\kronecker{-l}{-m}\neq -1$), 
there is a weakly holomorphic modular form with all integral Fourier coefficients
\[
f_m(\tau) = q^{-m} + O(q^2) \in M_1^{!,+}(\Gamma_0(23),\chi_{-23}).
\]
We can show this by constructing such a form for $1 \leq m \leq 23$ with $\kronecker{-l}{m} \neq -1$ 
and then use the same strategy as in the proof of Proposition \ref{prop:independence} to obtain all such forms.
Using \sage{}, we see that the space $M_{13}^{+}(\Gamma_0(23),\chi_{-23})$ has dimension $13$ and computed an integral basis.
Dividing each form in the integral basis by $\Delta(23\tau)$ gives us an element of $M_1^{!,+}(\Gamma_0(23),\chi_{-23})$.
This way, we obtain the $12$ weakly holomorphic modular forms $f_m$ for $m \in \{5, 7, 10, 11, 14, 15, 17, 19, 20, 21, 22, 23\}$.
As an example, the Fourier expansions of the first three forms start with
\begin{align*}
  f_5(\tau) &= q^{-5} - 6q^{2} + q^{3} - 7q^{4} - 8q^{6} + 19q^{8} + 20q^{9} + O(q^{11}),\\
  f_{7}(\tau) &= q^{-7} - 4q^{2} - 10q^{3} - 5q^{4} + 8q^{6} - 31q^{8} + 35q^{9} + O(q^{11}),\\
  f_{10}(\tau) &= q^{-10} - 13q^{2} - 14q^{3} + 13q^{4} + 13q^{8} - 78q^{9} + O(q^{11}).
\end{align*}
Now we obtain $f_m$ for $m>23$ as follows.
Write $m = 23a + b$ with $0 < b \leq 23$. We construct $f_m$ for all $b$ by induction on $a$.
Then $j(23\tau)^af_b$ lies in the plus-space and its Fourier expansion has only integral Fourier coefficients and
starts with $q^{-m}+O(q^{-23a+2})$.
Therefore, we can subtract off integral multiples of $f_n$ with $n<23a$ to obtain the desired principal part.
We will write $c_m(n)$ for the Fourier coefficient of index $n$ of $f_m$.

According to Theorem \ref{thm:value-phiz}, we have
\[
    c_{\pre{g}}^+(m) = 3(\Phi_\calO(1, f_m) - \Phi_\calO(h_{\bar{\fraka}},f_m)) = 3(\Phi_\calO(1,f_m) - \Phi_\fraka(1,f_m)),
\]
where $h_\fraka$ is the idele corresponding to $\fraka$.
To determine the theta lifts on the right hand side, we use the see-saw identity \eqref{eq:seesaw}.
We will need to know the principal part of $F_m = f_m \otimes \GJ$ for every $m$,
where we identify $f_m$ with a vector valued modular form in $M_{1,\calO}^!$.
We will simply write $C_m(n)=C_m(n,\mu)$ with $\mu=x+L_{23}$ and $x^2 \equiv n \bmod (46)$ for the coefficients of $F_m$.
This is independent of the choice of $x$ as every square modulo $92$ has at most two roots modulo $46$. 
The nonzero coefficients in the principal part are then given in Table \ref{tab:Fm}. Note that
the constant term of $F_m$ vanishes for every $m$.
\begin{table}[h]
  \centering
\begin{tabular}[h]{l ccccc}
\toprule
  $n$         & $\frac{-4m-23}{92}$ & $\frac{-4m}{92}$ & $\frac{-15}{92}$ & $\frac{-11}{92}$ & $\frac{-7}{92}$ \\
\midrule
  $C_m(n)$ & $1$                            & $10$                   & $c_m(2)$              & $c_m(3)$             & $c_m(4)$  \\
\bottomrule
\end{tabular}
  \caption{The principal part of $F_m$}
  \label{tab:Fm}
\end{table}

In order to determine the coefficients
of $f_m$ of index $2,3$ and $4$ for any $m$,
we will use the forms $f_2, f_3$ and $f_4$ which are contained in  $M_1^{!,-}(\Gamma_0(23),\chi_{-23})$ 
and are constructed in a similar fashion as $f_m$ above.
We have
\begin{align*}
  f_2 &= q^{-2} + q^{-1} - 1 + 6q^{5} + 4q^{7} + 13q^{10} + O(q^{11}),\\
  f_3 &= q^{-3} + q^{-1} - 1 - q^{5} + 10q^{7} + 14q^{10} + O(q^{11}),\\
  f_4 &= q^{-4} - 1 + 7q^{5} + 5q^{7} - 13q^{10} + O(q^{11}).
\end{align*}
We can easily determine the Fourier expansion of these $3$ forms to an arbitrary precision.
We note that these are also the first few in a sequence of weakly holomorphic modular forms
spanning the space $M_1^{!,-}(\Gamma_0(23),\chi_{-23})$, having all integral Fourier coefficients.
\begin{lemma}
  We have $c_m(2) = -c_2(m)$, $c_m(3)=c_3(m)$ and $c_m(4)=-c_4(m)$.
\end{lemma}
Using the Lemma and Theorem \ref{thm:borcherds}, we see that $\Phi_{L_{23}}(F_m,(z,1))$
for $L_{23}$ as in Section \ref{sec:embedding} for $N=23$ is equal to
\[
  \Phi_{L_{23}}((z,1),F_m) = -2\log|\Psi_{L_{23}}((z,1), F_m)|^2
\]
and $\Psi_{L_{23}}((z,1), F_m)^2$ is a meromorphic modular form for $\Gamma_0(23)$, invariant under the Fricke involution
and with divisor
\begin{align*}
    &Z\left(\frac{4m+23}{92}\right)
    + 10Z\left(\frac{4m}{92}\right)\\
    &+c_2(m) Z\left(\frac{15}{92}\right)
  +c_3(m) Z\left(\frac{11}{92}\right)
  +c_4(m) Z\left(\frac{7}{92}\right).
\end{align*}
Here, we briefly wrote $Z(n/92)$ for the divisor $Z(n/92,\mu) + Z(n/92,-\mu)$ with $\mu=x+L_{23}$ and $x^2 \equiv -n \bmod (46)$.
The modular curve $X=X_{0}^+(23)$ has genus $0$.
A generator $H_{23}$ of the function field of $X$ can be obtained as the Borcherds product $H_{23} = \Psi_{L_{23}}((z,1),F)$,
where $F(\tau) \in M^!_{1/2,L_{23}}$ is the unique form with principal part equal to $q^{-7/92}(\phi_{19} + \phi_{-19})$ 
and constant term $0$. It turns out that $H_{23}(\tau) = \frac{\theta_{\calO}(\tau)}{g(\tau)} - 2$ and we have
\begin{equation*}
   H_{23}(\tau) = q^{-1}+ 4q + 7q^{2} + 13q^{3} + 19q^{4} + 33q^{5} + 47q^{6} + 74q^{7} + 106q^{8} + 154q^{9} + O(q^{10}).
\end{equation*}
Using this information it is clear that
\[
  \Psi_{L_{23}}((z,1), F_m )^2 = R_m(H_{23}(z)),
\]
where $R$ is the rational function given by
\[
  R_m(x) = \prod_{z \in \divisor(\Psi(F_m))^2} (x-H_{23}(z))^{2\ord_z(\Psi(F_m))} = \prod_{d > 0} P_d(x)^{C_m(-d)},
\]
where
\[
  P_d(x) = \prod_{z \in Z(d/92)} (x-H_{23}(z)) \in \Z[x].
\]
We see that the coefficients of the holomorphic part of $\pre{g}$ can all be obtained
by rational functions in $H_{23}(z_{23})$, where $z_{23} = \frac{-23 + \sqrt{-23}}{64}$.
Explicitly, we obtain
\begin{align*}
c_\calO^+(m) &= \frac{-1}{12}\log\abs{R_m(H_{23}(z_{23}))}, \quad c_\fraka^+(m) = c_{\bar{\fraka}}^+(m) = \frac{-1}{12}\log\abs{R_m(H_{23}(z_{23})^{\sigma(\bar\fraka)})}, \\
c_{\pre{g}}^+(m) &= \frac{-1}{24}\log\abs{\frac{R_m(H_{23}(z_{23}))}{R_m(H_{23}(z_{23}))^{\sigma(\bar\fraka)}}}.
\end{align*}
We remark that the algebraic numbers defining $c_\fraka^+(m)$ and $c_{\bar{\fraka}}^+(m)$ in Corollary \ref{cor:scalar-pre}
are different. However, $\abs{R_m(H_{23}(z_{23})^{\sigma(\bar\fraka)})} = \abs{R_m(H_{23}(z_{23})^{\sigma(\fraka)})}$
because $H_{23}(z_{23})^{\sigma(\bar\fraka)} = \overline{H_{23}(z_{23})^{\sigma(\fraka)}}$ and the coefficients of $R_m$
are rational integers, which implies the equality $c_\fraka^+(m) = c_{\bar{\fraka}}^+(m)$.

A table with the polynomials $P_d$ for some small values of $d$ can be found
in \cite{MarynaInnerProd}. Note that we need to substitute $x-1$ for $x$ in the table
because our normalization of the Hauptmodul $H_{23}$ is different to match the Borcherds product.
Using these explicit formulas it is also easy to verify the factorization formula we gave 
and the numerical values obtained in \cite{DukeLiMock}.

We conclude this section with an interesting observation: the values $R_m(H_{23}(z_{23}))$ can also be interpreted as
follows. Consider the modular function $\tilde{H}_{23}(\tau) = H_{23}(\tau) - H_{23}(z_{23})$
and let $\tilde{F}_m$ be the unique weakly holomorphic modular form in $M_{1/2,L_{23}}^!$ with a principal part
equal to $q^{-m/92}(\phi_{\mu} + \phi_{-\mu})$ and $\mu = x + L_{23}$ with $x^2\equiv -m \bmod (92)$, as above.
Then we obtain
\[
\Phi((z_{23},1),\tilde{F}_m) = -2\log\abs{\prod_{z \in Z(m)} (H_{23}(z_{23})-H_{23}(z))} = -2 \log\abs{\tilde{H}_{23}(Z(m))}
\]
and by Theorem \ref{thm:value-phiz}, this implies that
\begin{align*}
\Theta(\tau) \otimes \pre{\Theta}_\calO^+(\tau) &= \frac{3}{2}\log\abs{\alpha}q^{-1/23}(\phi_{2/92} + \phi_{-2/92}) +c_{\calO}^+(0)\\
&\phantom{=}
-2 \sum_{\mu \bmod (46)}\!\!\!\!\!\sum_{\substack{m>0\\ -m\equiv \mu^2 \bmod (92)}}\!\!\!\!\!\log\abs{\tilde{H}_{23}(Z(m))}q^{\frac{m}{92}}(\phi_{\mu} + \phi_{-\mu}),
\end{align*}
where 
\[
  \Theta(\tau) = \sum_{n \in \Z}e\left(\frac{n^2}{4}\tau\right)\phi_{n/2 + \Z}
\]
is the vector valued theta function for the lattice $\Z$ with quadratic form $x^2$.
Moreover, note that $\Theta(\tau) \otimes \pre{\Theta}_\calO(\tau)$ is modular of weight $3/2$
and transforms with representation $\rho_{L_{23}^-}$. We can obtain the coefficients $c_\calO^+(m)$
as linear combinations of coefficients of $\Theta(\tau) \otimes \pre{\Theta}_\calO^+(\tau)$.

This suggests that the harmonic weak Maa\ss{} forms of weight one that we described in this paper
appear as building blocks for generating series as above -- note the similarity to the result
on the modularity of the generating series of \emph{traces of singular moduli} \cite{zagiertraces, ,brfu06}!
It should be possible to obtain these generating series explicitly using the Kudla-Millson theta lift
as in \cite{brfu06, funke-cm32, alfesehlentraces}. This would also yield an explicit construction of the
harmonic Maa\ss{} forms we considered using a theta lift. However, note that the fact that the
obstruction space $S_{3/2,L_{23}^-}$ is zero plays a cruical role in our simple description.
The situation is more complicated in general.
We will come back to this description in a sequel to this paper.

\section{An arithmetic theta function of weight one}
\label{sec:br-mod}
In the spirit of the Kudla program, we are led to form another generating series related to the cycles $\calZ(m,\fraka,\beta)$.
The non-holomorphic Eisenstein series of weight one should be seen as the generating series of the arithmetic cycles
equipped with Kudla's Green functions. 
We will complete them instead with the automorphic Green functions
obtained by the regularized theta lift and form the generating series of the completed cycles $\hat\calZ(m,\fraka,\beta)$
with values in $\widehat\CH^1(C_D) \otimes_\Z \Z[P'/P]$:
\begin{equation}
  \label{eq:phihat}
  \hat\Phi(\tau) = \sum_{\beta \in P'/P} \sum_{m > 0} \hat\calZ(m,\fraka,\beta) e(m\tau) \phi_{\beta}.
\end{equation}

Let us first define the completed cycles.
To keep the setup as simple as possible, we define them as Arakelov divisors \cite[Kapitel III]{neukirchalgzt}
on $\Spec \calO_{\Hkb}$.
We let
\[
  \hat\calZ(m,\fraka,\beta) = \sum_{\frakP \subset \calO_{\Hkb}} \calZ(m,\fraka,\beta)_{\frakP} \frakP
                           + \sum_{\sigma} \lambda(m,\beta,\sigma) \sigma.
\]
Here, $\sigma$ runs over all complex embeddings of $H$  and we define $\lambda(m,\beta,\sigma) \in \R$ as follows.
For $m \in \Z_{>0}$ and $\beta \in P'/P$, we let $g_{m,\beta} \in H_{1,P}$ with principal part
\[
  P_{g}(\tau) = \frac{1}{2}q^{-m}(\phi_{\beta} + \phi_{-\beta})
\]
and $c_{g_{m,\beta}}^{+}(0,0) = 0$. 
We start from a fixed embedding $\sigma_0$ of $\calO_{\Hkb}$ into $\C$ and define
\[
  \lambda(m,\beta,\sigma_0) := \frac{1}{w_{k}} \Phi_{P}(g_{m,\beta}(\tau), 1)
\]
Then, for any other embedding $\sigma$, if $\sigma = \sigma_0 \circ \sigma(h)$ or $\bar\sigma = \sigma_0 \circ \sigma(h)$, 
we let $h(\sigma) := h(\bar\sigma) := h$ and define
\[
  \lambda(m,\beta,\sigma) := \frac{1}{w_{k}} \Phi_{P}(g_{m,\beta}(\tau), h(\sigma)).
\]
However, all of this only makes sense if $\Phi_{P}(g_{m,\beta}(\tau), h) \in \R$, 
which we did not prove here. It should be possible to show this using e.g.
the embedding used in Section \ref{sec:embedding} and by working out the Fourier expansion
as in Section 7 of \cite{boautgra}. To be safe, we let 
\[
  \tilde{\Phi}_{P}(g_{m,\beta}(\tau), h) := \tfrac{1}{2}(\Phi_{P}(g_{m,\beta}(\tau), h) + \overline{\Phi_{P}(g_{m,\beta}(\tau), h)}) \in \R
\]
and let $\lambda(m,\beta,\sigma) := \frac{1}{w_\kb}\tilde{\Phi}_{P}(g_{m,\beta}(\tau), h(\sigma))$.
Then we define the formal power series $\hat\Phi_P(\tau)$ with coefficients in $\widehat\CH^1(\calO_{\Hkb}) \otimes_\Z \Z[P'/P]$ via \eqref{eq:phihat}.
To make this precise, we have to impose further conditions on $g_{m,\beta}$:
we demand that $c^+_{g_{m,\beta}}(p_j,\pi_j) = 0$ for all $j$ and where $p_j$ and $\pi_j$ have
been defined in Section \ref{sec:special-basis}.

Note that it is a consequence of Theorems \ref{thm:value-phiz} 
and \ref{thm:pre-pvals} that 
$\deghat\, \hat\calZ(m,\fraka,\mu)$ vanishes. 
(However, this also follows directly from Theorem 6.5 of \cite{bryfaltings}.)
Consequently, the coefficients 
of $\hat\Phi(\tau)$ vanish when viewed as a generating series valued in $\widehat\CH^1_\R(\calO_{\Hkb})$
and the degree generating series is thus ``trivially'' modular.

However, a stronger statement is true. In fact, what we have shown amounts to: 
\begin{corollary}
  If $f \in M_{1,P}^!$ is a weakly holomorphic modular form
  with only integral Fourier coefficients in its principal part, then
  \[
    \sum_{\beta \in P'/P} \sum_{n>0}c_f(-n,\beta)\widehat\calZ(n,\fraka,\beta) = 0 \in \widehat\CH^1(\calO_{\Hkb})\otimes_\Z \Q.
  \]
\end{corollary}
\begin{proof}
  First, we claim that
  \begin{equation}
    \label{eq:gmfm}
    \sum_{\beta \in P'/P} \sum_{n>0}c_f(-n,\beta)g_{m,\beta}(\tau) = \frac{1}{2}\sum_{\beta \in P'/P} \sum_{n>0}c_f(-n,\beta)f_{m,\beta}(\tau) := \tilde{f},
  \end{equation}
where $f_{m,\beta}$ has beed defined in Proposition \ref{prop:fbm}.
 In fact, the difference of the left-hand side and the right-hand side is a weakly holomorphic modular form with principal part equal to
\[
  \frac{1}{2}\sum_{\beta \in P'/P} \sum_{n>0}c_f(-n,\beta) \sum_{j=1}^{d^-} a_j(n,\beta) q^{-n_j} (\phi_{\beta_j}+\phi_{-\beta_j}).
\]
Since $f$ is weakly holomorphic, the double sum
\[
  \sum_{\beta \in P'/P} \sum_{n>0}c_f(-n,\beta) a_j(n,\beta) = \{f, g_j\}
\]
vanishes for all $j$.
Thus, the difference of the left-hand side and right-hand side in \eqref{eq:gmfm} is a holomorphic modular form with vanishing constant term
in $M_{1,P}$ and thus a cusp form since we assume that $P'/P$ has squarefree order $\abs{D}$.
However, since the normalization of $g_{m,\beta}$ for the positive index Fourier coefficients
agrees with the normalization of $f_{m,\beta}$, \eqref{eq:gmfm} follows.
Write
\[
  \sum_{\beta \in P'/P} \sum_{n>0}c_f(-n,\beta)\widehat\calZ(n,\fraka,\beta) = \sum_{\beta \in P'/P} \sum_{n>0}c_f(-n,\beta)\calZ(n,\fraka,\beta) + \frac{1}{w_\kb}\sum_\sigma\tilde{\Phi}_P(\tilde{f},h(\sigma))\sigma.
\]
Then Theorem \ref{thm:value-phiz} shows that
\[
  \Phi_P(\tilde{f},h(\sigma)) = \frac{1}{2}\sum_{\beta \in P'/P} \sum_{n>0}c_{f}(-n,\beta) c_P^+(h(\sigma),n,\beta)
\]
and by Theorem \ref{thm:pre-pvals}, we have that $c_P^+(1,n,\beta) = -\tfrac{2}{r} \log\abs{\alpha(1, n, \beta)} \in \R$ with 
\[
  \ord_\frakP(\alpha(1, n, \beta)) = r \cdot w_\kb \cdot \calZ(n,\fraka, \beta)_\frakP
\]
and $\alpha(h(\sigma),n,\beta) = \alpha(1,n,\beta)^{\sigma}$, which shows that the class defined by
\[
  \sum_{\beta \in P'/P} \sum_{n>0}c_f(-n,\beta)\widehat\calZ(n,\fraka,\beta)
\]
is torsion and thus vanishes in $\widehat\CH^1(\calO_{\Hkb})\otimes_\Z \Q$.
\end{proof}

\begin{lemma}
  The classes $\widehat\calZ(m,\fraka,\beta)$ lie in a finite dimensional subspace of $\widehat\CH^1(\calO_{\Hkb})\otimes_\Z \Q$.
\end{lemma}
\begin{proof}
  Clearly, taking $f = f_{m,\beta}$ in the corollary gives
  \[
    \frac{1}{2}\sum_{j=1}^{d^-} a_j(m,\beta)\widehat\calZ(n_j,\fraka,\beta_j) = \widehat\calZ(m,\fraka,\beta) \in \widehat\CH^1(\calO_{\Hkb})\otimes_\Z \Q,
  \]
  i.e. the special divisors lie in the subspace spanned by the $\widehat\calZ(n_j,\fraka,\beta_j)$, $j=1, \ldots, d^-$.
\end{proof}

\begin{corollary}
  The formal power series $\widehat\Phi(\tau)$ is the $q$-expansion of a modular form of weight one and representation $\rho_{P^-}$.
\end{corollary}
\begin{proof}
  This follows by applying the modularity criterion as in \cite[Lemma 4.3]{bogkz} and the existence of
  a basis of $M_{1,P^-}$ with integral Fourier coefficients \cite{McGrawBasis}.
\end{proof}

\begin{remark}
  It is a bit unsatisfactory that, in contrast to the non-holomorphic
  arithmetic generating series $\calE_P(\tau)$ we do not know that $\widehat\Phi(\tau) \neq 0$ and this seems to be a subtle question.
\end{remark}

\printnomenclature[2.8cm]
\printbibliography

\end{document}